\documentclass[11pt,twoside]{article}

\usepackage{fullpage}
\usepackage{booktabs}
\usepackage{algorithm}

\input{command-template}

\newcommand\blfootnote[1]{%
  \begingroup
  \renewcommand\thefootnote{}\footnote{#1}%
  \addtocounter{footnote}{-1}%
  \endgroup
}

\newcommand{\E}{\mathbb{E}}
\newcommand{\cB}{\mathcal{B}}
\newcommand{\cA}{\mathcal{A}}
\newcommand{\SO}{\mathrm{SO}}
\newcommand{\Cost}{\mathrm{Cost}}
\newcommand{\inner}[2]{\langle #1,#2\rangle}
\newcommand{\sge}[1]{{\color{teal}#1}}
\newcommand{\PG}{\mathsf{G}}
\def\feaReg{X}
\def\nn{\nonumber}
\def\beq{\begin{equation}}
\def\eeq{\end{equation}}
\def\bbe{\mathbb{E}}
\def\tsum{\textstyle\sum}

\begin{document}

\begin{center}
		
		{\bf{\LARGE{When is a better LLM worth paying for? \\
        Stochastic Optimization with Switching Mixed-Fidelity Oracles
        }}}
		
		\vspace*{.2in}
		
		{\large{
				\begin{tabular}{ccc}
					Swati Gupta$^\star$ & Tianjiao Li$^{\dagger,\diamond}$ 
				\end{tabular}
		}}
		\vspace*{.2in}
		
		\begin{tabular}{c}
			MIT Sloan School of Management$^\star$\\
			IBM Research$^\dagger$
		\end{tabular}
		
		\vspace*{.2in}

		\today
		
		\vspace*{.2in}
		
		\begin{abstract}
        \blfootnote{$^\diamond$Tianjiao Li was at MIT when most of this work was performed.}
        Large language models (LLMs) are increasingly used upstream of decision-making to clean data, detect label errors, and reconstruct missing values. Yet LLM-assisted data curation creates a new operational tradeoff: higher-fidelity models and richer prompts can be substantially more expensive, while the value of this additional fidelity depends on the downstream optimization accuracy.  This raises the central question of this work: \emph{when should one pay for a better LLM, {given a specific downstream task at hand?}} 

        We study this question through stochastic optimization with mixed-fidelity first-order oracles, {which can have systematic bias and state-dependent variance.} We develop two cost-aware algorithmic frameworks that allow adaptive switching between stochastic oracles. With known bias, the first meta-algorithm applies a restarted, cost-aware policy that adaptively selects oracle fidelity across optimization stages. Its cumulative cost is no larger than that of the best fixed oracle in hindsight and can be substantially smaller. With unknown bias, we develop a budgeted certification scheme that uses occasional trusted-gradient evaluations to certify the optimization progress and eliminate unsuccessful oracle sources, and show an explicit tradeoff between certification cost and delayed detection of biased sources. 
        We provide concrete realizations of these meta-algorithms by applying biased stochastic gradient extrapolation (SGE) and recursive regularization under state-dependent noise as subroutines. 
        Synthetic and LLM-completion experiments {illustrate the qualitative switching behavior predicted by the theory}: the adaptive policy skips several expensive LLM oracles and substantially reduces cumulative cost relative to fixed high-fidelity alternatives while achieving comparable accuracy. Our studies demonstrate that, when information quality can be purchased at different prices, the right fidelity need not be fixed; it depends on how valuable additional accuracy is at the current stage of the downstream decision problem.
		\end{abstract}
	\end{center}

\section{Introduction}
Large language models (LLMs) are evolving rapidly, with a growing number of models available through commercial APIs and open-source platforms.
These models are increasingly used not only to make predictions, but also to \emph{prepare the data on which downstream decisions are based}. Recent work uses LLMs for data cleaning and matching \citep{zhang2024jellyfish,parciak2024schema,peeters2025entity}, missing-value imputation \citep{hayat2024claim,he2025llmforest,wang2025llm4hrs}, denoising recommendation data \citep{wang2025llard,wang2025llm4dsr}, and data annotation \citep{gilardi2023chatgpt,goel2023llms,ye2025calibrating}. 
In these applications, choosing a model or prompt affects not only the cost of obtaining the data, but also its statistical quality for the downstream task. Stronger models, richer prompts, or ensembles may produce more accurate data, but often at a substantially higher price. The key question we consider in this work is: \emph{when is additional fidelity worth paying for?} We explore this in the context of downstream optimization tasks, like supervised machine learning. 



Consider a synthetic experiment on the \texttt{Aspiring Minds} hiring dataset with 10\% missing covariates (randomly masked). The goal is to predict \texttt{ComputerProgramming} scores of the candidates from numeric independent features like performance in high school, tier of the undergraduate college, agreeableness, etc. We ask different LLM model--prompt pipelines to reconstruct the missing values, {which induces} a stochastic-gradient oracle for downstream linear regression, with different fidelities. Table~\ref{tab:llm-motivation} reports their cost and representative fidelity parameters. Even within one model family, richer prompting can significantly change downstream fidelity at almost the same cost. Across different models, price and fidelity are not necessarily monotone. For example, Claude Sonnet 5 costs substantially more than Claude Haiku 4.5 while attaining essentially the same downstream objective in this experiment. Conversely, averaging several high-fidelity completions reduces bias but is much more expensive than a simple completion with GPT Nano. Thus neither ``always use the cheapest model'' nor ``always use the best model'' is compelling.

\begin{table}[t]
\centering
\caption{Representative LLM-completion pipelines at 10\% missingness. Cost is the one-time full-dataset completion price; Obj. is the clean downstream regression objective. The empirical diagnostics $\sigma_{\lambda,*}^2$, $\cL_\lambda$, and $\widehat{\cB}_\lambda$ are defined in Section~\ref{sec:numerics}. The ``clean data'' row is a benchmark and is unavailable to the adaptive policy. The \textbf{bolded} rows represent the stochastic oracles selected along the trajectory of our adaptive algorithm. }
\label{tab:llm-motivation}
\footnotesize
\setlength{\tabcolsep}{3.0pt}
\vspace{0.2cm}
\begin{tabular}{clrrrrr}
\toprule
$\lambda$ & Completion pipeline & Cost (\$) & Obj. & $\sigma_{\lambda,*}^2/10^4$ & $\cL_\lambda/10^3$ & $\widehat{\cB}_\lambda$ \\
\midrule
-- & Clean data (benchmark) & 0     & 3354 & 8.28 & 6.48 & 0 \\
1 & GPT-5.4 Nano, simple-7   & 0.182 & 3823 & 9.41 & 6.48 & 5.85 \\
\textbf{2} & \textbf{GPT-5.4 Nano, LR-short-7} & \textbf{0.168} & \textbf{3722} & \textbf{9.12} & \textbf{6.48} & \textbf{3.95} \\
3 & GPT-5.4 Nano, LR-short-10& 0.189 & 3691 & 9.11 & 6.48 & 3.26 \\
\textbf{4} & \textbf{GPT-5.4 Nano, LR-10}      & \textbf{0.188} & \textbf{3665} & \textbf{9.16} & \textbf{6.48} & \textbf{2.87} \\
\textbf{5} & \textbf{GPT-5.6 Luna, LR-10}      & \textbf{0.822} & \textbf{3617} & \textbf{8.69} & \textbf{6.48} & \textbf{2.31} \\
6 & Claude Haiku 4.5, LR-10  & 1.417 & 3596 & 8.71 & 6.48 & 2.01 \\
7 & Claude Sonnet 5, LR-10   & 3.158 & 3596 & 8.74 & 6.48 & 2.37 \\
\textbf{8} & \textbf{Avg. Haiku+Sonnet, LR-10} & \textbf{4.575} & \textbf{3581} & \textbf{8.68} & \textbf{6.48} & \textbf{1.86} \\
9 & Avg. Luna+Haiku+Sonnet, LR-10 & 5.397 & 3585 & 8.65 & 6.48 & 1.98 \\
\bottomrule
\end{tabular}
\end{table}

\begin{figure}[t]
\centering
\includegraphics[width=0.7\linewidth]{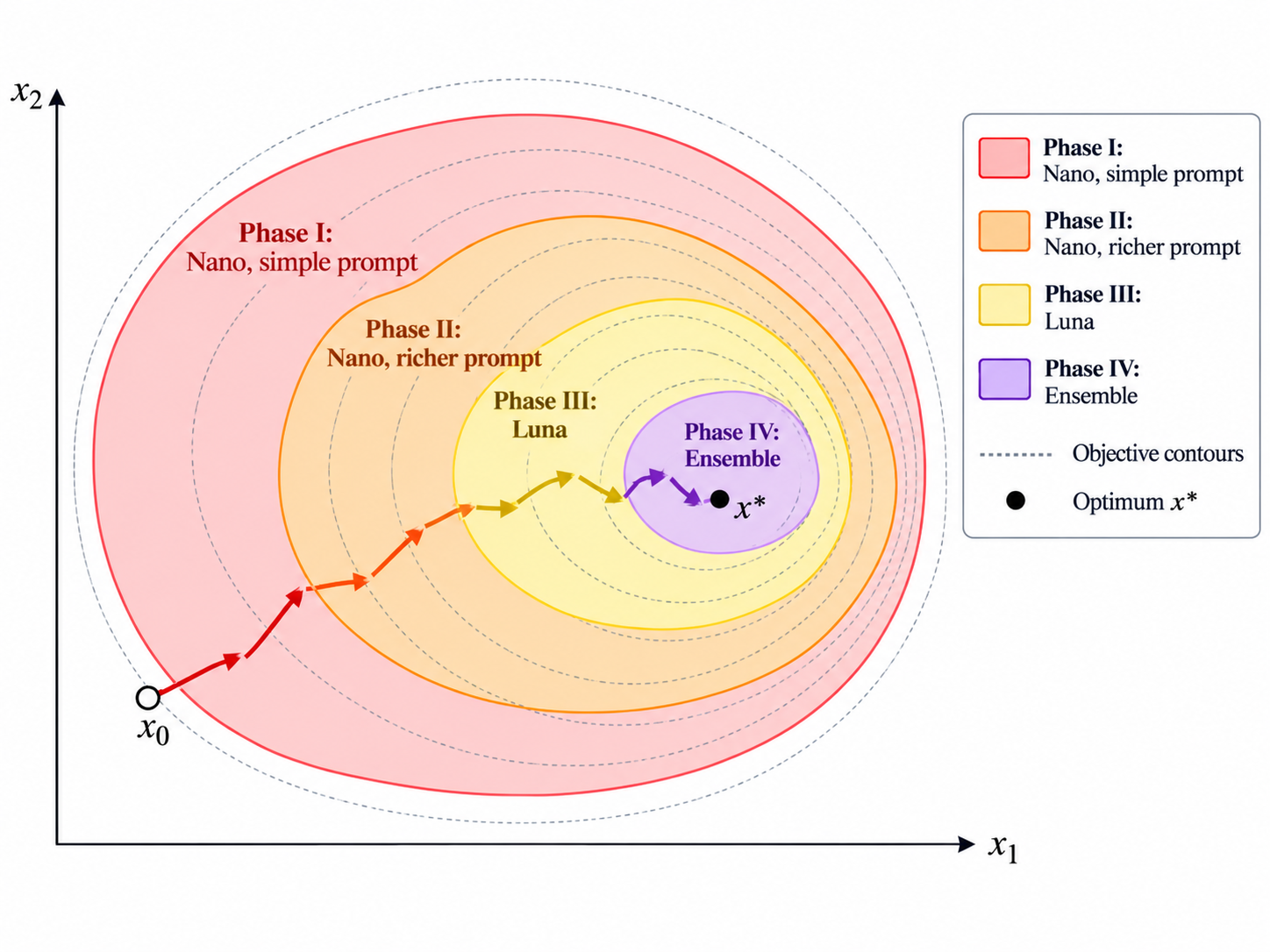}
\caption{Schematic illustration of stage-dependent oracle switching. At coarse accuracy scales, a low-cost oracle can be sufficient for continued progress; as optimization approaches $x^*$, higher-fidelity oracles become worthwhile or necessary. The actual switching points depend jointly on oracle cost, bias, variance, and the current optimization scale.}
\label{fig:illustration_plot}
\end{figure}

This question motivates us to propose a progress-dependent optimization strategy. Rather than committing to one oracle throughout optimization, we divide the optimization trajectory into stages of progressively finer accuracy. At a stage with current error scale $R$, an oracle does not need to be the most accurate one available; it only needs to be accurate enough to support the next constant-factor reduction in optimization error. Among the oracles that can achieve this reduction, we choose the most cost-efficient one, i.e., the one with the smallest total stage cost. Therefore, when far from the optimum, a cheap but relatively noisy or biased oracle can be the most economical choice. As the target radius shrinks, the required bias level becomes more stringent and residual noise becomes more important, causing the algorithm to switch to higher-fidelity sources. This design is informed by recent studies on state-dependent behavior of stochastic optimization algorithms \citep{woodworth2021even,juditsky2023sparse,ilandarideva2025state}. In the classical state-independent setting \citep{nemirovski2009robust, lan2012optimal}, the stochastic gradient noise is assumed to be uniformly bounded over the domain, which can be restrictive for important applications such as generalized linear regression with unbounded regressors. Under the more general state-dependent model, the noise level varies with the optimality state: the state-dependent component decreases as the iterates approach the optimum, while the residual component captures the irreducible noise at the optimum. With multiple priced oracles, these different scalings imply that the relative cost-efficiency of the oracles can change across optimization stages; thus there need not be a single oracle that is uniformly optimal throughout the trajectory. In our illustrative experiment with oracles shown in Table~\ref{tab:llm-motivation}, this logic produces the progression from inexpensive Nano pipelines to Luna and eventually an ensemble, while skipping several more expensive alternatives like Sonnet altogether. This algorithmic idea is further illustrated in Figure~\ref{fig:illustration_plot}.\footnote{Additionally, we use Claude Opus 5.5 to create an interactive playground \url{https://claude.ai/artifact/AsfEe1aEpfKCkPv9YXsDEg} for the reader to explore our algorithmic framework with different stochastic oracles.}

\paragraph{Framework for Mixed-Fidelity Stochastic Oracles} We consider a general model for stochastic optimization, where mixed-fidelity oracles can be used to provide stochastic gradients, with known costs and state-dependent bias and variance. The goal is to minimize 
\begin{equation}
  f^*:=\min_{x\in X} f(x),
  \label{eq:problem}
\end{equation}
where $X\subseteq\R^d$ is closed and convex, and $f$ is $L$-smooth and $\mu$-strongly convex\footnote{Here, a function $f(\cdot)$ is $L$-smooth and $\mu$-strongly convex if $\frac{\mu}{2}\|x-y\|^2 \leq f(x) - f(y) -\langle \nabla f(y), x-y\rangle \leq \frac{L}{2}\|x-y\|^2$
where $\|\cdot\|$ denote the Euclidean norm ($\ell_2$-norm) in $\mathbb{R}^d$.}. For example, one can consider the least squares loss for any regression problem. We assume that we have access to $\ell$ stochastic first-order oracles $\{SO_{\lambda}\}_{\lambda=1}^\ell$, where a query to $SO_{\lambda}$ costs $c_{\lambda}$ and this is known to the user. The fidelity of the oracle is governed by three parameters: (i) the systemic bias of the oracle $\cB_{\lambda}$; (ii) the state-dependent modulus $\cL_{\lambda}$; (iii) the residual variance near the optimum $\sigma_{\lambda, *}^2$. 
A query to $\SO_\lambda$ at point $x$ returns a stochastic gradient  $\cG_\lambda(x,\xi)$ with mean $g_\lambda(x)=\bbe_\xi[\cG_\lambda(x,\xi)]$, and 
\begin{align}
 \|g_\lambda(x)-\nabla f(x)\| &\le \cB_\lambda, \label{eq:bias}\\[-1mm]
 \E\|\cG_\lambda(x,\xi)-g_\lambda(x)\|^2
 &\le \cL_\lambda [f(x) - f^*]+\sigma_{\lambda,*}^2. \label{eq:variance}
\end{align}
The first condition captures systematic oracle bias. In our most general setting, we do not assume knowledge of the bias parameter.   The second condition allows the stochastic variance to depend on the current optimization error: the component $\cL_\lambda [f(x) - f^*]$ decreases as optimization progresses, while the residual variance $\sigma_{\lambda,*}^2$ is non-diminishing even near the optimum. This setting coincides with our framework where the noise introduced by stochastic oracles depend on the optimization progress. Together, $\cB_\lambda$, $\cL_\lambda$, $\sigma_{\lambda,*}^2$, and the price $c_\lambda$ determine the value of oracle $\lambda$ at a given accuracy scale. In an LLM pipeline, an oracle may correspond to a model, prompt, ensemble, or human--AI combination, while $f$ represents the downstream objective rather than the upstream completion or annotation task itself. Section~\ref{sec:applications} gives concrete examples and verifies these assumptions for LLM-generated labels and completed covariates.

\subsection{Contributions.}
Our main contributions lie in the design and analysis of cost-efficient adaptive oracle-switching meta-algorithms,  together with concrete stochastic optimization subroutines and LLM experiments.
\begin{itemize}
    \item \textbf{Multistage oracle-switching with known bias.} We develop a generic meta-algorithm that uses a chosen stochastic optimization subroutine across stages, so that the current squared-distance radius is reduced by a constant factor. At each stage, the policy selects the most cost-efficient oracle for achieving the next contraction among all oracles with admissible bias. As one concrete contraction subroutine, we analyze a strongly convex variant of stochastic gradient extrapolation (SGE) \citep{lan2018rge} under state-dependent noise and persistent bias. 
    For unbiased oracles, the cumulative cost of our framework is no larger than that of the best fixed oracle in hindsight, and the savings can be up to $\Theta(K/\log K)$ over $K$ stages.  With biased oracles, the framework can exploit cheaper sources early while ultimately reaching the accuracy scale allowed by lower-bias sources. 
    
    \item {\bf Characterizing cost with respect to desired accuracy:} The stagewise formulation yields explicit operational criteria for comparing oracle cost and fidelity. We characterize when one oracle is dominated by another and derive stage-dependent break-even prices that quantify how much one should be willing to pay for a higher-fidelity source at a given downstream accuracy level. Our results can help vendors check a simple condition dependent on their downstream task, and decide which oracles are useful in their pipeline. 

    \item \textbf{Budgeted certification with unknown bias.} When bias parameters are unavailable or hard-to-estimate, we develop a second meta-algorithm based on occasional trusted projected-gradient certification. We introduce a budget parameter to control how much uncertified training can be performed before the next checkpoint; failed sources are removed permanently while any certifiable progress is retained. We further show that recursive regularization combined with biased SGE provides a concrete projected-gradient reduction subroutine for certification, and characterize the resulting tradeoff between frequent certification and delayed detection of biased sources.

    \item \textbf{Empirical validation.} We conducted controlled simulations and LLM-completion experiments on the \texttt{Aspiring Minds} hiring dataset. The adaptive strategy exhibits the predicted stage-dependent switching and achieves substantial cost savings relative to strategies with a fixed stochastic oracle. At 10\% missingness, for example, the adaptive policy follows \[
\begin{aligned}
\text{Nano LR-short-7 (SO2)} \longrightarrow \text{Nano LR-10 (SO4)}\longrightarrow \text{Luna (SO5)} \longrightarrow \text{Avg. Haiku+Sonnet (SO8)},
\end{aligned}
\]
skipping several more expensive alternatives.
The cumulative cost is about \$72, versus about \$240 for always using the best oracle, a reduction of roughly 70\% while reaching the same accuracy scale.  Our experiments also provide clear evidence for the certification frequency tradeoff when the bias parameters are unknown.
\end{itemize}

\subsection{Related work}
\paragraph{LLM data curation and noisy supervision.}
LLMs are increasingly used to improve data quality before it enters a downstream learning or decision problem. For missing-data imputation, \citet{hayat2024claim} use contextual information from pre-trained language models, \citet{he2025llmforest} combine graph-augmented prompting with an ensemble of LLMs, and \citet{wang2025llm4hrs} develop an LLM-based spatiotemporal imputation model for highly sparse remote-sensing data. Related work uses LLMs for general-purpose data cleaning and matching \citep{zhang2024jellyfish,parciak2024schema,peeters2025entity}, and for identifying or correcting noisy interactions in recommender systems \citep{wang2025llard,wang2025llm4dsr}. On the labeling side, LLMs can serve as scalable annotators or assist human annotators \citep{gilardi2023chatgpt,goel2023llms}; recent work also studies how to detect label errors \citep{nahum2025labelerrors} and train classifiers robustly from LLM-generated noisy labels \citep{ye2025calibrating}. These papers primarily ask how LLMs can improve upstream data quality. Our focus is complementary: given several data-generation or curation pipelines with different costs and fidelities, we ask how their value changes with the stage of the downstream optimization problem.

\paragraph{Cost-aware model selection and data markets.}
Closest to our cost perspective are systems that choose different models for different inference queries. \citet{chen2024frugalgpt} propose FrugalGPT which exploits heterogeneous model prices and task performance to construct lower-cost cascades. \citet{ong2025routellm} propose RouteLLM that learns to route prompts between stronger and weaker models using preference data. 
  In a concurrent work, \citet{belloni2026online} develop an online learning approach to adaptively query and select LLM APIs. Our setting is different:
the source is selected according to the \emph{progress of a downstream optimization algorithm}, rather than independently for each inference query. 
More broadly, the economics literature studies data as a priced asset and analyzes data markets, valuation, and pricing distortions \citep{acemoglu2022data,farboodi2023data}. The value of clean versus corrupted data also appears outside the LLM setting. Denoising is central to diffusion-based generative modeling \citep{ho2020ddpm}, while \citet{daras2025ambient} explicitly quantify how noisy training samples trade off against clean samples in ambient diffusion.

\paragraph{Stochastic, multi-fidelity, and multilevel optimization.}
Multi-fidelity optimization and simulation exploit sources of different costs and accuracies, often to reduce the number of expensive evaluations \citep{peherstorfer2018multifidelity,kandasamy2017multifidelity}. In first-order methods and stochastic approximation \citep{polyak1990new,nemirovski2009robust}, inexact-oracle models quantify the effect of deterministic errors \citep{devolder2014inexact}, while biased stochastic-gradient analyses study persistent gradient bias; see, e.g., \citet{ajalloeian2020biased}. Our noise model is also related to stochastic optimization with state-dependent noise conditions that replace a uniform variance bound by one that decreases near the solution \citep{gower2019sgd, aybat2019universally, gower2021structured,woodworth2021even,kotsalis2022simple,juditsky2023sparse,ilandarideva2025state,rodomanov2024universality}. As one concrete implementation of our meta-algorithms, we use SGE-type acceleration proposed by \citet{lan2018rge}, whose extrapolation structure is well suited to state-dependent noise as shown in \citet{ilandarideva2025state}. Moreover, our unknown-bias construction is related to learning from weak or inexpensive information augmented by a small trusted source \citep{ratner2017snorkel,hendrycks2018trusted,angelopoulos2023ppi}, and uses recursive regularization for gradient minimization in stochastic settings \citep{allen2018gradient,foster2019complexity} for certification.

Closely related to our bias--variance--cost perspective is the multilevel Monte Carlo (MLMC) line of work on stochastic optimization with biased oracles. \citet{hu2021biasvariancecost,hu2024mlmc} study a hierarchy of increasingly accurate and increasingly expensive biased stochastic oracles and design MLMC gradient estimators that exploit the joint tradeoff among bias, variance, and query cost.
This line builds on MLMC and randomized-debiasing ideas for stochastic approximation \citep{giles2008multilevel,blanchet2015unbiased,dereich2019general,blanchet2017unbiased,beatson2019efficient,asi2021stochastic}. In particular, \citet{beatson2019efficient} tune randomized telescoping estimators according to optimization progress per unit of computation, while \citet{asi2021stochastic} use multilevel randomization to construct low-bias, low-cost estimators for optimization subroutines. Our setting shares the principle that lower bias should be purchased only when it is worth the additional cost, but the oracle model and algorithmic mechanism are different. The MLMC framework assumes a structured approximation hierarchy and exploits coupled difference estimators across adjacent fidelity levels. By contrast, we allow a finite collection of heterogeneous black-box stochastic oracles that need not be nested, ordered by fidelity, or coupled through difference estimators. 


\subsection{Notation and organization}\label{subsec_notation}
Throughout the paper, we let $\|\cdot\|$ denote the Euclidean norm for vectors and the spectral norm for matrices, and let $\langle\cdot,\cdot\rangle$ denote the Euclidean inner product. We let $x^*$ denote an optimal solution of \eqref{eq:problem}. We use $\E$ for expectation. For symmetric matrices $A$ and $B$, $A\preceq B$ means that $B-A$ is positive semidefinite, and $\lambda_{\max}(A)$ denotes the largest eigenvalue of $A$. For a nonempty closed convex set $Y$, we denote the Euclidean projection by
\[
    \Pi_Y(x):=\arg\min_{z\in Y}\|z-x\|.
\]
For a stepsize $\gamma>0$, the projected-gradient mapping of $f$ over $Y$ is defined as
\begin{align}\label{def_projected_gradient}
    \PG^f_{Y,\gamma}(x)
    :=
    \tfrac{1}{\gamma}\left[x-\Pi_Y\bigl(x-\gamma\nabla f(x)\bigr)\right].
\end{align}
We use the shorthand $\PG_{Y,\gamma}(x)$ when the function $f$ is clear from the context. Finally, $\widetilde O(\cdot)$ suppresses logarithmic factors.

The rest of the paper is organized as follows. Section~\ref{sec:applications} gives concrete examples on LLM-generated labels and LLM-completion and verifies the oracle assumptions. Section~\ref{sec:known_bias} develops the known-bias meta-algorithm and proposes biased SGE as a concrete contraction subroutine. Section~\ref{sec:when-pay} characterizes when higher fidelity is worth paying for, including oracle domination and break-even-price interpretations. Section~\ref{sec:unknown_bias} develops the unknown-bias budgeted-certification meta-algorithm, proposes a projected-gradient reduction subroutine using recursive regularization, and studies how often certification should be purchased. Section~\ref{sec:numerics} reports the numerical experiments.

\section{Applications: LLM-generated labels and completed data}\label{sec:applications}
In this section, we provide two concrete examples showing how the mixed-fidelity assumptions arise in LLM-assisted data pipelines. The first considers LLM-generated or LLM-corrected labels; while the second considers LLM reconstruction of missing covariates. In both cases, stronger models, richer prompts, ensembles, or human review can be represented as more expensive oracles with different bias and variance parameters.

\subsection{LLM-generated noisy labels}
Consider a generalized linear model with feature vector $\theta$ and clean conditional mean
\begin{align}\label{eq:glm}
\E[y|\theta]=u(\theta^\top x^*),
\end{align}
where $u=v'$ is $\bar r$-Lipschitz and $\underline r$-strongly monotone.  The clean population objective is
\begin{align}\label{eq:glm_objective}
f(x)=\E\!\left[v(\theta^\top x)-y\theta^\top x\right].
\end{align}
Assume $x^*$ is an interior point over the feasible region $X$, so that $\nabla f(x^*) = 0$.

Suppose the stochastic oracle $\lambda$ obtains labels from an LLM, a human annotator, or a hybrid procedure, such that
\begin{equation}
    y^\lambda=u(\theta^\top x^*)+b_\lambda(\theta)+\varepsilon^\lambda,
    \label{eq:label-source}
\end{equation}
where $\E[\varepsilon^\lambda|\theta]=0$ and $b_\lambda(\theta)$ represents systematic annotation error. Then the stochastic gradient
\[
\cG_\lambda(x;\theta,y^\lambda)=\theta\bigl(u(\theta^\top x)-y^\lambda\bigr)
\]
has mean $g_\lambda(x) = \bbe[\cG_\lambda(x;\theta,y^\lambda)]$ satisfying
\begin{align*}
g_\lambda(x)-\nabla f(x)=-\E[\theta b_\lambda(\theta)],
\end{align*}
so the bias condition \eqref{eq:bias} holds with
$\cB_\lambda=\|\E[\theta b_\lambda(\theta)]\|$. This systematic bias can be viewed as a natural distinction between, for example, a cheap LLM labeler with systematic class-specific errors and a more expensive expert-reviewed labeler.

The state-dependent variance condition \eqref{eq:variance} follows under a standard fourth-moment assumption. Let $\Sigma=\E[\theta\theta^\top]$ and suppose $
\E[\|\theta\|^2\theta\theta^\top]\preceq \varsigma\Sigma.$
Invoke that the strong monotonicity of $u$ indicates
$f(x) - f^*\ge (\underline r/2)(x-x^*)^\top\Sigma(x-x^*)$. We can decompose $\cG_\lambda-g_\lambda$ to obtain
\begin{align}
\E\|\cG_\lambda(x; \theta, y^\lambda)-g_\lambda(x)\|^2 &= \bbe\left\|\theta\bigl(u(\theta^\top x)-u(\theta^\top x^*)\bigr) - \theta(b_\lambda(\theta) + \varepsilon^\lambda) +  \E[\theta b_\lambda(\theta)]\right\|^2\nn\\
&\leq 2\bbe\left\|\theta\bigl(u(\theta^\top x)-u(\theta^\top x^*)\bigr)\right\|^2 +\sigma_{\lambda,*}^2\nn\\
&\leq 2\bar r^2\bbe[\|\theta\|^2\|\theta^\top (x- x^*)\|^2] +\sigma_{\lambda,*}^2\nn\\
&\leq 2\bar r^2 \varsigma (x-x^*)^\top\Sigma(x-x^*) + \sigma_{\lambda,*}^2\nn\\
&\leq
\tfrac{4\bar r^2\varsigma}{\underline r}[f(x) - f^*]+\sigma_{\lambda,*}^2,
\label{eq:label-state-noise}
\end{align}
where $
\sigma_{\lambda,*}^2
=2\E\left\|\theta\bigl(b_\lambda(\theta)+\varepsilon^\lambda\bigr)
-\E[\theta b_\lambda(\theta)]\right\|^2.
$
Thus, different label sources can share essentially the same state-dependent component that does not depend on the labeling error $b_\lambda(\theta)$ and $\epsilon^\lambda$, while having very different bias and residual variance. In practice, repeated annotations or a small trusted calibration set can be used to estimate these quantities, which is precisely the setting motivating the certification mechanism in Section~\ref{sec:unknown_bias}.

\subsection{LLM completion of missing covariates}\label{subsec:application_missing_data}
In this subsection, we consider the same generalized linear model \eqref{eq:glm} as in Subsection~2.1, but our observation of $\theta$ has various missing entries. As before, assume $x^*$ is an interior point of $X$, and the clean population objective is \eqref{eq:glm_objective}.
Suppose the model--prompt pipeline $\lambda$ reconstructs missing entries of $\theta$ to produce $\widetilde\theta_\lambda$. The resulting stochastic-gradient oracle is
\[
    \cG_\lambda(x;\widetilde\theta_\lambda,y)
    =\widetilde\theta_\lambda\bigl(u(\widetilde\theta_\lambda^\top x)-y\bigr),
\]
with mean $g_\lambda(x)=\E[\cG_\lambda(x;\widetilde\theta_\lambda,y)]$. Define the bias at the optimum by $b_\lambda:=g_\lambda(x^*)$.

For $a,b\in\R$, define 
$
q(a,b)=
\frac{u(a)-u(b)}{a-b}
$
for $a \neq b$,
and choose any $q(a,a)$ as an arbitrary number in $[\underline r,\bar r]$. By strong monotonicity and Lipschitz continuity of $u$, we have
$\underline r\le q(a,b)\le\bar r$. Let
\[
Q(x)=\E\!\left[q(\theta^\top x,\theta^\top x^*)\theta\theta^\top\right],
\quad
Q_\lambda(x)=\E\!\left[q(\widetilde\theta_\lambda^\top x,
\widetilde\theta_\lambda^\top x^*)
\widetilde\theta_\lambda\widetilde\theta_\lambda^\top\right].
\]
Since $\nabla f(x^*)=0$, we have
\begin{align}
    g_\lambda(x)-\nabla f(x) &=  b_\lambda + \bbe\left[ \widetilde\theta_\lambda\bigl(u(\widetilde\theta_\lambda^\top x)-u(\widetilde\theta_\lambda^\top x^*)\bigr) - \theta\bigl(u(\theta x)-u(\theta^\top x^*)\bigr)\right] \nn\\
    &=b_\lambda+\bigl(Q_\lambda(x)-Q(x)\bigr)(x-x^*).
    \label{eq:completion-bias-decomp_glm}
\end{align}
Therefore, assuming $\max_{x \in X}\|x-x^*\|\le D$ and defining
$
\Delta_\lambda(D)
:=\sup_{\|x-x^*\|\le D}\|Q_\lambda(x)-Q(x)\|,
$
the uniform-bias assumption \eqref{eq:bias} holds with
\begin{equation}
    \cB_\lambda
    \le \|b_\lambda\|+\Delta_\lambda(D)D.
    \label{eq:completion-bias-bound}
\end{equation}
Thus, higher-quality completions can reduce both the bias at the optimum and the local curvature mismatch induced by reconstructing the covariates.

The variance also satisfies the state-dependent noise model. Suppose
$
\E\!\left[
\|\widetilde\theta_\lambda\|^2
\widetilde\theta_\lambda\widetilde\theta_\lambda^\top
\right]
\preceq
\varsigma_\lambda\E[\theta\theta^\top], 
$
and let
$
\bar\sigma_{\lambda,*}^2
:=
\E\left\|
\cG_\lambda(x^*;\widetilde\theta_\lambda,y)-g_\lambda(x^*)
\right\|^2.
$
Using the $\bar r$-Lipschitz continuity of $u$ and strong monotonicity,
$
f(x) - f^*
\ge
\frac{\underline r}{2}
(x-x^*)^\top\E[\theta\theta^\top](x-x^*),
$
we obtain a similar bound as in Ineq.~\eqref{eq:label-state-noise}, i.e.,
\begin{equation}
\E\|\cG_\lambda(x,\xi)-g_\lambda(x)\|^2
\le
\tfrac{4\bar r^2\varsigma_\lambda}{\underline r}[f(x) - f^*]
+2\bar\sigma_{\lambda,*}^2,
\label{eq:completion-state-noise_glm}
\end{equation}
which verifies the state-dependent noise assumption \eqref{eq:variance}.

\paragraph{Linear regression as a special case.} Linear regression is recovered by taking $u(t)=t$. In this case, 
$Q(x)=\Sigma := \E[\theta\theta^\top]$ and
$Q_\lambda(x)=\Sigma_\lambda := \E[\widetilde\theta_\lambda\widetilde\theta_\lambda^\top]$, thus Eq.~\eqref{eq:completion-bias-decomp_glm} reduces to the Hessian-mismatch decomposition 
\begin{equation}
	g_\lambda(x)-\nabla f(x)
	=b_\lambda+(\Sigma_\lambda-\Sigma)(x-x^*),
	\label{eq:completion-bias-decomp}
\end{equation}
and $\cB_\lambda
\le \|b_\lambda\|+\|\Sigma_\lambda-\Sigma\|D$.   The variance can also be characterized in a more fine-grained manner. Define $M_{\lambda,\xi}:=\widetilde\theta_\lambda\widetilde\theta_\lambda^\top-H_\lambda$. We have
\[
\cG_\lambda(x;\widetilde\theta_\lambda,y)-g_\lambda(x)
=M_{\lambda,\xi}(x-x^*)+\cG_\lambda(x^*;\widetilde\theta_\lambda,y)-g_\lambda(x^*).
\]
Define
$
\nu_\lambda^2
:=\lambda_{\max}\!\left(\E[M_{\lambda,\xi}^\top M_{\lambda,\xi}]\right)$ and 
$
\bar\sigma_{\lambda,*}^2:=\E\|\cG_\lambda(x^*;\widetilde\theta_\lambda,y)-g_\lambda(x^*)\|^2.
$
Using strong convexity yields
\begin{equation}
	\E\|\cG_\lambda(x,\xi)-g_\lambda(x)\|^2
	\le
	\tfrac{4\nu_\lambda^2}{\mu}[f(x) -f^*]+2\bar\sigma_{\lambda,*}^2,
	\label{eq:completion-state-noise}
\end{equation}
which verifies the state-dependent noise assumption \eqref{eq:variance}.
This example directly matches our LLM-completion experiments: different models and prompts induce different reconstructed datasets, costs, bias levels, and state-dependent variance.

\section{Known bias: cost-aware oracle-switching}\label{sec:known_bias}
In this section, we introduce a multistage adaptive oracle switching strategy assuming the bias parameters $\{\cB_\lambda\}_{\lambda=1}^\ell$ are known. 
To enable the design of the multistage meta-algorithm, we first assume access to a stochastic optimization subroutine that can reduce the current expected squared distance to the optimum by a constant factor at a predictable cost. A concrete choice of an efficient subroutine will be discussed in Subsection~\ref{subsec:sge}.

\begin{definition}[Distance-contraction subroutine]\label{def:contraction-routine}
Consider an optimization subroutine
$
\mathsf{Contract}(y,R;\SO_\lambda),
$
which takes an initial point $y$, a target radius $R$, and a stochastic oracle $\SO_\lambda$ as inputs. Associated with the subroutine are a bias-admissibility condition $\mathsf{Adm}_{\rm C}(R,\cB_\lambda)$ and an oracle complexity bound $T(R,\lambda)$.
We say that $\mathsf{Contract}$ is a valid distance-contraction subroutine if, whenever
\[
\E\|y-x^*\|^2\le R^2
\quad\text{and}\quad
\mathsf{Adm}_{\rm C}(R,\cB_\lambda)\text{ holds},
\]
it uses at most $T(R,\lambda)$ calls to $\SO_\lambda$ and returns $y^+$ satisfying $\E\|y^+-x^*\|^2\le \frac{R^2}{2}.$
\end{definition}

Here $\mathsf{Adm}_{\rm C}(R,\cB_\lambda)$ represents the bias requirement needed for the stochastic oracle SO$_\lambda$ to achieve the desired contraction, which depends on oracle SO$_\lambda$ only through its bias level $\cB_\lambda$. In contrast, the complexity bound $T(R,\lambda)$ depends on problem parameters such as $L$ and $\mu$, and other oracle characteristics such as its variance parameters $\sigma_{\lambda,*}^2$ and $\cL_\lambda$. The initial point $y$ in Definition~\ref{def:contraction-routine} may itself be random. The expectation is taken over both the randomness of $y$ and the fresh randomness used by the subroutine. 

Algorithm~\ref{alg:known} requires only this input--output guarantee as in Definition~\ref{def:contraction-routine}, so the inner subroutine may be the classic stochastic approximation type methods, variance-reduced methods, or any other procedure with a known bias-admissibility condition and sample complexity budget. In particular, at stage $s$, given
$
\E\|y^{s-1}-x^*\|^2\le R_{s-1}^2,
$
the algorithm selects, among the admissible oracles, the one with the lowest cost for producing the next contraction.

\begin{algorithm}[h]
\caption{Cost-aware contraction meta-algorithm with known bias}\label{alg:known}
\begin{algorithmic}[1]\small
\STATE \textbf{Input:} initial point $y^0$, radius $R_0$, oracle pool $\cP=\{\SO_\lambda\}_{\lambda=1}^\ell$, contraction subroutine $\mathsf{Contract}$.
\FOR{$s=1,2,\ldots$}
\STATE $\qquad\cP_s\gets\{\lambda\in\cP:\mathsf{Adm}_{\rm C}(R_{s-1},\cB_\lambda)\text{ holds}\}$; stop if $\cP_s=\varnothing$.
\STATE $\qquad\lambda_s\gets\arg\min_{\lambda\in\cP_s} c_\lambda T(R_{s-1},\lambda)$.
\STATE $\qquad y^s\gets\mathsf{Contract}(y^{s-1},R_{s-1};\SO_{\lambda_s})$.
\STATE $\qquad R_s^2\gets R_{s-1}^2/2$.
\ENDFOR
\end{algorithmic}
\end{algorithm}

A straightforward convergence guarantee for Algorithm~\ref{alg:known} is stated as follows.
\begin{proposition}[Stagewise convergence]\label{prop:known-restart}
Suppose $\E\|y^0-x^*\|^2\le R_0^2$ and the admissible set $\cP_s$ in Algorithm~\ref{alg:known} is nonempty at every completed stage. Then
\[
    \E\|y^s-x^*\|^2\le R_s^2=2^{-s}R_0^2.
\]
Consequently, after $K=\lceil\log_2(R_0^2/\epsilon)\rceil$ stages, we have
$\E\|y^K-x^*\|^2\le\epsilon$.
\end{proposition}
\begin{proof}
The claim follows by induction. At stage $s$, the selected oracle $\lambda_s$ is admissible, so Definition~\ref{def:contraction-routine} guarantees
$
\E\|y^s-x^*\|^2\le R_{s-1}^2/2=R_s^2.
$
\end{proof}
For simplicity, let us consider the unbiased setting where the bias admissibility condition $\mathsf{Adm}_{\rm C}(R,\cB_\lambda)$ holds for any $R>0$ and $\lambda \in [\ell]$. For $\epsilon >0$, to find a solution $\bar y$ such that $\bbe\|\bar y - x^*\|^2\leq \epsilon$ requires $K=\lceil\log_2(R_0^2/\epsilon)\rceil$ stages. The total cumulative cost of running Algorithm~\ref{alg:known} is
\begin{align*}
\Cost_{\rm adap}(\epsilon)=\sum_{s=1}^K\min_\lambda T(R_{s-1}, \lambda).
\end{align*}
In comparison, the total cumulative cost of always using the best fixed oracle in each stage can be bounded by
\[
\Cost_{\rm fixed}(\epsilon)=\min_\lambda\sum_{s=1}^KT(R_{s-1}, \lambda).
\]
Therefore, we always have $\Cost_{\rm adap}(\epsilon)\le\Cost_{\rm fixed}(\epsilon)$. The largest potential improvement for $\Cost_{\rm adap}(\epsilon)$ compared with $\Cost_{\rm fixed}(\epsilon)$ will be discussed after presenting the stochastic gradient extrapolation (SGE) subroutine in the next subsection.

\subsection{Biased SGE: an efficient contraction subroutine}\label{subsec:sge}
We next show that stochastic gradient extrapolation (SGE) can serve as an efficient subroutine under our biased state-dependent noise model. Specifically, Algorithm~\ref{alg:sge} extends the existing design and analyses of SGE \citep{lan2018rge,ilandarideva2025state}  by allowing persistent gradient bias and incorporating strong convexity through an additional regularizer $\tfrac{\mu\alpha_t}{2}\|x_{t-1}-z\|^2$. This strongly convex variant of SGE will be reused in Subsection~\ref{subsec:rr-sge}, as a subroutine for the recursive regularization algorithm for gradient minimization. 

In this subsection, we consider a fixed stochastic oracle $\lambda$. Therefore, for the sake of simplicity, we use the simplified notation $\cG$ to denote the selected stochastic first-order oracle and $g(x) = \bbe[\cG(x, \xi)]$, and use $\cL$, $\sigma_*^2$, and $\cB$ to denote its variance and bias parameters. Additionally, we assume $x^*$ is an interior point of the feasible region $X$, so that $\nabla f(x^*) = 0$.

\begin{algorithm}[h]
\caption{Stochastic gradient extrapolation (SGE)}\label{alg:sge}
\begin{algorithmic}[1]\small
\STATE \textbf{Input:} $x_0=x_{-1}=z_0$ and parameters $\{\alpha_t,\eta_t,\tau_t\}$.
\STATE{\textbf{for} $t=1,2,\ldots$ \textbf{do}}
\STATE $\qquad \widetilde G_t\gets G_{t-1}(x_{t-1})+\alpha_t\big(G_{t-1}(x_{t-1})-G_{t-1}(x_{t-2})\big)$.
\STATE $\qquad z_t\gets\arg\min_{z\in X}\{\inner{\widetilde G_t}{z}+\tfrac{\mu\alpha_t}{2}\|x_{t-1}-z\|^2+\tfrac{\eta_t}{2}\|z_{t-1}-z\|^2\}$.
\STATE $\qquad x_t\gets(z_t+\tau_t x_{t-1})/(1+\tau_t)$ and draw $G_t(\cdot) = \tfrac{1}{m} \sum_{i=1}^{m} \cG(\cdot, \xi_{t,i})$.
\end{algorithmic}
\end{algorithm}


The convergence guarantees of Algorithm~\ref{alg:sge} are characterized in the following theorem.
\begin{theorem}[Biased SGE]\label{thm:sge}
Let $\{x_t\}$ be generated by Algorithm~\ref{alg:sge}. 
Set $\alpha_t=\frac{t-1}{t}$, $\tau_t=\frac{t-1}{3}$, $\eta_t=\frac{\eta}{t}$ for $t\le2$, and $\eta_t=\frac{\eta}{t}+\frac{\mu(t-3)}{2}$ for $t\ge3$, where
$\eta\ge\max\{24L,\frac{54\cL}{m} + \frac{162\cL^2}{\mu m^2}\}$. Then\footnote{The second inequality in \eqref{eq:sge-rate} is proved using $\nabla f(x^*) = 0$, and the first inequality is proved without this assumption.}, for $k\ge3$,
\begin{align}
\E[f(x_k)-f^*]
\le &\left(
\tfrac{72L}{k(k+2)}
+\tfrac{162\cL}{mk(k+2)}
+\tfrac{486\cL^2}{\mu m^2k(k+2)}
\right)\|x_0-x^*\|^2 + \tfrac{2}{9k(k+2)}[f(x_0)-f^*]+\tfrac{192\sigma_*^2}{\mu m(k+2)}
+\tfrac{216\cB^2}{\mu}\nn\\
\le&
\left(
\tfrac{73L}{k(k+2)}
+\tfrac{162\cL}{mk(k+2)}
+\tfrac{486\cL^2}{\mu m^2k(k+2)}
\right)\|x_0-x^*\|^2 +\tfrac{192\sigma_*^2}{\mu m(k+2)}
+\tfrac{216\cB^2}{\mu}. \label{eq:sge-rate}
\end{align}
In particular,  for $m = 1$ and $k \geq \max \{98\sqrt{\tfrac{L}{\mu}}, 144\sqrt{\tfrac{\cL}{\mu}}, \tfrac{102\cL}{\mu}, \tfrac{3072\sigma_*^2}{\mu^2\|x_0-x^*\|^2}\}$, we have
\begin{align}\label{sge_contraction}
\bbe\|x_k - x^*\|^2 \leq \tfrac{2}{\mu}\bbe[f(x_k)-f^*] \leq \tfrac{1}{4}\|x_0-x^*\|^2 + \tfrac{432 \cB^2}{\mu^2}.
\end{align}
\end{theorem}

\noindent\textit{Proof sketch.} 
To prove the claim, we need to control the state-dependent
variance, and also bound the error introduced by the persistent
oracle bias. We begin with the SGE recursion stated in Proposition \ref{prop:general_convergence}, whose proof follows from a similar analysis of SGE as in \cite{ilandarideva2025state} (see Appendix \ref{subsec:proof_of_proposition}).
With this recursion, we can handle the state-dependent variance and systematic bias transparently. Our parameter choices ensure that the state-dependent variance
terms can be absorbed into the converging function-value terms in this
recursion. Taking the telescoping sum over the recursions yields
\[
A_k \mathbb{E}[f(x_k)-f^\ast]
+\frac{P_k}{2}\mathbb{E}\|z_k-x^\ast\|^2
\leq
\Gamma_k+B\sum_{t=1}^k(3t-2)\mathbb{E}\|z_t-x^\ast\|,
\]
where $A_k$ and $P_k$ grow quadratically in $k$, and $\Gamma_k$
contains the initial-error and residual-variance terms. To bound the accumulated bias, observe that the same inequality
controls the distance of each $z_t$ from $x^\ast$. Thus, the
distances appearing on the right-hand side can themselves be
bounded in terms of the accumulated error. Applying Young's
inequality provides a bound of the form
$2\Gamma_k+O(k^2B^2/\mu)$, and dividing both sides of the inequality by $A_k$ yields the stated
function-value bound. By further using the strong convexity property, we can obtain the
distance-contraction guarantee. The full proof is deferred to Appendix~\ref{sec:proof_sge_thm}.\qed

\subsection{Combining the oracle-switching meta-algorithm with the SGE subroutine}
Theorem~\ref{thm:sge} provides an explicit implementation of Definition~\ref{def:contraction-routine}. For SGE, the bias-admissibility condition is
\begin{equation}
    \mathsf{Adm}_{\rm C}^{\rm SGE}(R,\cB)
    \quad\Longleftrightarrow\quad
    \cB\le \tfrac{\sqrt{3}\mu R}{72},
    \label{eq:sge-admissibility}
\end{equation}
and its oracle complexity is
\begin{equation}
T^{\rm SGE}(R,\lambda)
=
\max\left\{
98\sqrt{\tfrac{L}{\mu}},
144\sqrt{\tfrac{\cL_\lambda}{\mu}},
\tfrac{102\cL_\lambda}{\mu},
\tfrac{3072\sigma_{\lambda,*}^2}{\mu^2R^2}
\right\}.
\label{eq:stage-complexity}
\end{equation}
Thus, we can use $\mathsf{Contract}^{\rm SGE}(y,R;\SO_\lambda)$ to represent running Algorithm~\ref{alg:sge} from a point $y$ with oracle $\SO_\lambda$ for $T^{\rm SGE}(R,\lambda)$ iterations. By Theorem~\ref{thm:sge}, it achieves the contraction in expected squared distance to the optimal solution whenever $\mathsf{Adm}_{\rm C}^{\rm SGE}(R,\cB_\lambda)$ holds. 
Consequently, the adaptive oracle switching algorithm (Algorithm~\ref{alg:known}) can be instantiated with $\mathsf{Contract}^{\rm SGE}$.

We next discuss the cost comparison between Algorithm \ref{alg:known} and using the best fixed oracle in all stages, quantifying the benefit of adaptive switching.
\begin{corollary}[Comparison with the best fixed oracle]\label{corollary_known}
Suppose all oracles are unbiased, i.e., $B_\lambda = 0,~\forall \lambda$. Define $A_\lambda=c_\lambda \max \left\{98\sqrt{\frac{L}{\mu}}, 144\sqrt{\tfrac{\cL_\lambda}{\mu}}, \tfrac{102\cL_\lambda}{\mu}\right\}$ and $B_\lambda=c_\lambda \cdot \frac{3072\sigma^2_{\lambda,*}}{\mu^2R_0^2}$. For $\epsilon >0$, let $K=\lceil\log_2(R_0^2/\epsilon)\rceil$. Then the total costs of  Algorithm \ref{alg:known} and using the best-fixed oracle throughout satisfy
\begin{align}\label{total_cost_results}
\Cost_{\rm adap}(\epsilon)&=\sum_{s=1}^K\min_\lambda\max\{A_\lambda,B_\lambda2^{s-1}\}, \nn\\
\Cost_{\rm fixed}(\epsilon)&=\min_\lambda\sum_{s=1}^K\max\{A_\lambda,B_\lambda2^{s-1}\} = \mathcal{O}\left( \min_{\lambda} \left\{c_\lambda \cdot \left(\sqrt{\tfrac{L}{\mu}}\log_2\tfrac{R_0^2}{\epsilon} +\tfrac{\cL_\lambda}{\mu}\log_2\tfrac{R_0^2}{\epsilon} + \tfrac{\sigma_{\lambda,*}^2}{\mu^2 \epsilon}\right)\right\}  \right),
\end{align}
hence $\Cost_{\rm adap}(\epsilon)\le\Cost_{\rm fixed}(\epsilon)$. 
Moreover, there exist instances for which $ \frac{\Cost_{\rm fixed}(\epsilon)}
     {\Cost_{\rm adap}(\epsilon)}= \Theta\!\left(\tfrac{K}{\log K}\right) = \Theta\!\left(
\tfrac{\log(R_0^2/\epsilon)}
{\log\log(R_0^2/\epsilon)}
\right) $, and this order is tight for stage costs of the form
$\max\{A_\lambda,B_\lambda2^{s-1}\}$.
\end{corollary}

\noindent\textit{Proof sketch.}
The comparison with the best fixed oracle follows directly
from minimizing the cost at each stage: the sum of the
stagewise minima is no larger than the minimum of the sums.
It remains to show how large this improvement can be.

We first construct two oracles. The first has a smaller
coarse-scale cost but a larger residual-variance term, whereas
the second has a constant stage cost of $M$ throughout the
optimization horizon. The parameters are chosen so that the
first oracle is cheaper until the final
$m=\lceil\log_2 K\rceil$ stages. Thus, the adaptive policy
incurs a total cost of $\Theta(M\log K)$, while the best
fixed oracle incurs a cost of $MK$.

To show that this order cannot be improved, let $d_s$ denote
the minimum cost at stage $s$, and consider an oracle attaining
$d_{K-m}$. Its cost over all $K$ stages is at most
$O(Kd_{K-m})$. On the other hand, because $d_s$ is
nondecreasing, the adaptive policy pays at least
$(m+1)d_{K-m}$ over the final $m+1$ stages. Combining these
bounds gives an improvement of at most $O(K/\log K)$.
The construction with tight examples, and the complete bounds are presented
in Appendix~\ref{proof_corollary_known}. \qed

\vspace{0.2cm}
Notably, when ignoring the oracle cost, the sample complexity $\mathcal{O}\left( \sqrt{\tfrac{L}{\mu}}\log_2\tfrac{R_0^2}{\epsilon} +\tfrac{\cL_\lambda}{\mu}\log_2\tfrac{R_0^2}{\epsilon} + \tfrac{\sigma_{\lambda,*}^2}{\mu^2 \epsilon}\right) $ in \eqref{total_cost_results} matches the optimal complexity for stochastic optimization under state-dependent noise, supported by lower bounds in \cite{nemirovskij1983problem,woodworth2021even}. By adaptively switching among the available oracles, Algorithm~\ref{alg:known} can achieve more cost-efficient convergence while preserving such complexity structures.

\section{When is a better LLM worth paying for?}\label{sec:when-pay}
The cost-aware meta-algorithm (Algorithm~\ref{alg:known}) provides a simple economic interpretation of oracle fidelity. For any contraction subroutine with associated bias-admissibility condition $\mathsf{Adm}_{\rm C}$ and oracle complexity $T$, define
\begin{equation}
\cP(R):=\{\lambda:\mathsf{Adm}_{\rm C}(R,\cB_\lambda)\text{ holds}\},
\quad
C_\lambda(R):=c_\lambda T(R,\lambda).
\label{eq:admissible-stage-cost}
\end{equation}
The basic willingness-to-pay logic depends only on the admissibility rule and the cost of the next contraction.
Specifically, at radius $R$, an oracle is useful only if it is accurate enough to belong to $\cP(R)$, and among those oracles, it buys the cheapest next contraction. 

For the SGE implementation, $\cP(R)$ is given by condition \eqref{eq:sge-admissibility}. For simplicity, writing
\[
D_\lambda:=\max\left\{98\sqrt{\tfrac L\mu},
144\sqrt{\tfrac{\cL_\lambda}{\mu}},
\tfrac{102\cL_\lambda}{\mu}\right\},
\quad
E_\lambda:=\tfrac{3072\sigma_{\lambda,*}^2}{\mu^2},
\]
we have
\begin{equation}
C_\lambda(R)=c_\lambda\max\{D_\lambda,E_\lambda/R^2\}.
\label{eq:stage-cost-decomp}
\end{equation}
This decomposition separates three roles of fidelity. The bias determines \emph{whether} an oracle can support the requested accuracy; the $c_\lambda D_\lambda$ term governs its coarse-scale cost (or equivalently the cost in the non-asymptotic phase); and $c_\lambda E_\lambda/R^2$ governs its fine-scale cost as the shrinkage of $R$ (or equivalently the cost in the asymptotic phase). The bias requirement can be equivalently written as
$
R\ge R_\lambda^{\rm bias}:=\frac{72\cB_\lambda}{\sqrt3\mu},
$
so an oracle may be attractive early and become infeasible later.

The 10\% experiment in Table~\ref{tab:llm-motivation} makes the mechanism visible. Treating the reported bias diagnostic as the admissibility proxy used in the experiment and noticing that we use whitening to make $\mu=1$, then the bias threshold is $\sqrt3R/72$. At $R=200$, all displayed pipelines except Nano simple-7 pass the test, so a low-cost Nano pipeline is sufficient. At $R=100$, the threshold is about $2.41$, which removes the Nano pipelines and makes Luna the cheapest admissible high-fidelity alternative. At $R=80$, the threshold is about $1.92$, leaving Avg.~Haiku+Sonnet as the only displayed admissible pipeline. Price alone cannot rescue an oracle that fails this accuracy test.\footnote{
In practical implementations, one may calibrate a less conservative admissibility constant. Section 6.2 describes the rule used in our experiments; this empirical calibration is not used in our theoretical guarantees.}

\paragraph{Dominated oracles.}
The policy also identifies sources that would never be queried.
\begin{remark}[Domination]\label{rem:domination}
Oracle $\nu$ dominates oracle $\lambda$ if
\[
\cB_\nu\le\cB_\lambda,
\quad
c_\nu D_\nu\le c_\lambda D_\lambda,
\quad
c_\nu\sigma_{\nu,*}^2\le c_\lambda\sigma_{\lambda,*}^2.
\]
Then $\nu$ is admissible whenever $\lambda$ is, and $C_\nu(R)\le C_\lambda(R)$ for every $R$; hence Algorithm~\ref{alg:known} with SGE subroutine never selects $\lambda$.
\end{remark}
Therefore, a more expensive model is not valuable merely because it is nominally ``stronger'': it must improve a downstream quantity that matters at some optimization scale. In Table~\ref{tab:llm-motivation}, Claude Haiku~4.5 dominates Claude Sonnet~5
($\mathcal B_\lambda$: $2.01$ vs.\ $2.37$; $c_\lambda\sigma^2_{\lambda,*}$: $1.23\times10^{5}$
vs.\ $2.76\times10^{5}$; equal $\mathcal L_\lambda$), so Sonnet~5 is never selected --- under any stage schedule. Note that reduction of cost or bias achieved by this oracle, would have rendered it non-dominated in this experiment. 

\paragraph{Supplier's viewpoint: break-even prices.}
The oracle selection rule can be read from the supplier's side. For an admissible oracle $\lambda$, define its stage-$R$ break-even price
\begin{equation}
    c_\lambda^{\rm BE}(R)
    :=
    \frac{\displaystyle\min_{\nu\in\cP(R)\setminus\{\lambda\}}
    c_\nu T(R,\nu)}{T(R,\lambda)}.
    \label{eq:breakeven-price}
\end{equation}
At radius $R$, oracle $\lambda$ is price-competitive exactly when $c_\lambda\le c_\lambda^{\rm BE}(R)$, subject to admissibility.  For the SGE implementation, the exact expression is
\begin{equation}
    c_\lambda^{\rm BE}(R)
    =
    \frac{\displaystyle\min_{\nu\in\cP(R)\setminus\{\lambda\}}
    c_\nu\max\{D_\nu,E_\nu/R^2\}}
    {\max\{D_\lambda,E_\lambda/R^2\}}.
    \label{eq:breakeven-sge}
\end{equation}
This formula makes clear that a break-even price generally depends on both the state-dependent-noise term $D_\lambda$ and the residual-variance term $E_\lambda/R^2$, in addition to the bias-dependent competitor set $\cP(R)$. For a fixed competitor $\nu$, the corresponding pairwise break-even price is
\[
    c_{\lambda\leftarrow\nu}^{\rm BE}(R)
    =c_\nu\frac{\max\{D_\nu,E_\nu/R^2\}}
    {\max\{D_\lambda,E_\lambda/R^2\}}.
\]
If both oracles are in their coarse regimes, this reduces to $c_\nu D_\nu/D_\lambda$; if both are in their fine regimes, it reduces to
$c_\nu E_\nu/E_\lambda=c_\nu\sigma_{\nu,*}^2/\sigma_{\lambda,*}^2$. In a mixed regime it retains an explicit dependence on $R$.
Thus willingness to pay for a stochastic oracle, e.g., an LLM pipeline, is not a single global number: it depends on the current downstream accuracy and on the cost--fidelity frontier of the competing models. Eq.~\eqref{eq:breakeven-price} can therefore be used either by a buyer deciding whether to upgrade or by a provider asking what price a particular quality improvement can sustain. This is a stage-dependent value-of-information calculation: the same signal can have very different economic value at different downstream states \citep{howard1966information,sims2003rational}.


Coming back to our \texttt{Aspiring Minds} example in Table~\ref{tab:llm-motivation}. The numerical calculation below isolates the regime when the residual variance dominates. Specifically, suppose Sonnet~5 and the cheapest admissible competitor at radius $R$ both satisfy $E_\lambda/R^2\ge D_\lambda$. Then the break-even price depends only on $c_\nu\sigma_{\nu,*}^2$, thus the following numbers should be interpreted as \emph{fine-scale break-even prices}, rather than break-even prices valid at every radius.
Let $\lambda=7$ denote Sonnet~5 and let $\tau=\sqrt3\mu R/72$ be the bias threshold. Define
\[
E_Q(\tau):=\min\{c_\nu\sigma^2_{\nu,*}:\mathcal B_\nu\le\tau,\ \nu\neq 7\}.
\]
Whenever Sonnet and the minimizing competitor are both residual-variance dominated, its break-even price is
$
    c_7^{\rm BE,Q}(\tau)=\frac{E_Q(\tau)}{\sigma^2_{7,*}}.
$
By hypothetically varying Sonnet's bias while keeping its residual variance fixed at the value in Table~\ref{tab:llm-motivation}, we obtain the following step function:
\[
\frac{E_Q(\mathcal B_7)}{\sigma^2_{7,*}}=
\begin{cases}
\$4.54, & \mathcal B_7\in[1.86,2.01) \quad\text{(Avg.\ Haiku+Sonnet)},\\
\$1.41, & \mathcal B_7\in[2.01,2.31) \quad\text{(Haiku~4.5)},\\
\$0.82, & \mathcal B_7\in[2.31,2.87) \quad\text{(Luna)},\\
\$0.20, & \mathcal B_7\in[2.87,3.95) \quad\text{(Nano LR-10/LR-short-10)},\\
\$0.18, & \mathcal B_7\ge3.95 \qquad\quad\ \text{(Nano LR-short-7)}.
\end{cases}
\]
At the measured bias $\mathcal B_7=2.37$, the relevant fine-scale competitor is Luna, giving a break-even price of about $\$0.82$. Thus, in this residual-variance-dominated regime, Sonnet's observed price of $\$3.158$ would need to fall by about a factor of $3.9$ to become cost-competitive. 
Conversely, at Sonnet's actual price of $\$3.158$, the fine-scale calculation becomes favorable only if its hypothetical bias falls below $2.01$, at which point the cheapest admissible competitor becomes the substantially more expensive Avg.~Haiku+Sonnet ensemble. Only if Sonnet's bias falls below $1.86$ would it also become the least-biased displayed pipeline and extend the best reachable bias floor in Table~\ref{tab:llm-motivation}.


\section{Unknown bias: budgeted trusted certification}\label{sec:unknown_bias}
In real-world applications, the bias parameters $\{\cB_\lambda\}_{\lambda=1}^{\ell}$ may themselves be difficult to estimate from the training sources, making Algorithm~\ref{alg:known} difficult to implement directly.\footnote{Compared with the bias parameters, the variance parameters like $\cL$ and $\sigma_*^2$ are relatively easier to estimate. There are also existing adaptive methods for the state-dependent noise setting that do not require prior knowledge of $\cL_\lambda$ or $\sigma_{\lambda,*}^2$; see, e.g., \citet{rodomanov2024universality}. Therefore, our paper focuses on the more challenging unknown bias setting.} Therefore, in this section, we assume occasional access to an expensive \emph{certification module} $\cC$ that returns the trusted gradient $\nabla f(y)$ at a cost $c_{\rm cert}>0$ per query. The module is accurate enough to certify optimization progress, but can be too expensive to use for routine training. This setup follows the broad principle of combining abundant machine-generated information with a small amount of trusted data; see, e.g., \citet{hendrycks2018trusted,angelopoulos2023ppi}.

The price of the certification module $c_{\rm cert}$ introduces an additional operational decision: 
\begin{center}
  \emph{How often should the optimization progress be certified?}  
\end{center}
 Certifying after every stage of contraction detects overly biased sources quickly, but repeatedly pays for the expensive trusted information. Certifying too rarely saves trusted queries but can spend substantial training cost on an oracle that has already become too biased to achieve the desired accuracy. Our meta-algorithm proposed in this section interpolates between these extremes using a specified budget $B$ on uncertified training expenditure.

\begin{remark}
    We analyze exact trusted-gradient certification in order to isolate this certification-frequency tradeoff from the statistical problem of constructing the certificate itself. In applications, the certification module $\cC$ may instead be constructed using a sufficiently large trusted sample, expert-reviewed observations, or repeated measurements. Standard concentration arguments can then construct a high-probability confidence bound for the projected-gradient norm. We leave the joint tradeoff among trusted-sample size, statistical precision, and certification frequency for a stochastic certification oracle to future work.
\end{remark}

Additionally, it should be highlighted that we use first-order rather than zeroth-order certification. In particular, certifying $f(x)-f(x^*)$ would require knowledge or accurate estimation of the typically unknown optimal value $f(x^*)$, whereas a trusted gradient directly yields a computable projected-gradient certificate.
Specifically, we use the projected gradient mapping $\PG_{Y, \gamma}(x)$ defined in \eqref{def_projected_gradient} as our certification criterion. For notational simplicity, we let
\begin{align}\label{def:projected_gradient}
\PG_{X}(x)
:= \PG_{X, 1/(2L)}(x) = 
2L\left[
x-\Pi_{X}\left(x-\tfrac{1}{2L}\nabla f(x)\right)
\right].
\end{align}
Our trusted certification module can evaluate $\|\PG_{X}(x)\|$. The following standard error bounds explain why this quantity provides a valid distance certificate.

\begin{lemma}[Projected-gradient certificate]
\label{lem:pg-certificate}
Let $X$ be closed and convex and suppose $x^*\in X$.
For any $\gamma\le2/L$, we have
\begin{align}\label{decent_ineq}
    f(x) - f(x^+) \geq (\gamma - \tfrac{\gamma^2L}{2})
\|\PG_{X,\gamma}(x)\|^2,
\end{align}
and
\begin{align}\label{projected_gradient_ineq}
    \|x-x^*\|
    \le
    \tfrac{2}{\mu}
    \|\PG_{X,\gamma}(x)\|.
\end{align}
Consequently, $\|\PG_{X,\gamma}(x)\|
    \le
    \tfrac{\sqrt{2}\mu R}{4}$ implies $\|x-x^*\|^2\le\tfrac{R^2}{2}$.
\end{lemma}

We defer the proof of Lemma~\ref{lem:pg-certificate} to Appendix~\ref{proof:lem:pg-certificate}. This lemma indicates that the certificate
$
\|\PG_{X}(x)\|\le \sqrt{2}\mu R_{s-1}/4
$
 guarantees
$
\|x-x^*\|^2\le R^2_{s-1}/2.
$
Thus, instead of a distance-contraction subroutine as in Definition~\ref{def:contraction-routine}, we require in this section the following projected-gradient reduction subroutine.

\begin{definition}[Projected-gradient reduction subroutine]\label{def:pg-reduction-routine}
Consider a subroutine
$
\mathsf{PGReduce}(y,R;\SO_\lambda),
$
where $y$ is an initial point, $R$ is the current radius, and $\SO_\lambda$ is the stochastic oracle used for training. Associated with the subroutine are a bias-admissibility condition $\mathsf{Adm}_{\rm PG}(R,\cB_\lambda)$ and a known oracle complexity $\widetilde T(R,\lambda)$.
We say that $\mathsf{PGReduce}$ is valid if, whenever
\[
\E\|y-x^*\|^2\le R^2
\quad\text{and}\quad
\mathsf{Adm}_{\rm PG}(R,\cB_\lambda)\text{ holds},
\]
it uses at most $\widetilde T(R,\lambda)$ oracle calls to $\SO_\lambda$ and returns a candidate $y^+\in X$ satisfying
\[
   \bbe\|\PG_{X}(y^+)\|\le \tfrac{\sqrt{2}\mu R}{16}.
\]
\end{definition}

Note that the bias-admissibility condition $\mathsf{Adm}_{\rm PG}(R,\cB_\lambda)$ is used only in the theoretical analysis and need not be known by the meta-algorithm presented below. Whether an oracle can actually support the requested progress is evaluated by the trusted projected-gradient certificate. This distinction is important: violating a sufficient PG-admissibility condition does not itself force rejection, since an oracle may still happen to pass the certificate test.

\subsection{Budgeted sparse certification meta-algorithm}
\label{subsec:budgeted-meta}
In this section, we propose the sparse certification meta-algorithm, which runs in blocks and uses a budget parameter $B\in[0,\infty]$ to determine the certification frequency. The algorithm uses both the contraction subroutine in Definition~\ref{def:contraction-routine} and the projected-gradient (PG) reduction subroutine in Definition~\ref{def:pg-reduction-routine}. The algorithm will eliminate the stochastic oracle if it does not pass the certification round.

\paragraph{Homogeneous blocks and progressive certification.}

Our key design choice is to use only one stochastic oracle SO$_\lambda$ inside each uncertified block. This makes a failed certification attributable to a single source, so that the source can be safely discarded after each unsuccessful block without any additional screening procedure.

For the current certified level $s \in \mathbb{N}$, let $y^s$ denote the current certified point, i.e.,
$
\|y^s-x^*\|^2\le R_s^2.
$
Let $\cP$ denote the current pool of non-rejected oracles. Define the stage cost quantities
\begin{align}\label{def_stage_cost}
C_\lambda (s) := c_\lambda T(R_s, \lambda), \quad \widetilde C_\lambda (s) := c_\lambda \widetilde  T(R_s, \lambda),
\end{align}
where $T(\cdot, \cdot)$ and $\widetilde T(\cdot, \cdot)$ are defined in Definitions \ref{def:contraction-routine} and \ref{def:pg-reduction-routine}, respectively.
At the beginning of a block, choose
\begin{equation}
\lambda
\in
\arg\min_{\nu\in\cP} C_\nu(s).
\label{eq:block-oracle-choice}
\end{equation}
The same oracle is then retained throughout this block. We define the parameter $B\in[0,\infty]$, interpreted as a \emph{single-trial budget for uncertified training}. The endpoint of the block is determined before any stochastic training is performed. Starting from level $s$, we keep adding contraction stages in Definition~\ref{def:contraction-routine} as long as (i) their cumulative cost does not exceed $B$, and (ii) $\lambda$ remains the cheapest oracle in the current pool at the next accuracy level. The last stage of the block is replaced by a projected-gradient (PG) reduction as in Definition~\ref{def:pg-reduction-routine}, which is designed for certification. A block is also forced to terminate when a final target level $K$ is prespecified.\footnote{If no final target level $K$ is prespecified, we can set $K = \infty$ and the algorithm will keep running until all the stochastic oracles are eliminated from the pool or the target accuracy is reached.} More precisely, let $W$ be the cumulative cost of the current block. Starting from $W=0$ and $k=s$, we declare stage $k$ to be the terminal stage of the block if any of the following conditions holds:
\begin{align}
&W+C_\lambda(k)>B, &\text{ (block budget constraint)}\label{eq:block-stop-budget}\\
&\min_{\nu\in\cP} C_\nu(k+1) < C_\lambda(k+1)
\text{ for } k<K-1, &\text{ (economic switching)} \label{eq:block-stop-switch}\\
&k=K-1 &\text{ (final target level)}\label{eq:block-stop-final}
\end{align}
Otherwise, stage $k$ is scheduled as a distance-contraction stage; we set
$W\gets W+C_\lambda(k)$ and continue to stage $k+1$.
A terminal stage of $t-1$ means the block aims to move from certified level $s$ to level $t$. By our construction,
\begin{equation}
\sum_{k=s}^{t-2} C_\lambda(k)\le B.
\label{eq:block-budget}
\end{equation}
Note that the cost of the last PG stage is not fully included in $B$.

One independent trial of this homogeneous block can be written as
\begin{equation}
\mathsf{Block}(y^s;s,t,\lambda)
:=
\mathsf{PGReduce}_{t-1}
\circ
\mathsf{Contract}_{t-2}
\circ\cdots\circ
\mathsf{Contract}_{s}(y^s),
\label{eq:block-definition}
\end{equation}
where all subroutines use the same stochastic oracle $\SO_\lambda$ but independent random samples. The same block might be repeated several times to boost the probability arguments, where different block trials use independent randomness. 
In the special case when $t=s+1$, the block consists only of the terminal PG step, with no contraction steps. 

For an endpoint $y$ of a block from $s$ to $t$, the trusted gradient can certify not only whether the full target $t$ has been reached, but also the amount of \emph{partial progress}. In particular, we define
\begin{equation}
\mathsf{Lev}_{s,t}(y)
:=
\max\left(
\{s\}
\cup
\left\{
j\in\{s+1,\ldots,t\}:
\|\PG_X(y)\|
\le
\tfrac{\sqrt{2}\mu R_{j-1}}{4}
\right\}
\right).
\label{eq:certified-level}
\end{equation}
By Lemma~\ref{lem:pg-certificate}, we have that $\mathsf{Lev}_{s,t}(y)=j>s$ implies $\|y-x^*\|^2\le R_j^2$.
Therefore, although the full block may be unsuccessful in reaching the desired accuracy of stage $t$, we do not need to discard all progress made by the oracle.

\paragraph{Block admissibility.}
We call oracle $\lambda$ \emph{block-admissible} for a planned block $(s,t)$ if
\begin{equation}
\mathsf{Adm}_{\rm C}(R_k,\cB_\lambda)
\quad\text{holds for }k=s,\ldots,t-2, \text{ and }
\quad
\mathsf{Adm}_{\rm PG}(R_{t-1},\cB_\lambda)
\quad\text{holds}.
\label{eq:block-admissibility}
\end{equation}
By Definitions~\ref{def:contraction-routine} and~\ref{def:pg-reduction-routine}, a block-admissible oracle $\lambda$ has the following validity guarantee.

\begin{lemma}[Validity of an admissible block]
\label{lem:block-validity}
Suppose $y^s$ is certified at level $s$, and let $\lambda$ be block-admissible for a planned block $(s,t)$. Then the output point
$
z=\mathsf{Block}(y^s;s,t,\lambda)
$
satisfies
\begin{equation}
\E\|\PG_X(z)\|
\le
\tfrac{\sqrt{2}\mu R_{t-1}}{16},
\label{eq:block-expected-pg}
\end{equation}
and consequently
\begin{equation}
\Prob\left(
\|\PG_X(z)\|
\le
\tfrac{\sqrt{2}\mu R_{t-1}}{4}
\right)
\ge \tfrac34.
\label{eq:block-success}
\end{equation}
\end{lemma}

\begin{proof}
Let $u_s:=y^s$, thus $\|u_s-x^*\|^2\le R_s^2.$ For $k=s,\ldots,t-2$ define
\[
u_{k+1}
:=
\mathsf{Contract}(u_k,R_k;\SO_\lambda).
\]
By Definition~\ref{def:contraction-routine} and the contraction-admissibility conditions in \eqref{eq:block-admissibility}, a simple induction implies
\[
\E\|u_k-x^*\|^2\le R_k^2,
\quad \text{for }
k=s,\ldots,t-1.
\]
Therefore, the input $u_{t-1}$ to the terminal projected-gradient reduction satisfies the mean-square warm-start condition at radius $R_{t-1}$. Definition~\ref{def:pg-reduction-routine} and the terminal PG-admissibility condition then imply \eqref{eq:block-expected-pg}. By applying Markov's inequality, we have
\[
\Prob\left(
\|\PG_X(z)\|>
\tfrac{\sqrt{2}\mu R_{t-1}}{4}
\right)
\le
\tfrac{
\E\|\PG_X(z)\|
}{
\sqrt{2}\mu R_{t-1}/4
}
\le \tfrac14,
\]
which proves \eqref{eq:block-success}. 
\end{proof}


Now we are at a position to introduce the complete budgeted sparse certification meta-algorithm (Algorithm~\ref{alg:budgeted-certification}). 

\begin{algorithm}[h]
\caption{Budgeted sparse certification with homogeneous blocks}
\label{alg:budgeted-certification}
\begin{algorithmic}[1]\small
\STATE \textbf{Input:} initial point $y^0$ such that $\|y^0 - x^*\|\leq R$, oracle pool $\cP$, budget $B$, subroutines $\mathsf{Contract}$ and $\mathsf{PGReduce}$, repetition cap $\mathcal R$, stage cap $K$; set $s\gets0$.
\WHILE{$s<K$ and $\cP\neq\varnothing$}
\STATE Choose $\lambda\in\arg\min_{\nu\in\cP}C_\nu(s)$.
\STATE Determine the block endpoint $t>s$ using conditions \eqref{eq:block-stop-budget}--\eqref{eq:block-stop-final}.
\STATE Set $j_{\rm best}\gets s$ and $y_{\rm best}\gets y^s$.
\FOR{$r=1,\ldots,\mathcal R$}
\STATE Generate an independent block output
$
z_r\gets\mathsf{Block}(y^s;s,t,\lambda).
$
\STATE Query the trusted certification module $\cC$ at $z_r$ and compute
$
j_r\gets\mathsf{Lev}_{s,t}(z_r).
$
\IF{$j_r>j_{\rm best}$}
\STATE Set $j_{\rm best}\gets j_r$ and $y_{\rm best}\gets z_r$.
\ENDIF
\IF{$j_r=t$}
\STATE \textbf{Accept full block:} set $y^t\gets z_r$, $s\gets t$, and \textbf{break} the for-loop.
\ENDIF
\ENDFOR
\IF{no trial certified level $t$}
\STATE \textbf{Keep certified partial progress:} if $j_{\rm best}>s$, set
$y^{j_{\rm best}}\gets y_{\rm best}$ and $s\gets j_{\rm best}$.
\STATE Permanently remove $\lambda$ from $\cP$.
\ENDIF
\ENDWHILE
\end{algorithmic}
\end{algorithm}

It is noteworthy that, in Algorithm~\ref{alg:budgeted-certification}, we adopt a sequential way to amplify the constant probability in Ineq.~\eqref{eq:block-success} to high probability guarantees. In particular, a successful oracle does not need to use all $\mathcal R$ repetitions: as soon as one block certifies the full target level $t$, the block is accepted, and no further training or trusted-gradient queries are required. Only an oracle that fails to produce a fully certified block after all $\mathcal R$ trials will be removed. The following lemma shows sequential implementation can be substantially cheaper than generating all $\mathcal R$ candidates before certification.
\begin{lemma}[Sequential certification]
\label{prop:sequential-block-cert}
Suppose a planned block is block-admissible, and let $N$ be the number of trials used by the sequential rule in Algorithm~\ref{alg:budgeted-certification}, truncated after $\mathcal R$ trials. Then
\begin{equation}\label{eq:sequential-expected-trials}
\Prob(\text{the block certification fails})
\le4^{-\mathcal R}, \quad \text{and}\quad \E[N]\le\tfrac43.
\end{equation}
\end{lemma}
\begin{proof}
By Lemma~\ref{lem:block-validity}, every independent block trial succeeds with probability at least $3/4$. Hence the probability that the first $r$ trials all fail is at most $4^{-r}$. Therefore
\[
\Prob(\text{the block certification fails})
\le 4^{-\mathcal R}.
\]
The expectation for the total number of trials for a successful certification is bounded by
\[
\E[N]
=
\sum_{r=0}^{\mathcal R-1}\Prob(N>r)
\le
\sum_{r=0}^{\mathcal R-1}4^{-r}
\le
\tfrac43,
\]
which completes the proof.
\end{proof}

The homogeneous-block construction relieves the attribution issue that would arise if different oracles were used inside one uncertified block. Here, all $\mathcal R$ trials of a failed block use the same $\lambda$, so that the false elimination issue can be avoided with high probability. 


The next proposition characterizes the high-probability convergence guarantees of Algorithm~\ref{alg:budgeted-certification}.

\begin{proposition}[High-probability convergence guarantee of budgeted certification]
\label{prop:budgeted-correctness}
Let
$
K=\lceil\log_2(R_0^2/\epsilon)\rceil
$
and choose
$
\mathcal R
\ge
\left\lceil
\log_4\frac{K+\ell}{\delta}
\right\rceil.
$
Suppose that whenever Algorithm~\ref{alg:budgeted-certification} has not yet reached certified level $K$, the current pool contains at least one oracle that is block-admissible. 
Then, with probability at least $1-\delta$, the algorithm terminates with a certified point $y$ satisfying
$
\|y-x^*\|^2\le\epsilon.
$
The total number of blocks is bounded by
$K+\ell = O\!\left(
\ell+\log\tfrac{R_0^2}{\epsilon}
\right),
$
and the number of trusted-gradient evaluations is at most
\[
O\!\left\{
\left(
\ell+\log\tfrac{R_0^2}{\epsilon}
\right)
\log\!\left(
\tfrac{\ell+\log(R_0^2/\epsilon)}{\delta}
\right)
\right\}.
\]
\end{proposition}
\begin{proof}
Consider a block-admissible oracle. By Lemma~\ref{lem:block-validity}, one independent block trial fails to certify the full target with probability at most $1/4$. Therefore all $\mathcal R$ independent trials fail with probability at most
$
4^{-\mathcal R}\le\tfrac{\delta}{K+\ell}.
$
Every fully successful block advances the certified level by at least one, so there can be at most $K$ such blocks. Every unsuccessful block permanently removes one oracle, so there can be at most $\ell$ unsuccessful blocks. A union bound over at most $K+\ell$ block decisions shows that, with probability at least $1-\delta$, no block-admissible oracle is incorrectly removed.

Every advancement of the certified level, whether a full-block success or the partial progress, is justified by the exact projected-gradient certificate in Lemma~\ref{lem:pg-certificate}. Thus the certified state always satisfies
$
\|y^s-x^*\|^2\le R_s^2.
$
Under the stated availability condition, an unfinished level cannot exhaust all block-admissible oracles on the above high-probability event. Thus the certified level eventually reaches $K$, where
$
R_K^2\le\epsilon.
$

Finally, each certification round queries the trusted certification module at most $\mathcal{R}$ times. Multiplying $\mathcal{R}$ by $K + \ell$ yields the stated trusted-gradient complexity.
\end{proof}

\paragraph{Savings in trusted-gradient evaluations.} It is useful to compare the above complexity with directly using the trusted first-order module for optimization. If we were to use this trusted module alone, then an accelerated projected-gradient method for the \(\mu\)-strongly convex and \(L\)-smooth problem would require at least
$$ O\!\left(\sqrt{\tfrac{L}{\mu}}\log\tfrac{R_0^2}{\epsilon}\right) $$
trusted-gradient evaluations to reach \(\epsilon\)-accuracy \citep{nemirovskij1983problem}. In contrast, Proposition~\ref{prop:budgeted-correctness} shows that our framework requires only
$$ O\!\left[ \left( \ell+\log\tfrac{R_0^2}{\epsilon} \right) \log\!\left( \tfrac{\ell+\log(R_0^2/\epsilon)}{\delta} \right) \right] $$
trusted-gradient evaluations. Thus, when the number of candidate oracles \(\ell\) is moderate, the trusted gradient module is called only \(\widetilde O(\log(1/\epsilon))\) times. Up to logarithmic confidence factors, our certification strategy saves a factor of \(\sqrt{L/\mu}\). This saving can be significant for important ill-conditioned stochastic optimization problems, where $L/\mu$ can be large. Such ill-conditioning arises frequently in statistical learning problems, including least-squares, ridge, and logistic regression \citep{bottou2018optimization, jain2018accelerating}, as well as large-scale inverse problems such as tomography \citep{chen2018online, ehrhardt2025guide}.


\subsection{Recursive regularization with SGE: an efficient PG subroutine}
\label{subsec:rr-sge}
To instantiate the projected-gradient reduction subroutine in Definition~\ref{def:pg-reduction-routine}, we build on the recursive-regularization (RR) framework of \citet{allen2018gradient,foster2019complexity}. RR solves a sequence of increasingly regularized subproblems and converts sufficiently accurate solutions of these subproblems into a small projected-gradient mapping for the original objective. Our setting requires two modifications to the standard RR analysis: the inner solver is the biased SGE method with convergence guarantees in Theorem~\ref{thm:sge}, and the stochastic oracle satisfies a state-dependent rather than uniformly bounded variance condition. 

We first state the RR algorithm with the SGE subroutine and its main convergence guarantees, and then provide the detailed proof. Similar to Subsection~\ref{subsec:sge}, in this subsection, we consider a fixed stochastic oracle $\lambda$ and use the simplified notation $\cG$ to denote the selected stochastic first-order oracle and $g(x) = \bbe[\cG(x, \xi)]$, $\cL$, $\sigma_*^2$, and $\cB$ to denote its variance and bias parameters. 
We assume in this subsection $x^*$ is an interior point of the feasible region $X$, thus $\nabla f(x^*)=0$. 

\begin{algorithm}[h]
\caption{RR+SGE for projected gradient reduction}
\label{alg:rr-sge}
\begin{algorithmic}[1]\small
\STATE \textbf{Input:} closed convex set $X$, initial point $\widehat x_0$,
oracle $\SO_\lambda$, and
$S=1+\lceil\log_2(L/\mu)\rceil$; set $\phi_0=f$.
\FOR{$j=1,\ldots,S$}
\STATE Run SGE (Algorithm~\ref{alg:sge}) on $\phi_{j-1}$ over $X$, initialized
at $\widehat x_{j-1}$; denote
the output by $\widehat x_j$, which satisfies $ \E[
       \phi_{j-1}(\widehat x_j)-\phi_{j-1}^*
    ]
    \le\delta_j$.
\STATE Set
$
\rho_j=2^{j-1}\mu$, and $
\phi_j(x)
=
\phi_{j-1}(x)
+\frac{\rho_j}{2}\|x-\widehat x_j\|^2.
$
\ENDFOR
\STATE \textbf{Return:} $\widehat x_S$.
\end{algorithmic}
\end{algorithm}


The following theorem characterizes the convergence guarantees of Algorithm~\ref{alg:rr-sge}.
\begin{theorem}[RR+SGE projected-gradient reduction]
\label{thm:rr-sge-subroutine}
Set $
    S=1+\left\lceil\log_2\frac{L}{\mu}\right\rceil.
$ Suppose the initial point $\widehat x_0$ satisfies
$
    \E\|\widehat x_0-x^*\|^2\le R^2.
$
There exists a universal constant $c_0>0$ such that, if
\begin{equation}
    \cB
    \le
    \tfrac{c_0\mu R}{S},
    \label{eq:cert-bias-condition}
\end{equation}
then Algorithm~\ref{alg:rr-sge} returns a point $\widehat y\in X$ satisfying
\begin{equation}
    \E\|\PG_{X}(\widehat y)\|
    \le
    \tfrac{\sqrt{2}\mu R}{16}, ~~\text{and hence}~~\Prob\!\left(
       \|\PG_{X}(\widehat y)\|
       \le \tfrac{\sqrt{2}\mu R}{4}
    \right)
    \ge\tfrac34.
    \label{eq:cert-constant-prob}
\end{equation}
The total number of stochastic-gradient evaluations is at most
\begin{equation}
\mathcal{O}\left(
S\left(
    \sqrt{\tfrac{L}{\mu}}
    +
    \tfrac{\cL}{\mu}
\right)
+
\tfrac{S^3\sigma_*^2}{\mu^2R^2}
\right).
\label{eq:cert-candidate-complexity}
\end{equation}
Consequently, Algorithm~\ref{alg:rr-sge} is a valid projected-gradient reduction subroutine
$
\mathsf{PGReduce}^{\rm RR+SGE}(y,R;\SO_\lambda)
$
with $\cB_\lambda\le
\tfrac{c_0\mu R}{1+\lceil\log_2(L/\mu)\rceil}$ as the bias-admissibility condition $\mathsf{Adm}_{\rm PG}^{\rm RR+SGE}(R,\cB_\lambda)$
and
\[
\widetilde T^{\rm RR+SGE}(R,\lambda)
=
\widetilde O\!\left(T^{\rm SGE}(R,\lambda)\right).
\]
\end{theorem}

\paragraph{Compatibility with the SGE contraction subroutine.}
Without loss of generality, take the universal constant $c_0$ in Theorem~\ref{thm:rr-sge-subroutine} small enough that
$
c_0\le \sqrt{3}/72.
$
Then
\begin{align}\label{relationship_admissiblity}
\mathsf{Adm}_{\rm PG}^{\rm RR+SGE}(R,\cB)
\quad\Longrightarrow\quad
\mathsf{Adm}_{\rm C}^{\rm SGE}(R,\cB).
\end{align}
Moreover, because the radius sequence is nonincreasing, terminal PG-admissibility at radius $R_{t-1}$ implies contraction admissibility at every earlier radius $R_k$, $k\le t-2$. Hence, for the SGE/RR+SGE implementation, a homogeneous block is block-admissible whenever its oracle is PG-admissible at the terminal radius.

\paragraph{Complexity and near-optimality.}
Theorem~\ref{thm:rr-sge-subroutine} shows that converting the distance-contraction guarantee of biased SGE into a projected-gradient guarantee requires only extra logarithmic factors. In particular, by \eqref{eq:stage-complexity}, a single SGE distance-contraction stage has complexity
\[
T^{\rm SGE}(R,\lambda)
=
O\!\left(
\max\left\{
\sqrt{\tfrac{L}{\mu}},
\tfrac{\cL_\lambda}{\mu},
\tfrac{\sigma_{\lambda,*}^2}{\mu^2R^2}
\right\}
\right),
\]
whereas \eqref{eq:cert-candidate-complexity} differs only by powers of
$S=O(\log(L/\mu))$. Thus RR+SGE achieves
$
\widetilde T^{\rm RR+SGE}(R,\lambda)
=
\widetilde O(T^{\rm SGE}(R,\lambda))
$
while producing the stronger, directly certifiable projected-gradient guarantee. The bias requirement also loses only a logarithmic factor: distance contraction with SGE allows $\cB_\lambda=O(\mu R)$, while RR+SGE requires $\cB_\lambda=O(\mu R/S)$.

This complexity is order-optimal up to logarithmic factors. As established in Lemma~\ref{lem:pg-certificate}, any point satisfying
$
\|\PG_{X}(y)\|\le \sqrt{2}\mu R/4
$
also satisfies
$
\|y-x^*\|^2\le R^2/2.
$
Therefore projected-gradient reduction is at least as hard as a constant-factor distance contraction. Known lower bounds for strongly convex stochastic optimization under state-dependent noise imply that the polynomial dependence appearing in $T^{\rm SGE}(R,\lambda)$ cannot be improved \citep{nemirovskij1983problem,woodworth2021even}; see also the near-tight stochastic gradient-minimization bounds of \citet{foster2019complexity}. 

\paragraph{Proof ingredients.} We next work toward proving Theorem~\ref{thm:rr-sge-subroutine}. 
Throughout the proof, define
\[
    x_j^*\in\arg\min_{x\in Y}\phi_j(x),
    \quad
    \phi_j^*=\phi_j(x_j^*),
\]
and let
\[
    \cG_j(x,\xi)
    =
    \cG(x,\xi)
    +
    \sum_{i=1}^j\rho_i(x-\widehat x_i),
    \quad
    g_j(x)
    =
    g(x)
    +
    \sum_{i=1}^j\rho_i(x-\widehat x_i).
\]
Note that the original oracle satisfies
$
    \|g(x)-\nabla f(x)\|_*
    \le \cB,
$
and
\begin{align}
    \E\|\cG(x,\xi)-g(x)\|_*^2
    \le
    \cL(f(x)-f^*)+\sigma_*^2.
    \label{assump:state-dependent}
\end{align}
The first lemma shows that recursive regularization preserves the oracle bias and the state-dependent form of its variance.

\begin{lemma}[State-dependent noise under recursive regularization]
\label{lem:rr-noise}
suppose
$
    \E\bigl[
      \phi_{j-1}(\widehat x_j)-\phi_{j-1}^*
    \bigr]
    \le \delta_j.
$
Then the bias satisfies
\[
    \|g_j(x)-\nabla\phi_j(x)\|_*
    \le \cB.
\]
Moreover, conditioned on the previous recursive-regularization
iterates, the state-dependent variance satisfies
\[
    \E\|\cG_j(x,\xi)-g_j(x)\|_*^2
    \le
    \cL\bigl(\phi_j(x)-\phi_j^*\bigr)
    +\widehat\sigma_{j,*}^2,
\]
where
$
    \widehat\sigma_{j,*}^2
    :=
    \sigma_*^2
    +
    \cL
    \sum_{i=1}^j
    \tfrac{\rho_i}{2}
    \|x_{i-1}^*-\widehat x_i\|^2.
$
In particular, $
    \E[\widehat\sigma_{j,*}^2]
    \le
    \sigma_*^2
    +
    \cL\sum_{i=1}^j\delta_i.
$
\end{lemma}
We defer the proof of Lemma~\ref{lem:rr-noise} to Appendix~\ref{proof:lem:rr-noise}.

The second ingredient is the deterministic RR reduction: sufficiently accurate solutions of the regularized subproblems imply a small projected-gradient mapping for the original objective.

\begin{proposition}[Recursive-regularization reduction; cf.~\citet{allen2018gradient,foster2019complexity}]
\label{prop:rr-pg}
Let $S
    :=
    1+\left\lceil\log_2\tfrac{L}{\mu}\right\rceil$, 
    $
    \rho_j:=2^{j-1}\mu.
$
For $j=1,\ldots,S$, suppose $\widehat x_j$ is an approximate
minimizer of $\phi_{j-1}$ over $X$ satisfying
$
    \E[
       \phi_{j-1}(\widehat x_j)-\phi_{j-1}^*
    ]
    \le\delta_j.
$
Then, with $
    \gamma=\frac{1}{2L},
$
the final point $\widehat x_S$ satisfies
\begin{align}
    \E\|\PG_{X}(\widehat x_S)\|
    \le
    (2+\sqrt{2})\sum_{j=1}^{S-1}\sqrt{\rho_j\delta_j}
+
6\sqrt{\rho_S\delta_S}.
\label{eq:RR_reduction}
\end{align}
\end{proposition}
The proof of Proposition~\ref{prop:rr-pg} follows from the analysis techniques in \citet{allen2018gradient,foster2019complexity}. For the sake of completeness, we provide a complete proof in Appendix~\ref{proof:prop:rr-pg}.

\paragraph{Proof of Theorem~\ref{thm:rr-sge-subroutine}.}
\begin{proof}
For a sufficiently small universal constant $a>0$, set
\begin{equation}
    \delta_j
    :=
    \tfrac{a\mu^2R^2}{S^2\rho_j},
    \quad
    \rho_j=2^{j-1}\mu,
    \quad
    j=1,\ldots,S.
\label{eq:cert-delta-choice}
\end{equation}
We show that SGE can solve every recursively regularized subproblem
to the accuracy $\delta_j$ with the stated total complexity by inductive arguments.

For the first subproblem, condition on the possibly random warm start $\widehat x_0$, apply the SGE bound in Theorem~\ref{thm:sge}, and then average over $\widehat x_0$. Since
$
\E\|\widehat x_0-x^*\|^2\le R^2,
$
the resulting first-stage bound is identical to the deterministic-warm-start case. Applying \eqref{eq:sge-rate} with mini-batch size
$m=1$, it suffices to use
\begin{equation}
k_1
=
O\left(
S\sqrt{\tfrac{L}{\mu}}
+
\tfrac{S\cL}{\mu}
+
\tfrac{S^2\sigma_*^2}{\mu^2R^2}
\right)
\label{eq:cert-first-budget}
\end{equation}
iterations to make all vanishing terms in the SGE bound smaller than
a fixed fraction of $\delta_1$. Note that the bias floor in Ineq.~\eqref{eq:sge-rate} is
$O(\cB^2/\mu)$, condition \eqref{eq:cert-bias-condition}, with $c_0$ sufficiently small,
ensures
$
    \tfrac{\cB^2}{\mu}
    \lesssim
    \tfrac{a\mu R^2}{S^2}
    =
    \delta_1,
$
and consequently $\bbe[\phi_{0}(\widehat x_{1}) - \phi^*_{0}] \leq \delta_{1}$.

Now consider $j\ge2$. Assume
$\widehat x_{j-1}$ is a $\delta_{j-1}$-approximate minimizer of
$\phi_{j-2}$, i.e., $\bbe[\phi_{j-2}(\widehat x_{j-1}) - \phi^*_{j-2}] \leq \delta_{j-1}$. Then Ineq.~\eqref{eq:rr-new-dist} indicates that the initial distance to the optimal solution
$
\E\|\widehat x_{j-1}-x_{j-1}^*\|^2
\le
\tfrac{\delta_{j-1}}{\rho_{j-1}},
$
and the initial optimality gap
\begin{equation}
\bbe[\phi_{j-1}(\widehat x_{j-1}) -\phi_{j-1}^*]
\le \bbe[\phi_{j-2}(\widehat x_{j-1}) -\phi_{j-2}^*]\leq {\delta_{j-1}},
\label{eq:initial_optimality_gap}
\end{equation}
In the proof of Proposition~\ref{prop:rr-pg}, we show that the $j$-th subproblem $\phi_{j-1}$ is $\rho_j$-strongly convex and
has smoothness at most $2L$. By Lemma~\ref{lem:rr-noise}, the effective residual variance at this
level satisfies
\[
\E[\widehat\sigma_{j-1,*}^2]
\le
\sigma_*^2
+
\cL\sum_{i=1}^{j-1}\delta_i \le
\sigma_*^2
+
\tfrac{2a\cL\mu R^2}{S^2}.
\label{eq:cert-effective-variance}
\]
where the second inequality follows from \eqref{eq:cert-delta-choice}. By the first inequality in \eqref{eq:sge-rate} of Theorem~\ref{thm:sge}, we have
\begin{align*}
\E[\phi_{j-1}(x_{k_j})-\phi_{j-1}^*]
\le &\mathcal{O}\left(\left(\tfrac{L}{k_j^2}+\tfrac{\cL}{mk_j^2}+\tfrac{\cL^2}{\rho_j m^2k_j^2}\right)\tfrac{\delta_{j-1}}{\rho_{j-1}}+ \tfrac{\delta_{j-1}}{k_j^2}+\tfrac{\sigma_*^2}{\rho_j  mk_j} + \tfrac{a \cL\mu R^2}{\rho_jS^2 m k_j}+\tfrac{\cB^2}{\rho_j }\right). 
\end{align*}
Note that $
    \tfrac{
       \delta_{j-1}
    }{\rho_{j-1}\delta_j}
    =
    \tfrac{4}{\rho_j}.
$ Therefore, with mini-batch size
$m=1$, it suffices to use
\begin{equation}
k_j
=
O\left(
\sqrt{\tfrac{L}{\rho_j}}
+
\tfrac{\cL}{\rho_j}
+
\tfrac{S^2\sigma_*^2}{\mu^2R^2} + \tfrac{\cL}{\mu}
\right)
\label{eq:cert-j-th-budget}
\end{equation}
to make all vanishing terms (terms other than $\tfrac{\cB^2}{\rho_j}$) smaller than a fixed fraction of $\delta_j$. Meanwhile, condition \eqref{eq:cert-bias-condition} suggests that, when $c_0$ is sufficiently small, the bias floor satisfies
$
    \tfrac{\cB^2}{\rho_j}
    \lesssim
    \tfrac{a\mu^2R^2}{S^2\rho_j}
    =
    \delta_j.
$
As a consequence, we have $\bbe[\phi_{j-1}(\widehat x_{j}) - \phi^*_{j-1}] \leq \delta_{j}$, and the inductive argument is complete.

Now we check the total sample complexity. Summing up $k_j$ over all
$S=O(\log(L/\mu))$ recursive-regularization levels, we obtain
\begin{align*}
\sum_{j=1}^S k_j = \mathcal{O}\left(
    S\sqrt{\tfrac{L}{\mu}}
    +
    \tfrac{S\cL}{\mu}
+
\tfrac{S^3\sigma_*^2}{\mu^2R^2}
\right).
\end{align*}

It remains to prove the projected-gradient guarantee.
By Ineq.~\eqref{eq:RR_reduction} in Proposition~\ref{prop:rr-pg} and Ineq.~\eqref{eq:cert-delta-choice},
\[
\begin{aligned}
\E\|\PG_{X}(\widehat y)\|
\leq (2+\sqrt{2})\sum_{j=1}^{S-1}\sqrt{\rho_j\delta_j}
+
6\sqrt{\rho_S\delta_S}
=
\big((2+\sqrt{2})S + 4 -\sqrt{2}\big)\tfrac{\sqrt a\mu R}{S}
\leq 6
\sqrt a\mu R.
\end{aligned}
\]
Choosing $\sqrt{a}\le\sqrt{2}/96$ yields $\E\|\PG_{X}(\widehat y)\|
    \le
    \tfrac{\sqrt{2}\mu R}{16}$, which proves the first inequality in Ineq.~\eqref{eq:cert-constant-prob}.
Finally, by using Markov's inequality, we have
$
\Prob\left(
  \|\PG_{X}(\widehat y)\|
  >
  \tfrac{\sqrt{2}\mu R}{4}
\right)
\le
\tfrac{\sqrt{2}\mu R/16}{\sqrt{2}\mu R/4}
=
\tfrac14,
$
which proves the second inequality in Ineq.~\eqref{eq:cert-constant-prob} and completes the proof.
\end{proof}


\subsection{How often should we certify? The training-certification cost tradeoff}\label{subsec:certification-frequency}

This subsection studies the tradeoff for the certification frequency. First, we provide an efficient algorithmic benchmark for comparison.


\paragraph{A PG-admissible switching benchmark.}
The most straightforward algorithmic benchmark is the adaptive switching policy in Algorithm~\ref{alg:known} with known bias from Section~\ref{sec:known_bias}. However, Algorithm~\ref{alg:known} only relies on the contraction-admissible condition, while in this section, we require a slightly stronger PG-admissible condition. Thus, to avoid the compatibility issue between the benchmark and the certification subroutine, we compare with a slightly more restrictive version of the oracle-informed switching policy as in Algorithm~\ref{alg:known}. Define
\begin{equation}
\cA_k^{\rm PG}
:=
\left\{
\lambda:
\mathsf{Adm}_{\rm PG}(R_k,\cB_\lambda)
\text{ holds}
\right\},
\quad
C_k^\star
:=
\min_{\lambda\in\cA_k^{\rm PG}}C_\lambda(k),
\label{eq:pg-switching-benchmark-stage}
\end{equation}
where $C_\lambda(k)$ is defined in \eqref{def_stage_cost}.
Assume $\cA_k^{\rm PG}\neq\varnothing$ for $k=0,\ldots,K-1$. The corresponding single-trial switching cost is
\begin{equation}
C_{\rm sw}^{\rm PG}
:=
\sum_{k=0}^{K-1}C_k^\star.
\label{eq:pg-switching-benchmark}
\end{equation}
Here, the only difference from the known-bias benchmark in Algorithm~\ref{alg:known} is that the admissible set is defined using the projected-gradient condition rather than the distance-contraction condition. For RR+SGE, Theorem~\ref{thm:rr-sge-subroutine} shows that this strengthens the bias requirement only by the logarithmic factor of
$
S=1+\lceil\log_2(L/\mu)\rceil.
$
Additionally, the compatibility observation \eqref{relationship_admissiblity} implies that terminal PG-admissibility automatically makes the corresponding homogeneous block admissible, thus no additional block-validity assumption is needed.

Note that the set $\cA_k^{\rm PG}$ shrinks as the radius decreases. On the high-probability event of Proposition~\ref{prop:budgeted-correctness}, an oracle that belongs to $\cA_k^{\rm PG}$ cannot have been falsely removed at an earlier, coarser level. Thus the current oracle pool should contain an oracle attaining $C_k^\star$. Moreover, by the homogeneous-block construction and the economic-switch rule \eqref{eq:block-stop-switch}, the block oracle $\lambda$ remains a cheapest oracle in the current pool at every stage $k$ of that block, including its terminal stage; therefore
\begin{equation}
C_k(\lambda)
\le
C_k^\star
\quad
\text{for every stage $k$ visited by the block.}
\label{eq:block-vs-pg-switching}
\end{equation}

 We define the first PG-inadmissible level of oracle $\lambda$ by
\begin{equation}
\kappa_\lambda
:=
\min\left\{
k\in\{0,\ldots,K-1\}:
\lambda\notin\cA_k^{\rm PG}
\right\},
\label{eq:first-pg-inadmissible-stage}
\end{equation}
with $\kappa_\lambda=K$ if the oracle remains PG-admissible throughout. We further define a cost overhead quantity
\begin{equation}
C_{\rm bad}^{\rm PG}
:=
\sum_{\lambda:\,\kappa_\lambda<K}
\max_{\substack{
k\in\{\kappa_\lambda,\ldots,K-1\}:\\
C_k(\lambda)\le C_k^\star
}}
\widetilde C_k(\lambda),
\label{eq:bad-oracle-identification-cost}
\end{equation}
where the maximum over an empty set is zero. The quantity $C_{\rm bad}^{\rm PG}$ is independent of $B$ and captures an unavoidable oracle-identification effect. Specifically, when an oracle becomes economically competitive but PG-inadmissible, it will be selected by the algorithm and may result in a failure block, thus being eliminated consequently.
However, failure of $\mathsf{Adm}_{\rm PG}$ is only failure of a sufficient condition: an oracle can already be PG-inadmissible and nevertheless continue to pass the exact certificate for some later stages. Thus, even for the extreme choice $B=0$, which certifies after every stage, the algorithm may continue using such an oracle until a later certification finally fails. The maximum in \eqref{eq:bad-oracle-identification-cost} quantifies this overhead PG cost for each oracle once for the most expensive scale at which it is already PG-inadmissible but still economically competitive with the best PG-admissible oracle. 


Moreover, to quantify the tradeoff for choosing different $B$, we define $M$ as the number of failed blocks and $Q$ as the number of fully successful blocks terminated by the economic-switch condition \eqref{eq:block-stop-switch}. Clearly, we have $M \leq \ell$. The next lemma provides a $B$-independent upper bound for $Q$.

\begin{lemma}[Economic switches]
\label{lem:block-counts}
Under the SGE stage-cost \eqref{eq:stage-complexity}, we have
$
Q\le\binom{\ell}{2}.
$
\end{lemma}
\begin{proof}
From \eqref{eq:stage-complexity}, the stage cost of SGE has the form
$
C_k(\lambda)=\max\{\alpha_\lambda,\beta_\lambda2^k\}.
$
For any two oracles $\lambda$ and $\nu$, their strict ordering can change at most once as $k$ increases. 
Thus one unordered stochastic oracle pair can trigger at most one strict economic switch. Since removed oracles never return to the pool, we have $Q\le\binom{\ell}{2}$.
\end{proof}
It should be noted that Lemma~\ref{lem:block-counts} only provides a conservative upper bound for $Q$. However, an oracle can be selected in two non-consecutive blocks only when, in the later block, the more economically competitive oracles are already eliminated from the pool. This does not happen very often in real implementations. If most oracles are not selected in two non-consecutive blocks, $Q$ should be bounded by $O(\ell)$. Moreover, as discussed in Section~\ref{sec:when-pay} and shown in our numerical experiments, a few oracles may be strictly dominated by others and will never be selected, which further reduces $Q$.

For notational simplicity, we also define a ratio factor $\zeta\ge1$ between the cost of a contraction subroutine (Definition~\ref{def:contraction-routine})  and a PG subroutine (Definition~\ref{def:pg-reduction-routine}). Specifically, 
\begin{equation}
\widetilde T(R, \lambda)
\le
\zeta \cdot T(R, \lambda),
\quad
\forall R,\lambda.
\label{eq:pg-cost-comparable}
\end{equation}
By Theorem~\ref{thm:rr-sge-subroutine}, RR+SGE satisfies this relation with $\zeta = O(S^3)$ containing only additional logarithmic factors in the condition number. In the whitened numerical setting of Section~\ref{sec:numerics}, $L=\mu$ and hence $S=1$, so we have $\widetilde T=O(T)$ and $\zeta=O(1)$. 

Let $\Cost_B$ denote the total cost of our budgeted sparse certification algorithm (Algorithm~\ref{alg:budgeted-certification}), including both training and certification costs. We have the following upper bound guarantees.

\begin{theorem}[Cost of budgeted sparse certification]
\label{thm:budgeted-cost}
Consider Algorithm~\ref{alg:budgeted-certification} on the high-probability event in Proposition~\ref{prop:budgeted-correctness}. For $B>0$, let
$
\bar B:=\min\{B,C_{\rm sw}^{\rm PG}\}.
$
We have
\begin{align}
\Cost_B
\le\;&
\mathcal R
\left[
\zeta C_{\rm sw}^{\rm PG}
+
C_{\rm bad}^{\rm PG}
+
M\bar B
\right]
\nn\\
&+
\mathcal R c_{\rm cert}
\left[
\min\left\{
K,\,
\left\lceil\frac{C_{\rm sw}^{\rm PG}}{B}\right\rceil+Q+1
\right\}
+M
\right].
\label{eq:budgeted-cost-first}
\end{align}
\end{theorem}

\begin{proof}
First, note that the fully successful blocks cover disjoint certified level intervals. By \eqref{eq:block-vs-pg-switching}, every contraction stage in such blocks costs no more than the corresponding stage of our benchmark, while the cost of the PG stage is at most $\zeta$ times its SGE-equivalent stage cost by \eqref{eq:pg-cost-comparable}. Therefore, the single-trial stochastic-training cost of all fully successful blocks is at most $\zeta C_{\rm sw}^{\rm PG}$.\footnote{Here, we do not carefully distinguish the cost of the contraction stage and PG stage, as $\zeta$ is only a logarithmic factor of the condition number.}  Considering the sequential implementation rule with at most $\mathcal{R}$ trials, we have
\begin{align}\label{eq:successful_blocks}
    \Cost(\text{all fully successful blocks}) \leq \mathcal{R}\zeta C_{\rm sw}^{\rm PG}.
\end{align}

For a failed block, by the budget constraint rule in \eqref{eq:block-stop-budget}, the uncertified training cost of all contraction stages is at most $B$, as shown in Ineq.~\eqref{eq:block-budget}.  Meanwhile, in case $B$ is very large or even $\infty$, we can always bound the block's contraction cost by $C_{\rm sw}^{\rm PG}$ due to \eqref{eq:block-vs-pg-switching}. Therefore, the total cost of the contraction stages in all failed blocks is bounded by
\begin{align}\label{eq:failed_blocks}
    \Cost(\text{all contraction stages in all failed blocks}) \leq \mathcal{R}M \min\{B,  C_{\rm sw}^{\rm PG}\}= \mathcal{R}M \bar B.
\end{align}

It remains to account for the terminal PG costs of failed blocks, i.e.,
\[
\Cost(\text{PG stages in all failed blocks})
=
\mathcal{R}\sum_{b\ {\rm failed}}
\widetilde C_{k_b}(\lambda_b),
\]
where $\lambda_b$ and $k_b$ are the oracle and terminal stage of failed block $b$. On the high-probability event under consideration, a removed oracle cannot be PG-admissible at its terminal stage: for the SGE/RR+SGE implementation, terminal PG-admissibility makes the whole block block-admissible, and Lemma~\ref{lem:block-validity} together with Proposition~\ref{prop:budgeted-correctness} rules out such a false rejection on this event. Thus $k_b\ge\kappa_{\lambda_b}$. Moreover, because $\lambda_b$ is selected at $k_b$ while the current pool contains an oracle attaining $C_{k_b}^\star$, we have 
$
C_{k_b}(\lambda_b)\le C_{k_b}^\star
$
by \eqref{eq:block-vs-pg-switching}.
Therefore the terminal PG cost of each failed oracle is bounded by its corresponding summand in \eqref{eq:bad-oracle-identification-cost}. Since each oracle can fail at most once, we have
\begin{equation}
\Cost(\text{PG stages in all failed blocks})\le \mathcal{R}C_{\rm bad}^{\rm PG}.
\label{eq:elim-by-cbad}
\end{equation}

Finally, we need to compute the total certification cost. To do so, we need to count the successful certification blocks. If a fully successful block is terminated by the budget condition and its terminal stage is $t-1$, then
\[
\sum_{k=s}^{t-2}A_k(\lambda)+A_{t-1}(\lambda)>B.
\]
The corresponding successful block intervals are disjoint, and their total (single-trial) contraction subroutine cost is at most $C_{\rm sw}^{\rm PG}$ by \eqref{eq:block-vs-pg-switching}. Therefore, there are at most
$
\lceil C_{\rm sw}^{\rm PG}/B\rceil
$
budget-triggered successful blocks. In addition there are $Q$ economic-switch blocks and at most one final block, while the total number of successful blocks is at most $K$. Thus, the total number of successful blocks is bounded by $\min\left\{
K,\,
\left\lceil\tfrac{C_{\rm sw}^{\rm PG}}{B}\right\rceil+Q+1
\right\}$, and 
the total number of certification blocks is at most
$
\min\left\{
K,\,
\left\lceil\tfrac{C_{\rm sw}^{\rm PG}}{B}\right\rceil+Q+1
\right\}
+M,
$
indicating a total certification cost of
\begin{equation}
\Cost(\text{all certification steps})\le \mathcal{R}c_{\rm cert}
\left[
\min\left\{
K,\,
\left\lceil\tfrac{C_{\rm sw}^{\rm PG}}{B}\right\rceil+Q+1
\right\}
+M
\right].
\label{eq:total_certification}
\end{equation}
Combining Ineqs.~\eqref{eq:successful_blocks}, \eqref{eq:failed_blocks},  \eqref{eq:elim-by-cbad}, and \eqref{eq:total_certification} proves Ineq.~\eqref{eq:budgeted-cost-first}. 
\end{proof}

The theorem separates three conceptually different training costs. The term $\zeta C_{\rm sw}^{\rm PG}$ is the oracle-informed switching benchmark, which is tight up to $\zeta$, a logarithmic factor of the condition number, due to the use of the PG-reduction oracle. The quantity $C_{\rm bad}^{\rm PG}$ is a $B$-independent price of identifying oracles whose sufficient PG-admissibility condition has already failed but may still pass some later certification rounds. Finally, $M\bar B$ is the main budget-dependent training cost: a larger $B$ allows more training steps in a block before a failed oracle is detected, thus resulting in larger cost overhead. 


It should be noted that the worst-case guarantee above charges $\mathcal R$ trials to every block for the high-probability amplification, but the actual number of trials can be much smaller due to the use of the sequential strategy. In particular, Lemma~\ref{prop:sequential-block-cert} shows that this is conservative for successful block-admissible oracles: their expected number of block trials and trusted-gradient evaluations is at most $4/3$, independent of $\mathcal R$. Only an oracle that fails all trials and is permanently removed necessarily pays the full $\mathcal R$-trial price. 

\paragraph{The U-shaped dependence on $B$.} 
When $B=0$, every block is terminated immediately, so each
block consists of one PG reduction followed by certification. The benefit is that a bad oracle is discovered immediately, so there is no training overhead $M\bar B$. However, the certification module is queried at every successful stage, giving a certification cost of order $K c_{\rm cert}$. 

At the other extreme, $B=\infty$ eliminates budget-triggered checkpoints. Certification occurs only when the cost model predicts an economic switch or when the final target level is reached. If one inexpensive oracle remains the predicted cost minimizer throughout, there may be no intermediate certificate. However, if that oracle becomes too biased at a fine scale, the method can then spend a substantial fraction of the full optimization horizon before the final certificate detects the failure. 

The following corollary isolates this tradeoff between training and certification costs introduced by the budget $B$.

\begin{corollary}[U-shaped certification tradeoff]
\label{cor:budgeted-U}
For $0<B\le C_{\rm sw}^{\rm PG}$, define
\begin{equation}
C_{\rm B-independent}
:=
\mathcal R
\left[
\zeta C_{\rm sw}^{\rm PG}
+
C_{\rm bad}^{\rm PG}
\right]
+
\mathcal R c_{\rm cert}
\left[
\binom{\ell}{2}+\ell+2
\right].
\label{eq:budgeted-fixed-cost}
\end{equation}
Then we have
\begin{equation}
\Cost_B
\le
C_{\rm B-independent}
+
\mathcal R
\left[
\ell B
+
\tfrac{c_{\rm cert}C_{\rm sw}^{\rm PG}}{B}
\right].
\label{eq:budgeted-clean-U}
\end{equation}
The $B$-dependent upper bound is minimized at
$
B^\star
=
\sqrt{
\tfrac{c_{\rm cert}C_{\rm sw}^{\rm PG}}{\ell}
}$,
whenever $B^\star\le C_{\rm sw}^{\rm PG}$, with 
$
\Cost_{B^*}
\le
C_{\rm B-independent}
+2\mathcal R
\sqrt{
\ell c_{\rm cert}C_{\rm sw}^{\rm PG}
}$.
\end{corollary}
\begin{proof}
For $B\le C_{\rm sw}^{\rm PG}$, we have $\bar B=B$ and
$
\lceil C_{\rm sw}^{\rm PG}/B\rceil
\le C_{\rm sw}^{\rm PG}/B+1.
$
Also recall the upper bound $M\leq \ell$ and $Q\leq \binom{\ell}{2}$ in Lemma~\ref{lem:block-counts}. 
 Applying these inequalities in \eqref{eq:budgeted-cost-first} yields \eqref{eq:budgeted-clean-U}. 
 By minimizing the two $B$-dependent terms, we can show $B^*$ and $\Cost_{B^*}$.
\end{proof}

Corollary~\ref{cor:budgeted-U} makes the certification-frequency tradeoff explicit. In Ineq.~\eqref{eq:budgeted-clean-U}, the term $\ell B$ increases with the certification budget because an unsuccessful oracle may consume more uncertified training before being rejected, whereas $c_{\rm cert}C_{\rm sw}^{\rm PG}/B$ decreases because larger blocks require fewer trusted certifications. Balancing these two terms yields a theoretically efficient choice of $B^*$. In real implementations, the quantities $c_{\rm cert}$ and $\ell$ are given explicitly, while the quantity $C_{\rm sw}^{\rm PG}$ depends on the desired accuracy and can be hard to estimate. However, since $C_{\rm sw}^{\rm PG}$ is a total-cost benchmark, we can heuristically replace its position in $B^*$ with a fraction of the total available budget. We provide clear evidence for this U-shaped certification tradeoff in our numerical experiments in Subsection~\ref{subsec:numerics_budgeted_certification}.


\section{Numerical experiments}\label{sec:numerics}
In this section, we use simulated and LLM-based experiments to illustrate three key implications of our framework: (i) the cost-minimizing oracle can change with the optimization scale, (ii) stage-dependent switching can be substantially cheaper than committing to a single oracle, and (iii) when oracle biases are unknown, there is a clean tradeoff between frequent certification and delayed detection of biased sources. 

\subsection{Simulated experiments - regression with noisy labels}
We first corroborate the switching mechanism in a strongly convex linear-regression problem with two unbiased stochastic oracles. We use $d=10$ Rademacher features $a\in\{-1,+1\}^d$ and responses $y=a^\top x^*+\varepsilon_\lambda$, where $\|x^*\|=1$ and the initial distance is $R_0=4$. Hence
\[
f(x)-f(x^*)=\tfrac12\|x-x^*\|^2,\quad L=\mu=1.
\]
Oracle 1 is inexpensive but has larger residual noise, whereas oracle 2 is more expensive and more accurate. Both oracles are unbiased in this experiment. We use mini-batches of size $100$, halve the squared radius after each successful stage, and average 25 independent repetitions.

The resulting behavior is shown in Figure~\ref{fig:synthetic-two-oracle}. At coarse accuracy, Algorithm~\ref{alg:known} selects SO1 because its lower price more than offsets its larger variance. As the radius decreases, the residual-variance term becomes dominant and the method switches to SO2. At the accuracy around $10^{-3}$, the adaptive policy costs about $2.7\times10^4$, compared with about $7.0\times10^4$ for always using oracle 2 and $8.5\times10^4$ for always using oracle 1. Thus, even in this simple setting, the most cost-efficient oracle depends on the current optimization scale, and adaptively switching between oracles can lead to substantial cost savings.

\begin{figure}[h]
\centering
\includegraphics[width=0.55\linewidth]{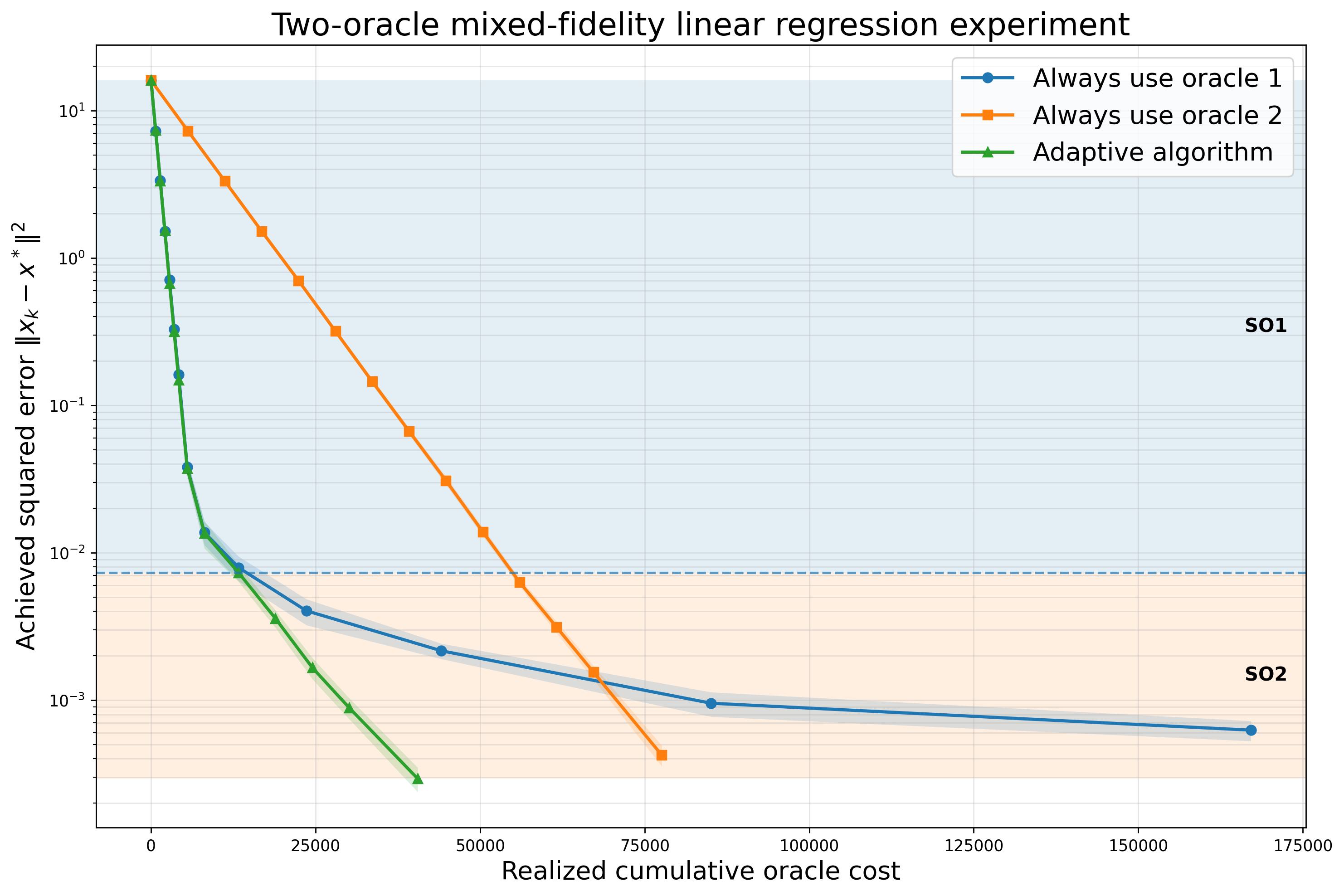}
\caption{Controlled two-oracle experiment. Curves show achieved squared error versus realized cumulative oracle cost; shaded intervals around the curves are pointwise 95\% confidence intervals. Filled regions report the selected oracles in different phases.}
\label{fig:synthetic-two-oracle}
\end{figure}

\subsection{LLM-completion experiments - known bias}\label{subsec_known_bias}
We next consider the \texttt{Aspiring Minds} data described in the introduction and the data completion objective defined in Subsection~\ref{subsec:application_missing_data}. The clean data defines the target linear-regression objective, while each model--prompt completion pipeline defines a stochastic oracle. We predict \texttt{ComputerProgramming} from 22 numerical applicant covariates. After adding an intercept, we whiten the design using the clean least-squares Hessian, which preserves predictions and objective values while giving
\[
\Sigma_{\rm clean}=I,\quad L=\mu=1,\quad
f_{\rm clean}(x)-f_{\rm clean}(x^*)=\tfrac12\|x-x^*\|^2.
\]

We use the clean Hessian for whitening only as an experimental normalization, so that distances, curvature, and oracle characteristics can be compared on a common scale across pipelines. The clean data and clean objective are not included among the candidate stochastic oracles and are never selected by the adaptive policy. In a deployment in which the original Hessian is unavailable, one could instead work in the original coordinates or estimate a preconditioner from available data.

The nine paid oracles correspond to four GPT-5.4 Nano prompt variants, GPT-5.6 Luna, Claude Haiku 4.5, Claude Sonnet 5, and two averaged ensembles. Table~\ref{tab:llm-motivation} reports the 10\% diagnostics. For the low-rank variants, ``LR-7'' and ``LR-10'' mean that the prompt receives 7 and 10 features as input, respectively. The exact prompt templates used in the completion experiments are listed in Appendix~\ref{app:llm-prompts}.

\paragraph{Cost interpretation.} 
For completed dataset $\lambda$, a sampled-row gradient is
\[
\cG_\lambda(x;j)=\widetilde \theta_{\lambda,j}(\widetilde \theta_{\lambda,j}^\top x-y_j).
\]
For ensemble oracles, we average component gradients on the same sampled row and sum their component costs. To compare optimization trajectories on a common scale, we divide the one-time full-dataset completion price by the number of rows to obtain an equivalent per-observation 
query price; a mini-batch of size $m$ is charged $m$ times this price. This is only an accounting device for comparing information costs.

Our goal in this experiment is to place heterogeneous information sources on a common cost scale, rather than to reproduce the billing mechanics of a particular deployment. The completion pipelines incur an upfront cost for producing datasets of different fidelities, whereas the optimization framework is expressed in terms of oracle-query costs. We therefore amortize each completion cost across observations and use the resulting per-row price to compare the information consumed along different optimization trajectories. In applications where acquiring an oracle instead incurs a fixed setup cost, or where completed data can be reused at zero marginal cost, the appropriate cost model would additionally include such activation or reuse costs; incorporating these costs is complementary to the fidelity-switching principle studied here.

As discussed in Subsection~\ref{subsec:application_missing_data}, the bias of the constructed operator admits the exact decomposition
\[
g_\lambda(x)-\nabla f_{\rm clean}(x)
=b_\lambda+(\Sigma_\lambda-\Sigma_{\rm clean})(x-x^*).
\]
We report $\widehat\cB_\lambda=\|b_\lambda\|=\|g_\lambda(x^*)\|$ and $\beta_\lambda=\|\Sigma_\lambda-\Sigma_{\rm clean}\|_2$; the stronger radius-$R$ bias bound is $\widehat\cB_\lambda+\beta_\lambda R$. The residual variance and state-dependent coefficient are computed directly from the completed dataset:
$
\sigma_{\lambda,*}^2
=\frac1n\sum_{j=1}^n
\|\cG_\lambda(x^*;j)-g_\lambda(x^*)\|^2,$ $
\cL_\lambda 
=4\lambda_{\max}\!\left(\frac1n\sum_{j=1}^nM_{\lambda,j}^\top M_{\lambda,j}\right),
$
where $M_{\lambda,j}=\widetilde \theta_{\lambda,j}\widetilde \theta_{\lambda,j}^\top-\Sigma_\lambda$. 


\paragraph{Practical admissibility rule.}
The theoretical admissibility condition for Algorithm~\ref{alg:known} in Section~3 is a sufficient worst-case condition and has conservative constants. In the LLM experiments, our goal is to examine the stage-dependent switching mechanism rather than to reproduce these worst-case constants. We therefore use the practical rule
\[
\widehat{B}_\lambda \le 0.5\mu R,
\]
where
$
\widehat{B}_\lambda
=
\|b_\lambda\|
=
\|g_\lambda(x^\ast)\|.
$
For the completion model, the full radius-dependent bias bound derived in Section~2.2 is $
\widehat{B}_\lambda + \beta_\lambda R;
$
thus, the rule above should be interpreted as a practical proxy for oracle fidelity rather than as an exact implementation of the sufficient theoretical condition. This choice affects the empirical switching thresholds but not any of our theoretical guarantees. In all reported experiments, we use mini-batches of size~20, at most 20 restart stages, and 50 independent repetitions.


\paragraph{Experiment results.}
Figure~\ref{fig:llm-10} shows the experiment with 10\% randomly masked entries. The adaptive sequence is
\[
\begin{aligned}
\text{Nano LR-short-7 (SO2)} &\longrightarrow \text{Nano LR-10 (SO4)}\\
&\longrightarrow \text{Luna (SO5)} \longrightarrow \text{Avg. Haiku+Sonnet (SO8)}.
\end{aligned}
\]
The adaptive method (Algorithm~\ref{alg:known}) skips several available pipelines, including both stand-alone Claude models at the tightest stage. At the finest displayed accuracy, its cumulative equivalent LLM cost is about $\$72$, compared with about $\$240$ for always using SO8, corresponding to a cost reduction of roughly 70\%. The experiment therefore exhibits the same qualitative mechanism as the simulation: inexpensive sources are preferred early, while additional fidelity becomes worthwhile only at finer optimization scales.

\begin{figure}[h]
\centering
\includegraphics[width=0.55\linewidth]{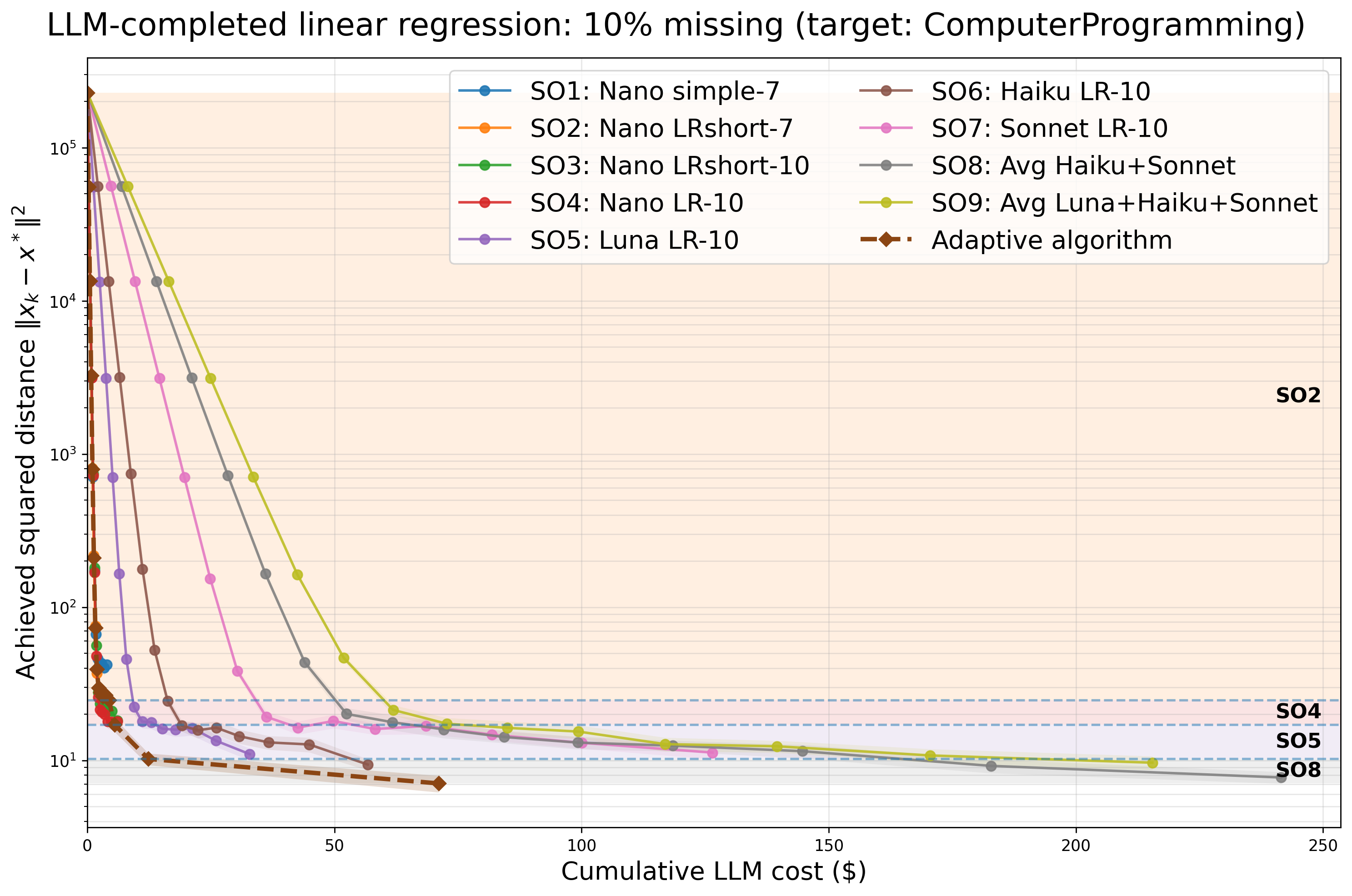}
\caption{LLM-completion experiment for Algorithm~\ref{alg:known} at 10\% missingness. Curves show achieved squared distance versus cumulative equivalent LLM cost, averaged over 50 repetitions; shaded intervals around the curves are pointwise 95\% confidence intervals. Filled regions report the selected oracles in different phases, and dashed lines mark switches.}
\label{fig:llm-10}
\end{figure}

Table~\ref{tab:llm-extra-params} and Figure~\ref{fig:llm-extra-missing} repeats the experiment at 20\% and 40\% missingness. The stage-dependent switching behavior persists, but the selected sequence changes because the completion-induced bias and variance depend on the amount of missing data. At 20\% missingness, the adaptive method uses only inexpensive Nano pipelines over the displayed trajectory, and achieves a cumulative cost comparable to that of the best fixed oracle, SO4.

At 40\% missingness, the adaptive method initially uses Nano LR-short-7 and later switches to Haiku as higher fidelity becomes necessary. At the finest displayed accuracy, its cumulative cost is approximately \$11.5, compared with approximately \$30 for the best fixed oracle, SO6, corresponding to a cost reduction of about 62\%.
These results further illustrate that the preferred oracle fidelity depends jointly on the downstream optimization stage and the quality of the upstream data.

\begin{figure}[h]
\centering
\begin{minipage}[t]{0.47\linewidth}
    \centering
    \includegraphics[width=\linewidth]{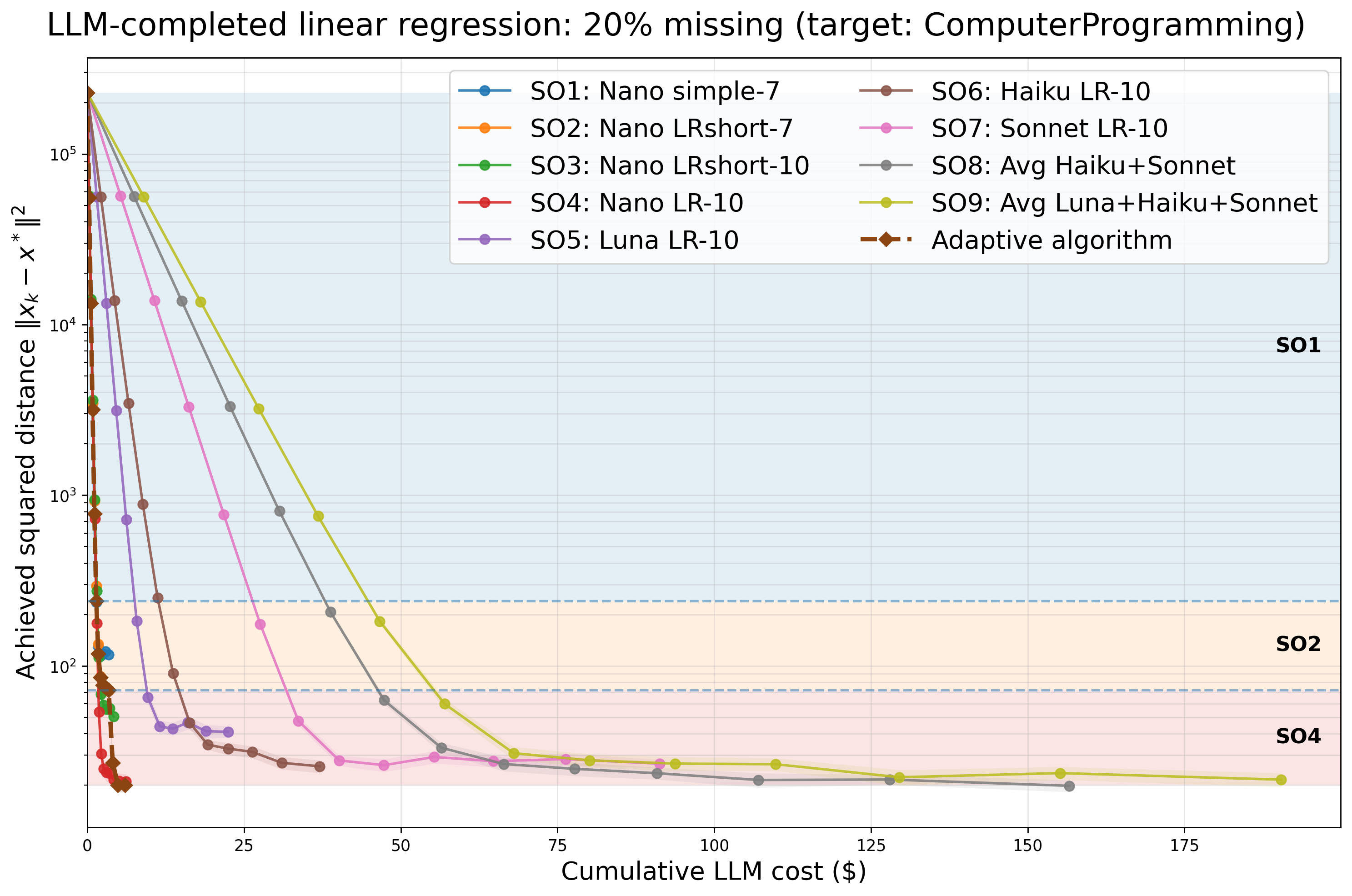}
    {(a) 20\% randomly masked entries.}
\end{minipage}
\hfill
\begin{minipage}[t]{0.47\linewidth}
    \centering
    \includegraphics[width=\linewidth]{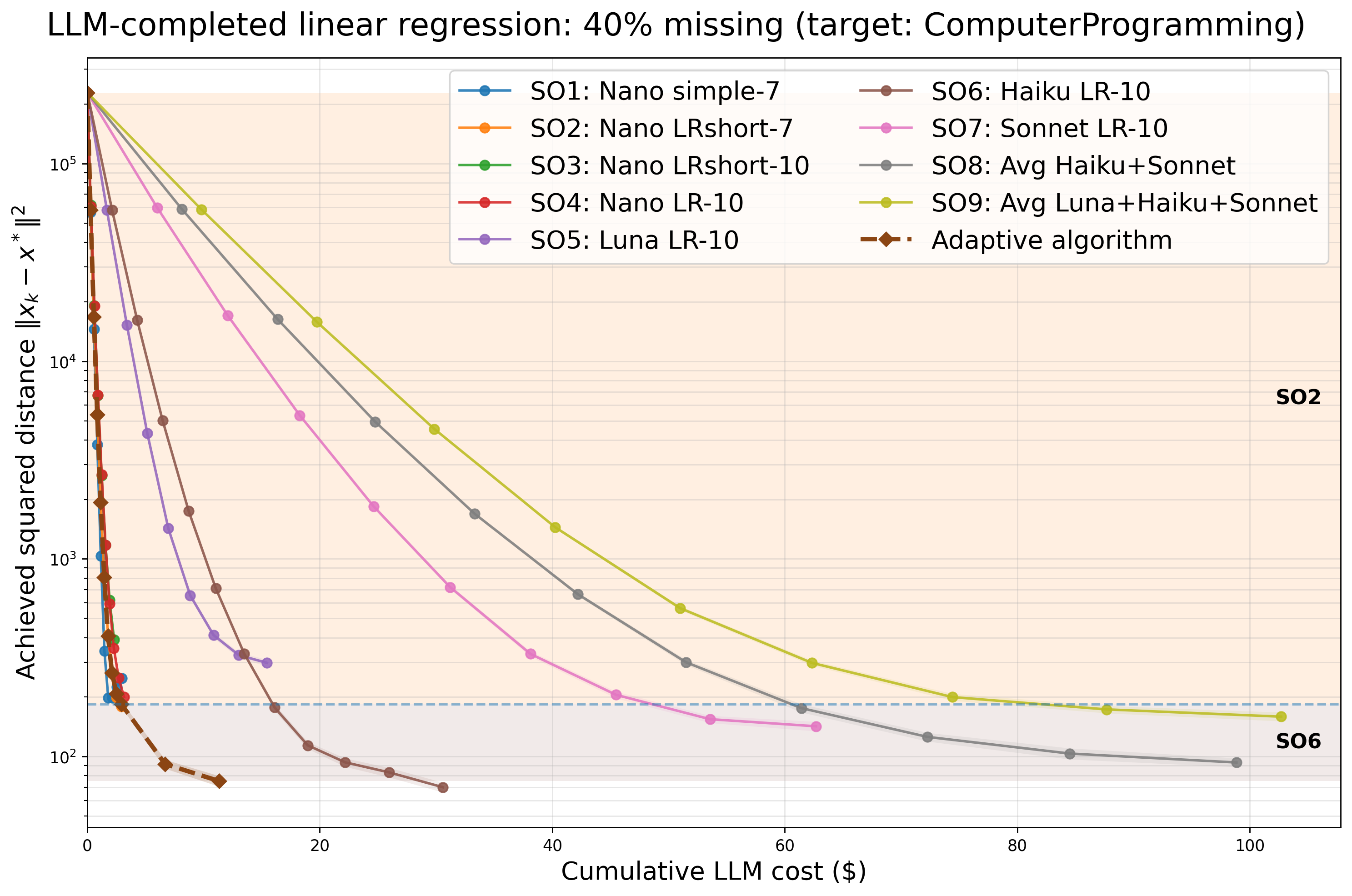}
    {(b) 40\% randomly masked entries.}
\end{minipage}
\caption{LLM-completion experiments for Algorithm~\ref{alg:known} under heavier missingness. Curves show achieved squared distance versus cumulative equivalent LLM cost over 50 repetitions, with pointwise 95\% confidence intervals. Filled regions report the selected oracles in different phases.}
\label{fig:llm-extra-missing}
\end{figure}

\begin{table}[h]
\centering
\caption{Oracle diagnostics for the 20\% and 40\% missingness experiments. Cost is the one-time full-dataset completion price; residual variance is scaled by $10^4$. The \textbf{bolded} rows represent the stochastic oracles selected along the trajectory of our adaptive algorithm.
}
\label{tab:llm-extra-params}
\footnotesize
\setlength{\tabcolsep}{3.0pt}
\begin{tabular}{cllrrrrr}
\toprule
Missing & $\lambda$ & Completion pipeline & Cost (\$) & Obj. & $\sigma_{\lambda,*}^2/10^4$ & $\cL_\lambda/10^3$ & $\widehat\cB_\lambda$\\
\midrule
20\% & - & Clean data (reference) & 0 & 3354 & 8.28 & 6.48 & 0 \\
 & \textbf{1}& \textbf{GPT-5.4 Nano, simple-7} & \textbf{0.180} & \textbf{4208} & \textbf{11.12} & \textbf{6.50} & \textbf{9.63} \\
 & \textbf{2}& \textbf{GPT-5.4 Nano, LR-short-7} & \textbf{0.181} & \textbf{4149} & \textbf{10.47} & \textbf{6.50} & \textbf{7.62} \\
 & 3& GPT-5.4 Nano, LR-short-10 & 0.189 & 4024 & 10.11 & 6.48 & 6.32 \\
 & \textbf{4}& \textbf{GPT-5.4 Nano, LR-10} & \textbf{0.192} & \textbf{4026} & \textbf{9.39} & \textbf{6.48} & \textbf{3.28} \\
 & 5& GPT-5.6 Luna, LR-10 & 0.993 & 4014 & 11.45 & 6.49 & 5.31 \\
 & 6& Claude Haiku 4.5, LR-10 & 1.423 & 3882 & 9.02 & 6.48 & 3.84 \\
 & 7& Claude Sonnet 5, LR-10 & 3.502 & 3879 & 9.05 & 6.48 & 4.18 \\
 & 8& Avg. Haiku+Sonnet, LR-10 & 4.925 & 3847 & 8.93 & 6.48 & 3.30 \\
 & 9& Avg. Luna+Haiku+Sonnet, LR-10 & 5.918 & 3846 & 9.34 & 6.48 & 3.59 \\
\midrule
40\%& - & Clean data (reference) & 0 & 3354 & 8.28 & 6.48 & 0 \\
 & 1& GPT-5.4 Nano, simple-7 & 0.185 & 4950 & 10.60 & 6.52 & 12.21 \\
 & \textbf{2}& \textbf{GPT-5.4 Nano, LR-short-7} & \textbf{0.183} & \textbf{4819} & \textbf{10.13} & \textbf{6.52} & \textbf{10.83} \\
 & 3& GPT-5.4 Nano, LR-short-10 & 0.195 & 5102 & 14.01 & 6.51 & 23.77 \\
 & 4& GPT-5.4 Nano, LR-10 & 0.197 & 4790 & 10.06 & 6.51 & 11.86 \\
 & 5& GPT-5.6 Luna, LR-10 & 1.104 & 5152 & 14.10 & 6.50 & 17.81 \\
 & \textbf{6}& \textbf{Claude Haiku 4.5, LR-10} & \textbf{1.402} & \textbf{4510} & \textbf{9.00} & \textbf{6.49} & \textbf{6.47} \\
 & 7& Claude Sonnet 5, LR-10 & 3.955 & 4701 & 9.27 & 6.49 & 10.78 \\
 & 8& Avg. Haiku+Sonnet, LR-10 & 5.357 & 4529 & 8.87 & 6.49 & 8.14 \\
 & 9& Avg. Luna+Haiku+Sonnet, LR-10 & 6.461 & 4602 & 9.51 & 6.49 & 10.96 \\
\bottomrule
\end{tabular}
\end{table}

\subsection{LLM-completion experiments - budgeted certification with unknown bias}\label{subsec:numerics_budgeted_certification}
We finally test the unknown-bias meta-algorithm in Section~\ref{sec:unknown_bias}. The algorithm does not use the bias diagnostics in its oracle-selection rule. Instead, it trains with one oracle over a homogeneous block and uses the clean gradient as the trusted certification module only at block endpoints. As shown above, whitening guarantees $L=\mu=1$, so the recursive-regularization depth is
\[
S=1+\left\lceil\log_2(L/\mu)\right\rceil=1.
\]
Therefore, the RR+SGE subroutine consists of only one SGE solve on the original objective, and we use $\widetilde T(R,\lambda)=T(R,\lambda)$ in the implementation. We assign each trusted-gradient evaluation the explicit certification price
$
c_{\rm cert}=\$ 1.0,
$
measured on the same scale as the stochastic-oracle training cost. We use $\mathcal R=3$ sequential block trials, and a block stops as soon as one candidate certifies successfully. 
We average 50 independent repetitions for each configuration and report the pointwise 95\% confidence intervals.

\paragraph{Certification-frequency tradeoff.}
Figure~\ref{fig:llm-budgeted-u-shape} varies the uncertified-training budget $B$ and reports the final cost required for Algorithm~\ref{alg:budgeted-certification} to reach the target accuracy. The blue curve represents the total cumulative cost, decomposed into stochastic-oracle training cost and trusted-certification cost. The experiments exhibit the tradeoff predicted by Corollary~\ref{cor:budgeted-U}. When $B$ is too small, the method certifies frequently and therefore pays a large certification cost. As $B$ increases, this certification cost decreases. Beyond an intermediate range, however, the training cost increases because a block can spend more on an oracle before a failed certificate causes that oracle to be rejected. 
\begin{figure}[h]
\centering
\begin{minipage}[t]{0.32\linewidth}
    \centering
    \includegraphics[width=\linewidth]{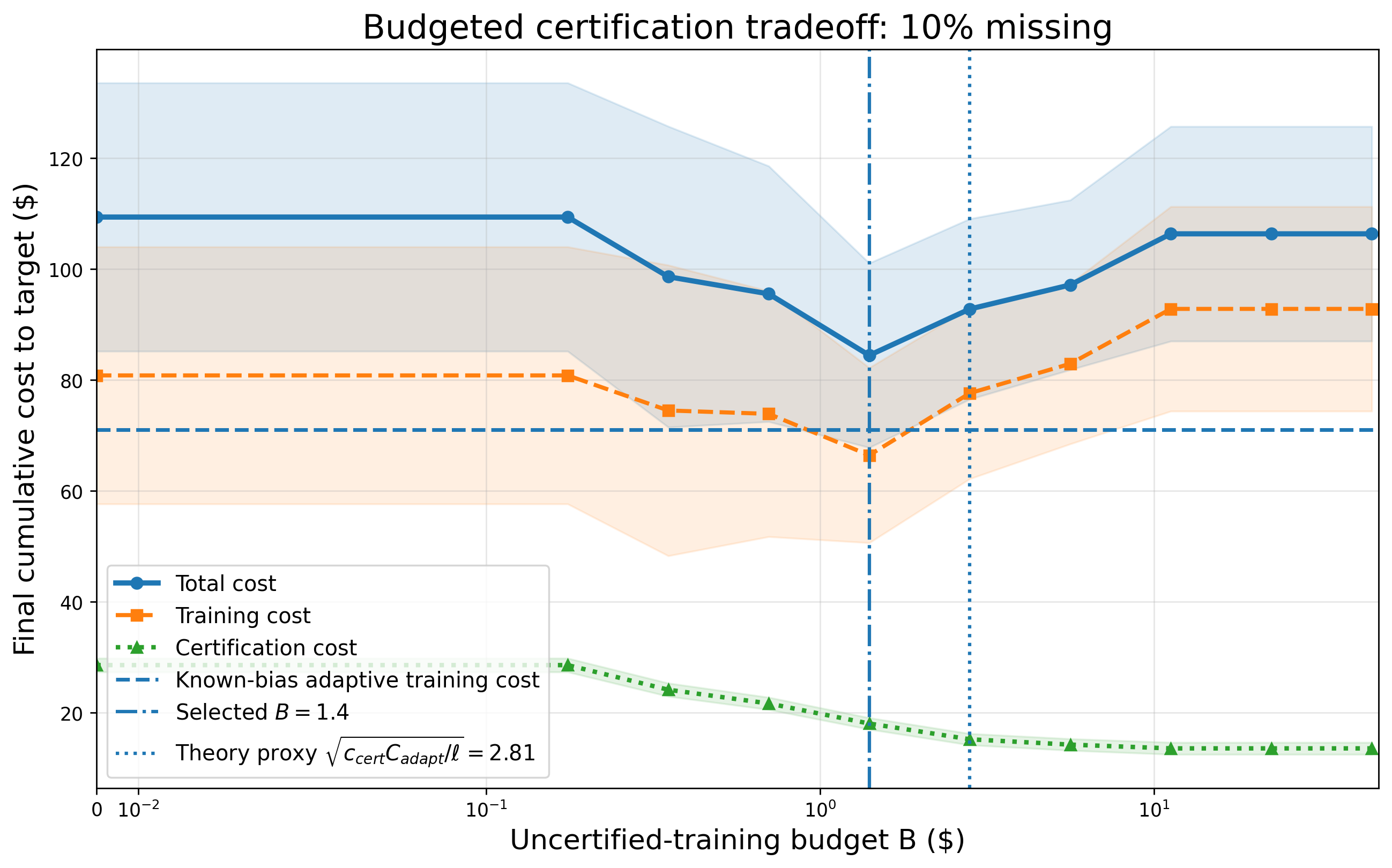}
    {(a) 10\% missingness.}
\end{minipage}
\hfill
\begin{minipage}[t]{0.32\linewidth}
    \centering
    \includegraphics[width=\linewidth]{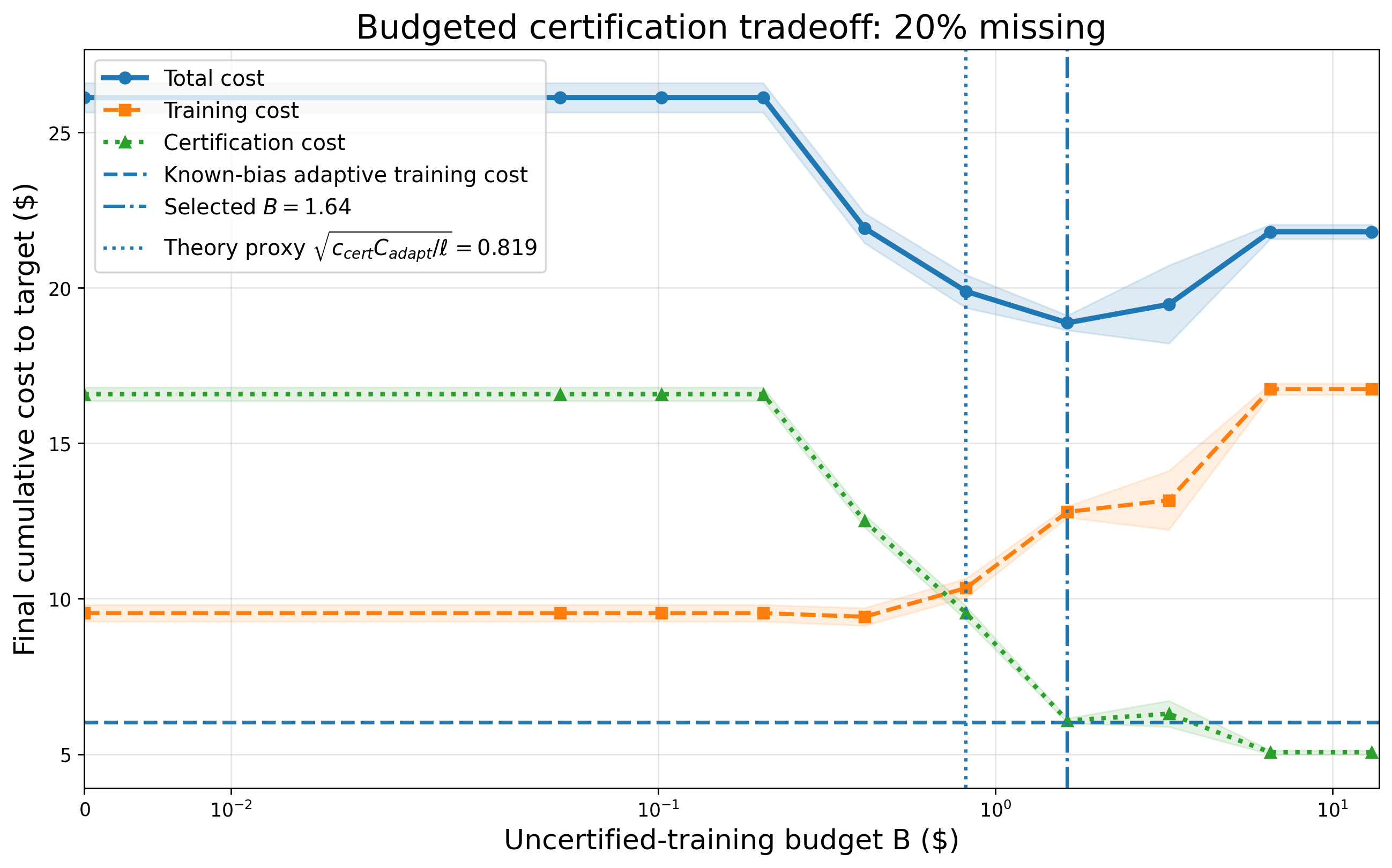}
    {(b) 20\% missingness.}
\end{minipage}
\hfill
\begin{minipage}[t]{0.32\linewidth}
    \centering
    \includegraphics[width=\linewidth]{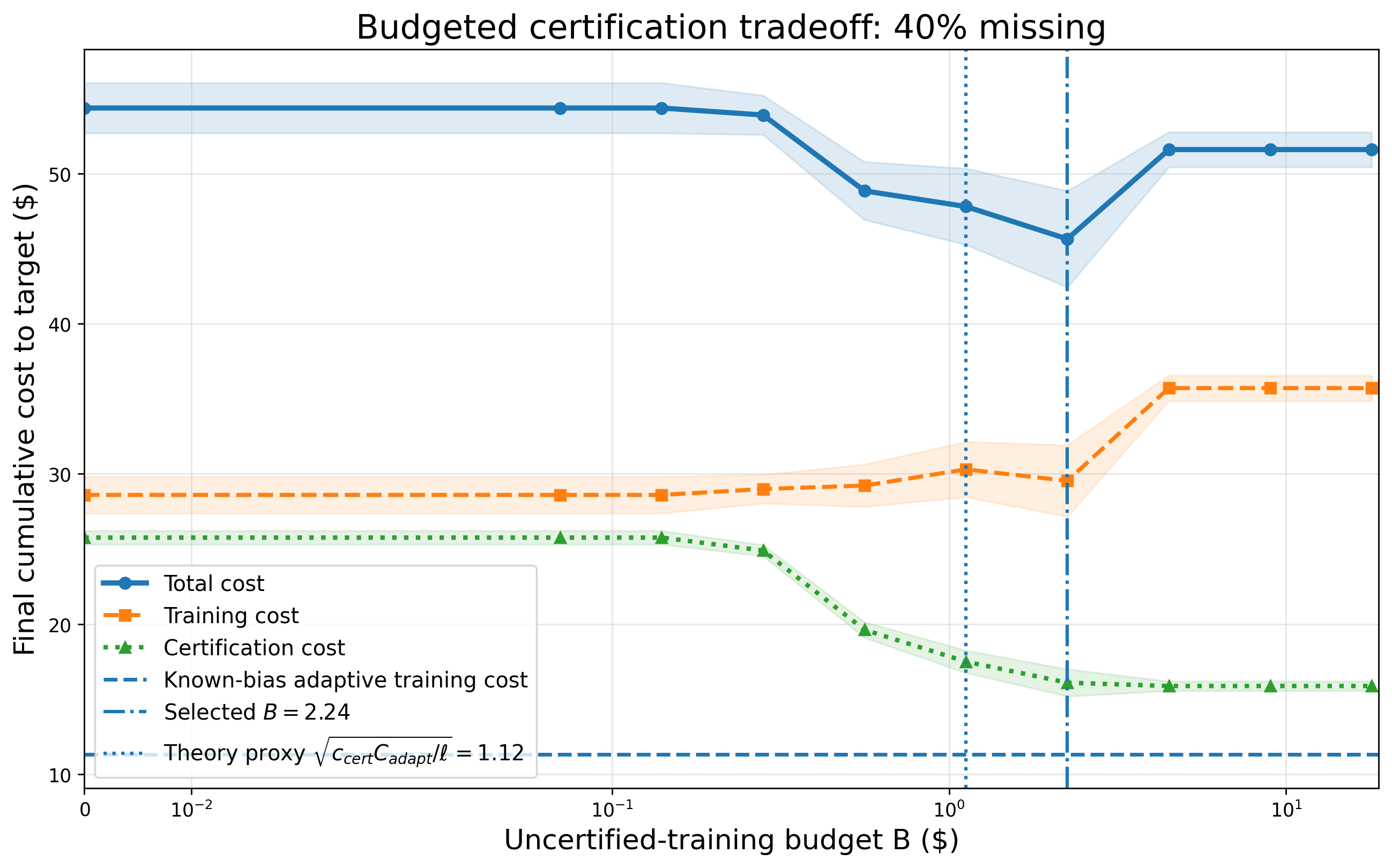}
    {(c) 40\% missingness.}
\end{minipage}
\caption{Certification-frequency tradeoff for Algorithm~\ref{alg:budgeted-certification}. The curves report the final cumulative total, training, and certification costs for different budget $B$, averaged over 50 repetitions, with pointwise 95\% confidence intervals. The horizontal dashed line is the training cost of the known-bias adaptive benchmark (Algorithm~\ref{alg:known}). The dash-dotted vertical line marks the best choice of $B$, and the dotted vertical line reports the theory-inspired proxy $\sqrt{c_{\rm cert}C_{\rm adapt}/\ell}$ in Corollary~\ref{cor:budgeted-U}.}
\label{fig:llm-budgeted-u-shape}
\end{figure}

For 10\%, 20\%, and 40\% missingness, the selected budgets are approximately $B=1.40$, $1.64$, and $2.24$, respectively. The corresponding theory-inspired proxies
$\sqrt{c_{\rm cert}C_{\rm adapt}/\ell}$ are approximately $2.81$, $0.82$, and $1.12$, where $C_{\rm adapt}$ is the realized training cost of the known-bias adaptive benchmark. Although this proxy is not intended to predict the exact empirical minimizer of $B$, it identifies the correct scale: in each case the selected budget is within a factor of two of the proxy, demonstrating the practical value of the theoretical backbone in Corollary~\ref{cor:budgeted-U}. 

Figure~\ref{fig:llm-budgeted-cost-decomposition} reports how training and certification costs accumulate across certified levels. 
The decomposition makes the price of unknown bias explicit. At coarse scales, one certificate validates several inexpensive training stages and the certification cost grows slowly. Near the finest scales, unsuccessful trials and higher-fidelity training cause both components to increase more sharply. 

\begin{figure}[h]
\centering
\begin{minipage}[t]{0.32\linewidth}
    \centering
    \includegraphics[width=\linewidth]{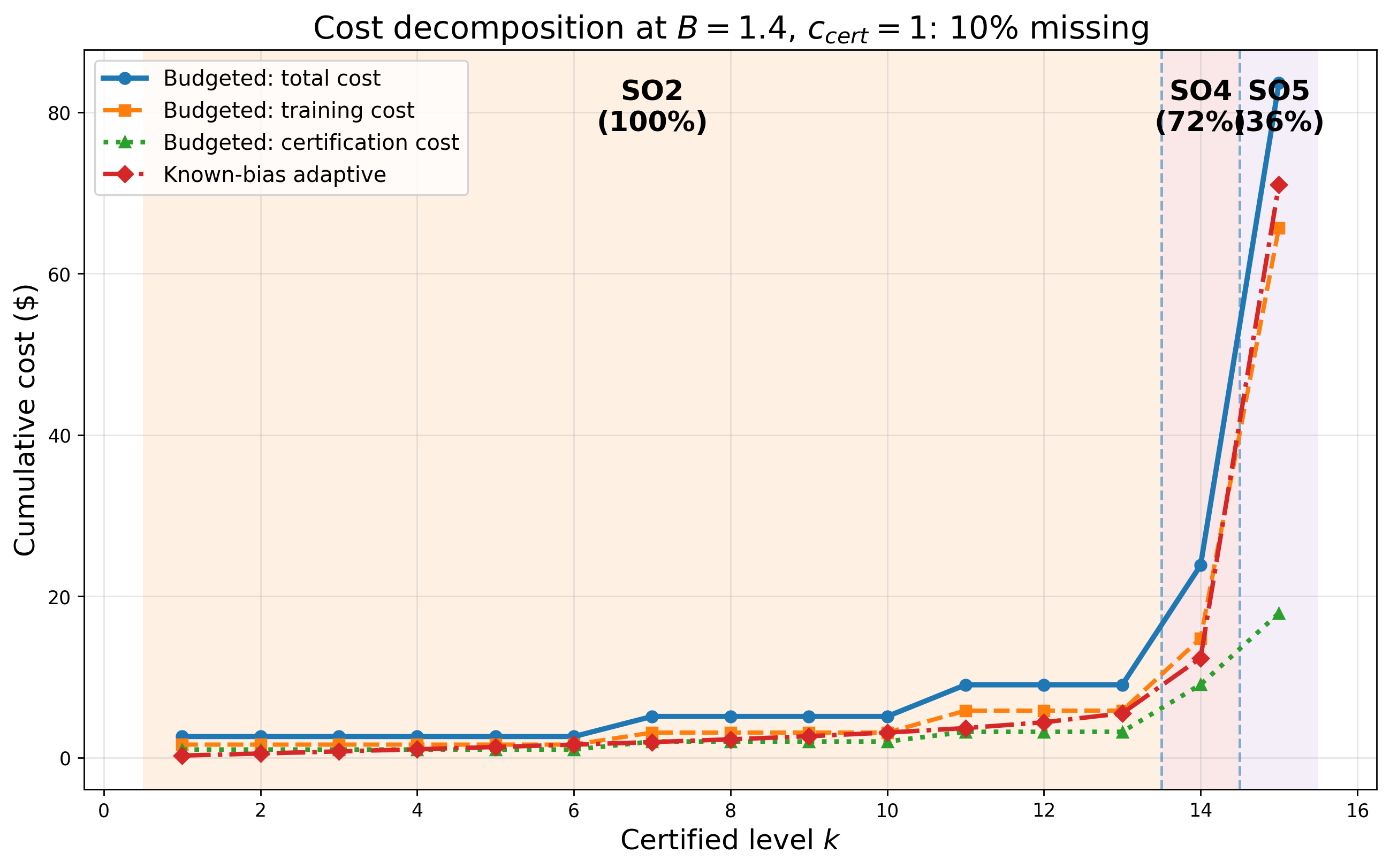}
    {(a) 10\% missingness, $B=1.40$.}
\end{minipage}
\hfill
\begin{minipage}[t]{0.32\linewidth}
    \centering
    \includegraphics[width=\linewidth]{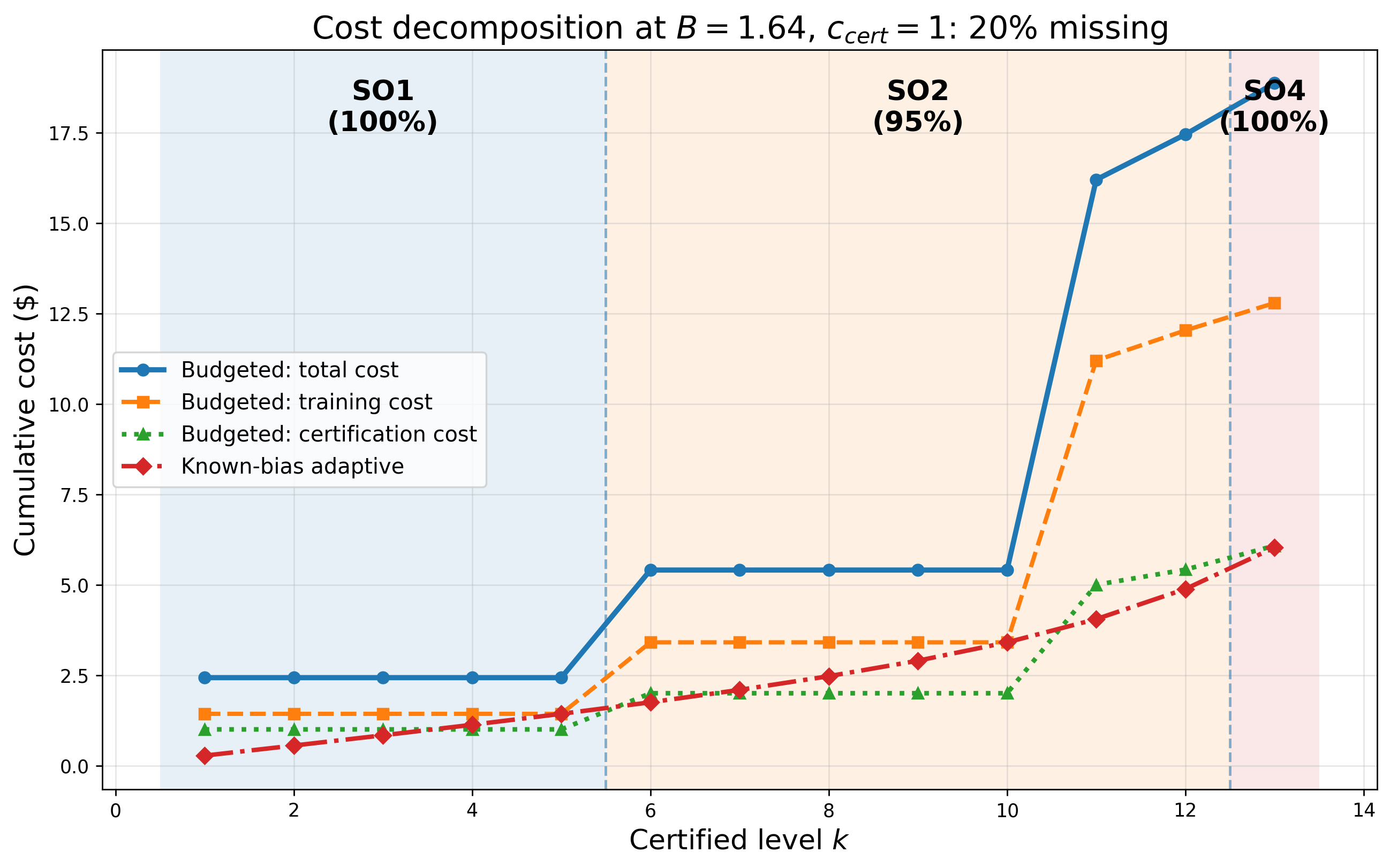}
    {(b) 20\% missingness, $B=1.64$.}
\end{minipage}
\hfill
\begin{minipage}[t]{0.32\linewidth}
    \centering
    \includegraphics[width=\linewidth]{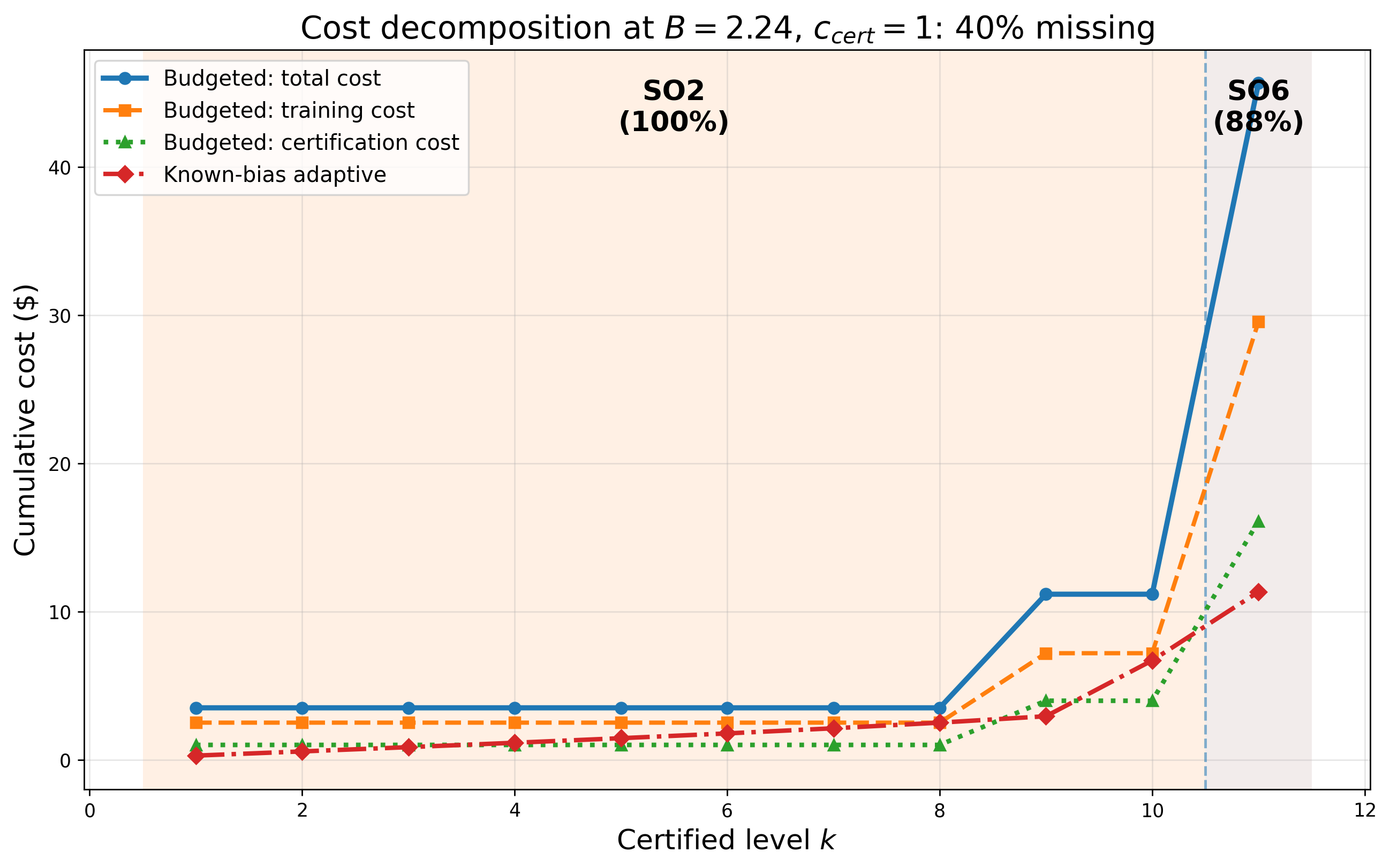}
    {(c) 40\% missingness, $B=2.24$.}
\end{minipage}
\caption{Cumulative cost decomposition for budgeted certification at the selected values of $B$. Blue curves report total cost, orange curves report stochastic-oracle training cost, and green curves report trusted-certification cost; the red curve is the known-bias adaptive benchmark. Filled regions indicate mostly selected oracles in Algorithm~\ref{alg:budgeted-certification} and the percentage of repetitions that select them.}
\label{fig:llm-budgeted-cost-decomposition}
\end{figure}

Figure~\ref{fig:llm-budgeted-vs-known} compares the optimization progress of Algorithm~\ref{alg:budgeted-certification} with Algorithm~\ref{alg:known}. 
As expected, the known-bias policy is more cost-efficient, but the unknown-bias method recovers the same qualitative coarse-to-fine fidelity pattern without access to these bias parameters. At 10\% missingness, the two trajectories remain relatively close until the finest stages, after which the unknown-bias method pays for additional certification and failed attempts. For 20\% and 40\% missingness, the final switch to a sufficiently accurate oracle requires more failed trials and certification cost.

\begin{figure}[h]
\centering
\begin{minipage}[t]{0.32\linewidth}
    \centering
    \includegraphics[width=\linewidth]{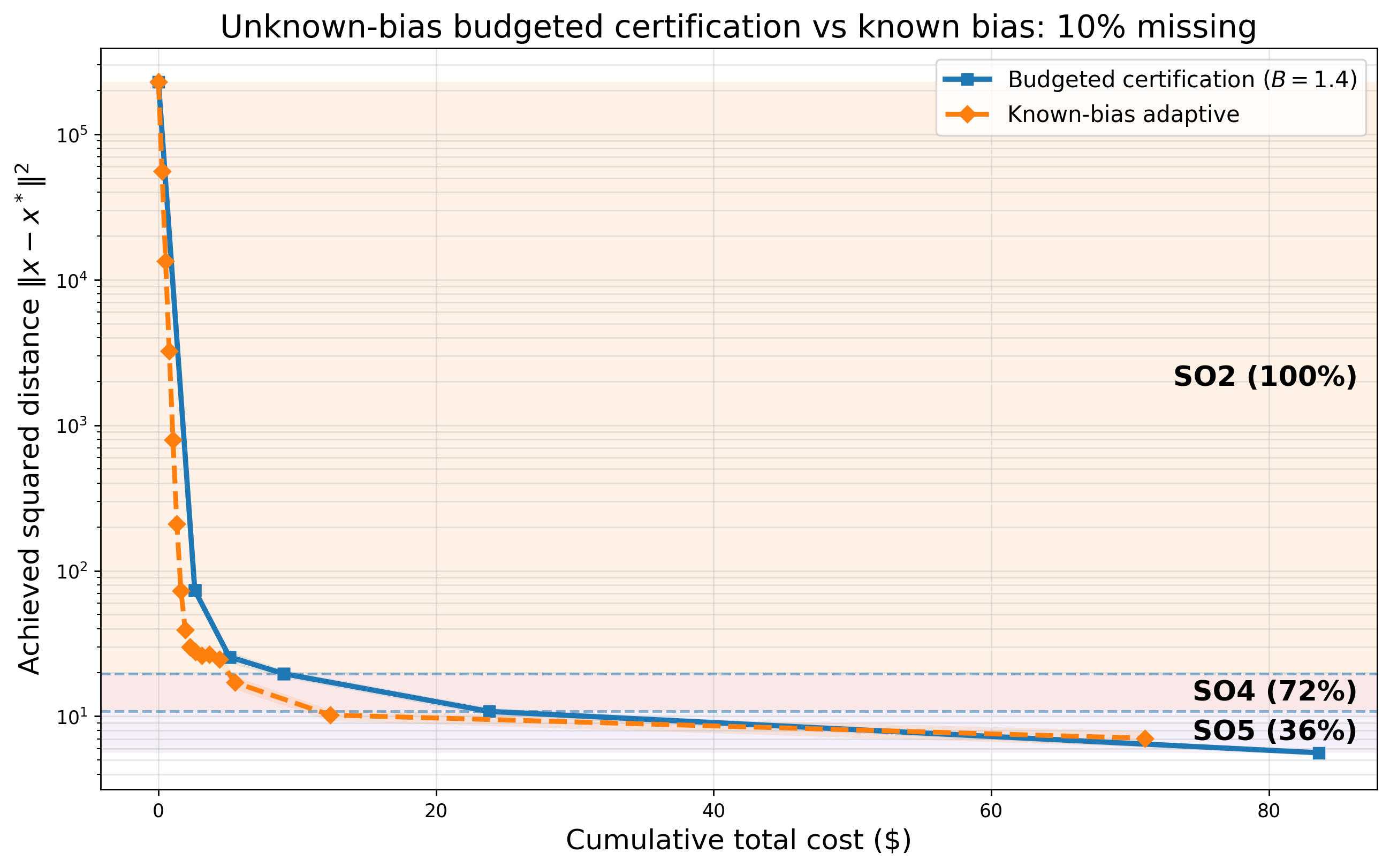}
    {(a) 10\% missingness, $B=1.40$.}
\end{minipage}
\hfill
\begin{minipage}[t]{0.32\linewidth}
    \centering
    \includegraphics[width=\linewidth]{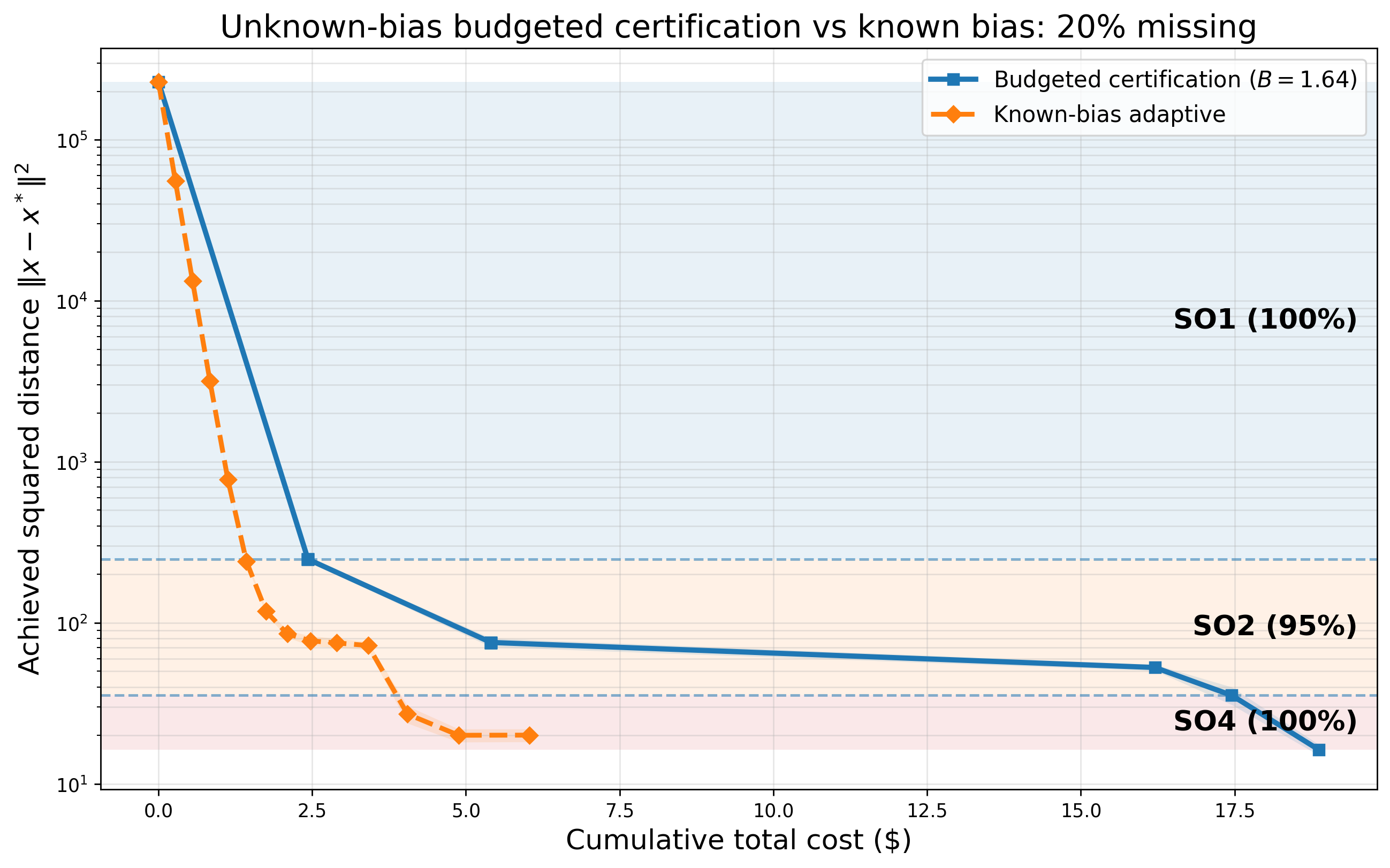}
    {(b) 20\% missingness, $B=1.64$.}
\end{minipage}
\hfill
\begin{minipage}[t]{0.32\linewidth}
    \centering
    \includegraphics[width=\linewidth]{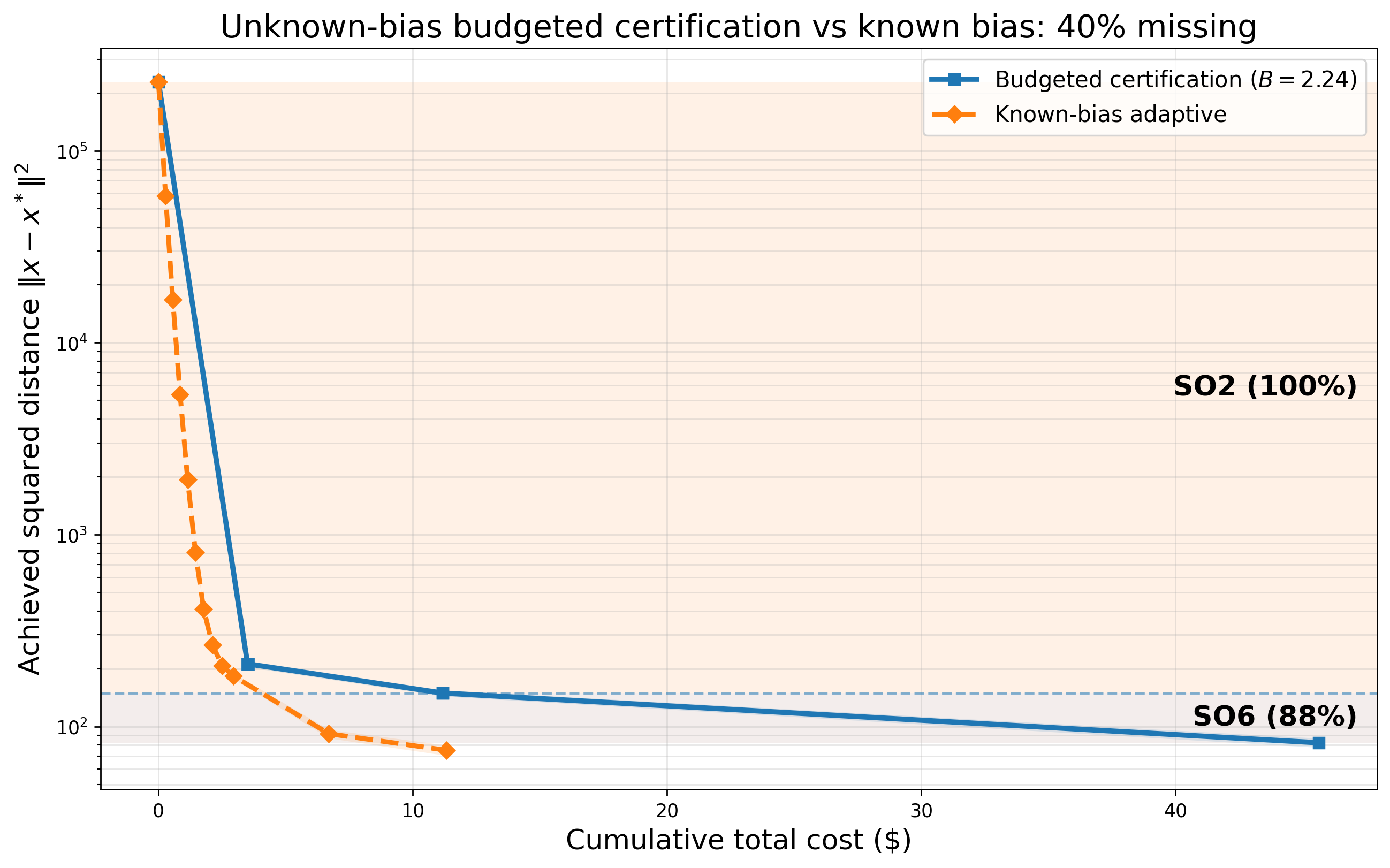}
    {(c) 40\% missingness, $B=2.24$.}
\end{minipage}
\caption{Unknown-bias budgeted certification v.s. the known-bias adaptive policy. Curves report achieved squared distance to $x^*$ versus cumulative total cost, including both stochastic-oracle training and trusted-gradient evaluations. Filled regions indicate the stochastic oracles mostly selected by the budgeted method and the percentage of repetitions that select them.}
\label{fig:llm-budgeted-vs-known}
\end{figure}

Taken together, this set of experiments provides a proof of concept for Section~\ref{sec:unknown_bias}. First, trusted certification removes the need for prior knowledge of the bias parameters and detects when a low-cost oracle can no longer deliver the required progress. Second, the budget $B$ controls how much training cost is incurred between successive certifications. The empirical results reveal a clear tradeoff between training and certification costs, showing that the certification frequency is itself an economically meaningful algorithmic decision.

Lastly, we would like to comment that the regression experiments in Subsections~\ref{subsec_known_bias}--\ref{subsec:numerics_budgeted_certification} are intentionally chosen as a controlled setting in which the clean downstream objective and the mixed-fidelity oracle parameters can be measured directly. Our purpose is not to claim that regression exhausts LLM-assisted decision pipelines, but to test the mechanism predicted by the theory while holding the downstream optimization problem fixed. The same framework applies whenever alternative data-generation procedures induce stochastic first-order information of different cost and fidelity, including downstream estimation, prediction-and-optimization, and decision-support pipelines.

\section{Concluding remarks}\label{sec:conclusion}
The central message is that the optimal use of mixed-fidelity information depends critically on the stage of downstream optimization. A cheap source can be the economically correct choice when the optimizer is far from the solution, even if its fidelity would be unacceptable at the final accuracy. Conversely, high fidelity is worth paying for only when it enables a more cost-efficient algorithmic progress or a lower attainable bias floor. The resulting break-even prices connect the optimization analysis to classical ideas on the value of information:
information should be purchased according to its decision value, not its intrinsic accuracy alone \citep{howard1966information,sims2003rational}.

This perspective also suggests a complementary way to benchmark LLMs. General evaluation efforts such as HELM and Chatbot Arena provide valuable standardized comparisons of model capabilities and human preferences; see, e.g., \citet{bommasani2023helm,chiang2024arena}. For optimization-driven applications, however, a single global rank may not identify the economically best model. A useful task-specific ``oracle card'' could report not only benchmark accuracy and token price, but also downstream quantities such as bias, residual variance, and state-dependent variance on a representative calibration set. For data labeling, these quantities can be estimated using repeated or trusted labels; for data completion, they can be estimated from a held-out fully observed subset. Such reporting would make it possible to translate model improvements directly into stage-dependent willingness-to-pay curves such as \eqref{eq:breakeven-price}.

Although LLM-assisted data curation motivates this work, the framework is broader. The same mixed-fidelity structure arises when stochastic first-order information is produced by human experts of different skill levels, crowd workers, sensors of different precision, approximate simulators, laboratory measurements, or computational solvers with different tolerances. Interesting future directions include learning oracle parameters online rather than from calibration data and allowing correlated errors across different data sources. 

Our analysis can be used in general for any portfolio of stochastic oracles. For LLM-based oracles, we abstract away several deployment-specific features. In particular, oracle acquisition may involve fixed setup costs in addition to marginal query costs; trusted certification may itself be estimated from finite samples; and practical implementations may estimate fidelity parameters only approximately. The empirical cost normalization and admissibility rules used in Section 6 are intended to expose the stage-dependent switching mechanism rather than to model every feature of a particular API deployment. Extending the framework to jointly learn oracle quality, acquisition cost, and certification uncertainty is an interesting direction for future work.

\bibliographystyle{abbrvnat}
\bibliography{main_arXiv}

\begin{thebibliography}{57}
\providecommand{\natexlab}[1]{#1}
\providecommand{\url}[1]{\texttt{#1}}
\expandafter\ifx\csname urlstyle\endcsname\relax
  \providecommand{\doi}[1]{doi: #1}\else
  \providecommand{\doi}{doi: \begingroup \urlstyle{rm}\Url}\fi

\bibitem[Acemoglu et~al.(2022)Acemoglu, Makhdoumi, Malekian, and
  Ozdaglar]{acemoglu2022data}
D.~Acemoglu, A.~Makhdoumi, A.~Malekian, and A.~Ozdaglar.
\newblock Too much data: Prices and inefficiencies in data markets.
\newblock \emph{American Economic Journal: Microeconomics}, 14\penalty0
  (4):\penalty0 218--256, 2022.
\newblock \doi{10.1257/mic.20200200}.

\bibitem[Ajalloeian and Stich(2020)]{ajalloeian2020biased}
A.~Ajalloeian and S.~U. Stich.
\newblock On the convergence of {SGD} with biased gradients.
\newblock \emph{arXiv preprint arXiv:2008.00051}, 2020.

\bibitem[Allen-Zhu(2018)]{allen2018gradient}
Z.~Allen-Zhu.
\newblock How to make the gradients small stochastically: Even faster convex
  and nonconvex sgd.
\newblock \emph{Advances in Neural Information Processing Systems}, 31, 2018.

\bibitem[Angelopoulos et~al.(2023)Angelopoulos, Bates, Jordan, and
  Malik]{angelopoulos2023ppi}
A.~N. Angelopoulos, S.~Bates, M.~I. Jordan, and J.~Malik.
\newblock Prediction-powered inference.
\newblock \emph{Science}, 382\penalty0 (6671):\penalty0 669--674, 2023.

\bibitem[Asi et~al.(2021)Asi, Carmon, Jambulapati, Jin, and
  Sidford]{asi2021stochastic}
H.~Asi, Y.~Carmon, A.~Jambulapati, Y.~Jin, and A.~Sidford.
\newblock Stochastic bias-reduced gradient methods.
\newblock In \emph{Advances in Neural Information Processing Systems},
  volume~34, pages 10810--10822, 2021.

\bibitem[Aybat et~al.(2019)Aybat, Fallah, Gurbuzbalaban, and
  Ozdaglar]{aybat2019universally}
N.~S. Aybat, A.~Fallah, M.~Gurbuzbalaban, and A.~Ozdaglar.
\newblock A universally optimal multistage accelerated stochastic gradient
  method.
\newblock \emph{Advances in neural information processing systems}, 32, 2019.

\bibitem[Beatson and Adams(2019)]{beatson2019efficient}
A.~Beatson and R.~P. Adams.
\newblock Efficient optimization of loops and limits with randomized
  telescoping sums.
\newblock In \emph{Proceedings of the 36th International Conference on Machine
  Learning}, volume~97 of \emph{Proceedings of Machine Learning Research},
  pages 534--543, 2019.

\bibitem[Belloni et~al.(2026)Belloni, Chen, and Wei]{belloni2026online}
A.~Belloni, Y.~Chen, and Y.~Wei.
\newblock Online pandora's box for contextual llm cascading.
\newblock \emph{arXiv preprint arXiv:2606.07392}, 2026.

\bibitem[Blanchet et~al.(2017)Blanchet, Goldfarb, Iyengar, Li, and
  Zhou]{blanchet2017unbiased}
J.~Blanchet, D.~Goldfarb, G.~Iyengar, F.~Li, and C.~Zhou.
\newblock Unbiased simulation for optimizing stochastic function compositions.
\newblock \emph{arXiv preprint arXiv:1711.07564}, 2017.

\bibitem[Blanchet and Glynn(2015)]{blanchet2015unbiased}
J.~H. Blanchet and P.~W. Glynn.
\newblock Unbiased monte carlo for optimization and functions of expectations
  via multi-level randomization.
\newblock In \emph{2015 Winter Simulation Conference}, pages 3656--3667. IEEE,
  2015.

\bibitem[Bommasani et~al.(2023)Bommasani, Liang, and Lee]{bommasani2023helm}
R.~Bommasani, P.~Liang, and T.~Lee.
\newblock Holistic evaluation of language models.
\newblock \emph{Annals of the New York Academy of Sciences}, 1525\penalty0
  (1):\penalty0 140--146, 2023.
\newblock \doi{10.1111/nyas.15007}.

\bibitem[Bottou et~al.(2018)Bottou, Curtis, and
  Nocedal]{bottou2018optimization}
L.~Bottou, F.~E. Curtis, and J.~Nocedal.
\newblock Optimization methods for large-scale machine learning.
\newblock \emph{SIAM Review}, 60\penalty0 (2):\penalty0 223--311, 2018.
\newblock \doi{10.1137/16M1080173}.

\bibitem[Chen et~al.(2018)Chen, Li, and Liu]{chen2018online}
K.~Chen, Q.~Li, and J.-G. Liu.
\newblock Online learning in optical tomography: A stochastic approach.
\newblock \emph{Inverse Problems}, 34\penalty0 (7):\penalty0 075010, 2018.
\newblock \doi{10.1088/1361-6420/aac220}.

\bibitem[Chen et~al.(2024)Chen, Zaharia, and Zou]{chen2024frugalgpt}
L.~Chen, M.~Zaharia, and J.~Zou.
\newblock Frugalgpt: How to use large language models while reducing cost and
  improving performance.
\newblock \emph{Transactions on Machine Learning Research}, 2024.

\bibitem[Chiang et~al.(2024)Chiang, Zheng, Sheng, Angelopoulos, Li, Li, Zhu,
  Zhang, Jordan, Gonzalez, and Stoica]{chiang2024arena}
W.-L. Chiang, L.~Zheng, Y.~Sheng, A.~N. Angelopoulos, T.~Li, D.~Li, B.~Zhu,
  H.~Zhang, M.~I. Jordan, J.~E. Gonzalez, and I.~Stoica.
\newblock Chatbot arena: An open platform for evaluating llms by human
  preference.
\newblock In \emph{Proceedings of the 41st International Conference on Machine
  Learning}, volume 235 of \emph{Proceedings of Machine Learning Research},
  pages 8359--8388, 2024.

\bibitem[Daras et~al.(2025)Daras, Cherapanamjeri, and
  Daskalakis]{daras2025ambient}
G.~Daras, Y.~Cherapanamjeri, and C.~Daskalakis.
\newblock How much is a noisy image worth? data scaling laws for ambient
  diffusion.
\newblock In \emph{International Conference on Learning Representations}, 2025.

\bibitem[Dereich and M{\"u}ller-Gronbach(2019)]{dereich2019general}
S.~Dereich and T.~M{\"u}ller-Gronbach.
\newblock General multilevel adaptations for stochastic approximation
  algorithms of robbins--monro and polyak--ruppert type.
\newblock \emph{Numerische Mathematik}, 142\penalty0 (2):\penalty0 279--328,
  2019.
\newblock \doi{10.1007/s00211-019-01024-y}.

\bibitem[Devolder et~al.(2014)Devolder, Glineur, and
  Nesterov]{devolder2014inexact}
O.~Devolder, F.~Glineur, and Y.~Nesterov.
\newblock First-order methods of smooth convex optimization with inexact
  oracle.
\newblock \emph{Mathematical Programming}, 146\penalty0 (1--2):\penalty0
  37--75, 2014.
\newblock \doi{10.1007/s10107-013-0677-5}.

\bibitem[Ehrhardt et~al.(2025)Ehrhardt, Kereta, Liang, and
  Tang]{ehrhardt2025guide}
M.~J. Ehrhardt, {\v Z}.~Kereta, J.~Liang, and J.~Tang.
\newblock A guide to stochastic optimisation for large-scale inverse problems.
\newblock \emph{Inverse Problems}, 41\penalty0 (5):\penalty0 053001, 2025.
\newblock \doi{10.1088/1361-6420/adc0b7}.

\bibitem[Farboodi and Veldkamp(2023)]{farboodi2023data}
M.~Farboodi and L.~Veldkamp.
\newblock Data and markets.
\newblock \emph{Annual Review of Economics}, 15:\penalty0 23--40, 2023.
\newblock \doi{10.1146/annurev-economics-082322-023244}.

\bibitem[Foster et~al.(2019)Foster, Sekhari, Shamir, Srebro, Sridharan, and
  Woodworth]{foster2019complexity}
D.~J. Foster, A.~Sekhari, O.~Shamir, N.~Srebro, K.~Sridharan, and B.~Woodworth.
\newblock The complexity of making the gradient small in stochastic convex
  optimization.
\newblock In \emph{Conference on Learning Theory}, pages 1319--1345. PMLR,
  2019.

\bibitem[Gilardi et~al.(2023)Gilardi, Alizadeh, and Kubli]{gilardi2023chatgpt}
F.~Gilardi, M.~Alizadeh, and M.~Kubli.
\newblock Chatgpt outperforms crowd workers for text-annotation tasks.
\newblock \emph{Proceedings of the National Academy of Sciences}, 120\penalty0
  (30):\penalty0 e2305016120, 2023.
\newblock \doi{10.1073/pnas.2305016120}.

\bibitem[Giles(2008)]{giles2008multilevel}
M.~B. Giles.
\newblock Multilevel monte carlo path simulation.
\newblock \emph{Operations Research}, 56\penalty0 (3):\penalty0 607--617, 2008.
\newblock \doi{10.1287/opre.1070.0496}.

\bibitem[Goel et~al.(2023)Goel, Gueta, Gilon, Liu, Erell, Nguyen, Hao, Jaber,
  Reddy, Kartha, Steiner, Laish, and Feder]{goel2023llms}
A.~Goel, A.~Gueta, O.~Gilon, C.~Liu, S.~Erell, L.~H. Nguyen, X.~Hao, B.~Jaber,
  S.~Reddy, R.~Kartha, J.~Steiner, I.~Laish, and A.~Feder.
\newblock Llms accelerate annotation for medical information extraction.
\newblock In \emph{Proceedings of the 3rd Machine Learning for Health
  Symposium}, volume 225 of \emph{Proceedings of Machine Learning Research},
  pages 82--100, 2023.

\bibitem[Gower et~al.(2019)Gower, Loizou, Qian, Sailanbayev, Shulgin, and
  Richt{\'a}rik]{gower2019sgd}
R.~M. Gower, N.~Loizou, X.~Qian, A.~Sailanbayev, E.~Shulgin, and
  P.~Richt{\'a}rik.
\newblock {SGD}: General analysis and improved rates.
\newblock In \emph{Proceedings of the 36th International Conference on Machine
  Learning}, pages 5200--5209, 2019.

\bibitem[Gower et~al.(2021)Gower, Sebbouh, and Loizou]{gower2021structured}
R.~M. Gower, O.~Sebbouh, and N.~Loizou.
\newblock Sgd for structured nonconvex functions: Learning rates, minibatching
  and interpolation.
\newblock In \emph{Proceedings of the 24th International Conference on
  Artificial Intelligence and Statistics}, volume 130 of \emph{Proceedings of
  Machine Learning Research}, pages 1315--1323, 2021.

\bibitem[Hayat and Hasan(2025)]{hayat2024claim}
A.~Hayat and M.~R. Hasan.
\newblock A context-aware approach for enhancing data imputation with
  pre-trained language models.
\newblock In \emph{Proceedings of the 31st International Conference on
  Computational Linguistics}, pages 5668--5685, Abu Dhabi, UAE, 2025.
  Association for Computational Linguistics.
\newblock URL \url{https://aclanthology.org/2025.coling-main.380/}.

\bibitem[He et~al.(2025)He, Ban, Zou, Wei, Cook, and He]{he2025llmforest}
X.~He, Y.~Ban, J.~Zou, T.~Wei, C.~Cook, and J.~He.
\newblock {LLM}-forest: Ensemble learning of {LLM}s with graph-augmented
  prompts for data imputation.
\newblock In \emph{Findings of the Association for Computational Linguistics:
  ACL 2025}, pages 6921--6936, Vienna, Austria, 2025. Association for
  Computational Linguistics.
\newblock \doi{10.18653/v1/2025.findings-acl.361}.

\bibitem[Hendrycks et~al.(2018)Hendrycks, Mazeika, Wilson, and
  Gimpel]{hendrycks2018trusted}
D.~Hendrycks, M.~Mazeika, D.~Wilson, and K.~Gimpel.
\newblock Using trusted data to train deep networks on labels corrupted by
  severe noise.
\newblock In \emph{Advances in Neural Information Processing Systems}, 2018.

\bibitem[Ho et~al.(2020)Ho, Jain, and Abbeel]{ho2020ddpm}
J.~Ho, A.~Jain, and P.~Abbeel.
\newblock Denoising diffusion probabilistic models.
\newblock In \emph{Advances in Neural Information Processing Systems},
  volume~33, pages 6840--6851, 2020.

\bibitem[Howard(1966)]{howard1966information}
R.~A. Howard.
\newblock Information value theory.
\newblock \emph{IEEE Transactions on Systems Science and Cybernetics},
  2\penalty0 (1):\penalty0 22--26, 1966.
\newblock \doi{10.1109/TSSC.1966.300074}.

\bibitem[Hu et~al.(2021)Hu, Chen, and He]{hu2021biasvariancecost}
Y.~Hu, X.~Chen, and N.~He.
\newblock On the bias-variance-cost tradeoff of stochastic optimization.
\newblock In \emph{Advances in Neural Information Processing Systems},
  volume~34, pages 22119--22131, 2021.

\bibitem[Hu et~al.(2024)Hu, Wang, Chen, and He]{hu2024mlmc}
Y.~Hu, J.~Wang, X.~Chen, and N.~He.
\newblock Multi-level monte-carlo gradient methods for stochastic optimization
  with biased oracles.
\newblock \emph{arXiv preprint arXiv:2408.11084}, 2024.

\bibitem[Ilandarideva et~al.(2025)Ilandarideva, Juditsky, Lan, and
  Li]{ilandarideva2025state}
S.~Ilandarideva, A.~B. Juditsky, G.~Lan, and T.~Li.
\newblock Accelerated stochastic approximation with state-dependent noise.
\newblock \emph{Mathematical Programming}, 213\penalty0 (1--2):\penalty0
  239--280, 2025.

\bibitem[Jain et~al.(2018)Jain, Kakade, Kidambi, Netrapalli, and
  Sidford]{jain2018accelerating}
P.~Jain, S.~M. Kakade, R.~Kidambi, P.~Netrapalli, and A.~Sidford.
\newblock Accelerating stochastic gradient descent for least squares
  regression.
\newblock In \emph{Proceedings of the 31st Conference on Learning Theory},
  volume~75 of \emph{Proceedings of Machine Learning Research}, pages 545--604.
  PMLR, 2018.

\bibitem[Juditsky et~al.(2023)Juditsky, Kulunchakov, and
  Tsyntseus]{juditsky2023sparse}
A.~Juditsky, A.~Kulunchakov, and H.~Tsyntseus.
\newblock Sparse recovery by reduced variance stochastic approximation.
\newblock \emph{Information and Inference: A Journal of the IMA}, 12\penalty0
  (2):\penalty0 851--896, 2023.

\bibitem[Kandasamy et~al.(2017)Kandasamy, Dasarathy, Schneider, and
  P{\'o}czos]{kandasamy2017multifidelity}
K.~Kandasamy, G.~Dasarathy, J.~Schneider, and B.~P{\'o}czos.
\newblock Multi-fidelity bayesian optimisation with continuous approximations.
\newblock In \emph{Proceedings of the 34th International Conference on Machine
  Learning}, pages 1799--1808, 2017.

\bibitem[Kotsalis et~al.(2022)Kotsalis, Lan, and Li]{kotsalis2022simple}
G.~Kotsalis, G.~Lan, and T.~Li.
\newblock Simple and optimal methods for stochastic variational inequalities,
  ii: Markovian noise and policy evaluation in reinforcement learning.
\newblock \emph{SIAM Journal on Optimization}, 32\penalty0 (2):\penalty0
  1120--1155, 2022.

\bibitem[Lan(2012)]{lan2012optimal}
G.~Lan.
\newblock An optimal method for stochastic composite optimization.
\newblock \emph{Mathematical Programming}, 133\penalty0 (1):\penalty0 365--397,
  2012.

\bibitem[Lan and Zhou(2018)]{lan2018rge}
G.~Lan and Y.~Zhou.
\newblock Random gradient extrapolation for distributed and stochastic
  optimization.
\newblock \emph{SIAM Journal on Optimization}, 28\penalty0 (4):\penalty0
  2753--2782, 2018.

\bibitem[Nahum et~al.(2025)Nahum, Calderon, Keller, Szpektor, and
  Reichart]{nahum2025labelerrors}
O.~Nahum, N.~Calderon, O.~Keller, I.~Szpektor, and R.~Reichart.
\newblock Are {LLM}s better than reported? detecting label errors and
  mitigating their effect on model performance.
\newblock In \emph{Proceedings of the 2025 Conference on Empirical Methods in
  Natural Language Processing}, pages 26782--26809, 2025.

\bibitem[Nemirovski et~al.(2009)Nemirovski, Juditsky, Lan, and
  Shapiro]{nemirovski2009robust}
A.~Nemirovski, A.~Juditsky, G.~Lan, and A.~Shapiro.
\newblock Robust stochastic approximation approach to stochastic programming.
\newblock \emph{SIAM Journal on Optimization}, 19\penalty0 (4):\penalty0
  1574--1609, 2009.

\bibitem[Nemirovskij and Yudin(1983)]{nemirovskij1983problem}
A.~S. Nemirovskij and D.~B. Yudin.
\newblock Problem complexity and method efficiency in optimization.
\newblock 1983.

\bibitem[Ong et~al.(2025)Ong, Almahairi, Wu, Chiang, Wu, Gonzalez, Kadous, and
  Stoica]{ong2025routellm}
I.~Ong, A.~Almahairi, V.~Wu, W.-L. Chiang, T.~Wu, J.~E. Gonzalez, M.~W. Kadous,
  and I.~Stoica.
\newblock Routellm: Learning to route llms from preference data.
\newblock In \emph{International Conference on Learning Representations}, 2025.

\bibitem[Parciak et~al.(2024)Parciak, Vandevoort, Neven, Peeters, and
  Vansummeren]{parciak2024schema}
M.~Parciak, B.~Vandevoort, F.~Neven, L.~M. Peeters, and S.~Vansummeren.
\newblock Schema matching with large language models: an experimental study.
\newblock In \emph{Proceedings of Workshops at the 50th International
  Conference on Very Large Data Bases (VLDB 2024), Tabular Data Analysis
  Workshop (TaDA)}, 2024.

\bibitem[Peeters et~al.(2025)Peeters, Steiner, and Bizer]{peeters2025entity}
R.~Peeters, A.~Steiner, and C.~Bizer.
\newblock Entity matching using large language models.
\newblock In \emph{Proceedings of the 28th International Conference on
  Extending Database Technology (EDBT)}, pages 529--541, 2025.
\newblock \doi{10.48786/edbt.2025.42}.

\bibitem[Peherstorfer et~al.(2018)Peherstorfer, Willcox, and
  Gunzburger]{peherstorfer2018multifidelity}
B.~Peherstorfer, K.~Willcox, and M.~Gunzburger.
\newblock Survey of multifidelity methods in uncertainty propagation,
  inference, and optimization.
\newblock \emph{SIAM Review}, 60\penalty0 (3):\penalty0 550--591, 2018.

\bibitem[Polyak(1990)]{polyak1990new}
B.~T. Polyak.
\newblock New stochastic approximation type procedures.
\newblock \emph{Automat. i Telemekh}, 7\penalty0 (98-107):\penalty0 2, 1990.

\bibitem[Ratner et~al.(2017)Ratner, Bach, Ehrenberg, Fries, Wu, and
  R{\'e}]{ratner2017snorkel}
A.~Ratner, S.~H. Bach, H.~Ehrenberg, J.~Fries, S.~Wu, and C.~R{\'e}.
\newblock Snorkel: Rapid training data creation with weak supervision.
\newblock \emph{Proceedings of the VLDB Endowment}, 11\penalty0 (3):\penalty0
  269--282, 2017.

\bibitem[Rodomanov et~al.(2024)Rodomanov, Jiang, and
  Stich]{rodomanov2024universality}
A.~Rodomanov, X.~Jiang, and S.~Stich.
\newblock Universality of adagrad stepsizes for stochastic optimization:
  Inexact oracle, acceleration and variance reduction.
\newblock \emph{Advances in Neural Information Processing Systems},
  37:\penalty0 26770--26813, 2024.

\bibitem[Sims(2003)]{sims2003rational}
C.~A. Sims.
\newblock Implications of rational inattention.
\newblock \emph{Journal of Monetary Economics}, 50\penalty0 (3):\penalty0
  665--690, 2003.
\newblock \doi{10.1016/S0304-3932(03)00029-1}.

\bibitem[Wang et~al.(2025{\natexlab{a}})Wang, Liu, Zhang, Chen, Wu, Zhou, Lou,
  Wang, Feng, Chen, and Wang]{wang2025llm4dsr}
B.~Wang, F.~Liu, C.~Zhang, J.~Chen, Y.~Wu, S.~Zhou, X.~Lou, J.~Wang, Y.~Feng,
  C.~Chen, and C.~Wang.
\newblock {LLM4DSR}: Leveraging large language model for denoising sequential
  recommendation.
\newblock \emph{ACM Transactions on Information Systems}, 44\penalty0 (1),
  2025{\natexlab{a}}.
\newblock \doi{10.1145/3762182}.

\bibitem[Wang et~al.(2025{\natexlab{b}})Wang, Li, Yang, Guan, Liu, Zhang, Qin,
  and Zhou]{wang2025llm4hrs}
S.~Wang, W.~Li, H.~Yang, J.~Guan, X.~Liu, Y.~Zhang, R.~Qin, and S.~Zhou.
\newblock {LLM4HRS}: {LLM}-based spatiotemporal imputation model for highly
  sparse remote sensing data.
\newblock \emph{IEEE Transactions on Geoscience and Remote Sensing},
  63:\penalty0 4204117, 2025{\natexlab{b}}.
\newblock \doi{10.1109/TGRS.2025.3555634}.

\bibitem[Wang et~al.(2025{\natexlab{c}})Wang, Zheng, Sui, and
  Xiong]{wang2025llard}
S.~Wang, Z.~Zheng, Y.~Sui, and H.~Xiong.
\newblock Unleashing the power of large language model for denoising
  recommendation.
\newblock In \emph{Proceedings of the ACM Web Conference 2025}, pages 252--263,
  2025{\natexlab{c}}.
\newblock \doi{10.1145/3696410.3714758}.

\bibitem[Woodworth and Srebro(2021)]{woodworth2021even}
B.~E. Woodworth and N.~Srebro.
\newblock An even more optimal stochastic optimization algorithm: minibatching
  and interpolation learning.
\newblock \emph{Advances in neural information processing systems},
  34:\penalty0 7333--7345, 2021.

\bibitem[Ye et~al.(2025)Ye, Shah, Zhang, and Chava]{ye2025calibrating}
L.~Ye, A.~Shah, C.~Zhang, and S.~Chava.
\newblock Calibrating pre-trained language classifiers on {LLM}-generated noisy
  labels via iterative refinement.
\newblock In \emph{Proceedings of the 31st ACM SIGKDD Conference on Knowledge
  Discovery and Data Mining V.2}, pages 3598--3609, 2025.
\newblock \doi{10.1145/3711896.3736871}.

\bibitem[Zhang et~al.(2024)Zhang, Dong, Xiao, and Oyamada]{zhang2024jellyfish}
H.~Zhang, Y.~Dong, C.~Xiao, and M.~Oyamada.
\newblock Jellyfish: Instruction-tuning local large language models for data
  preprocessing.
\newblock In \emph{Proceedings of the 2024 Conference on Empirical Methods in
  Natural Language Processing}, pages 8754--8782, 2024.

\end{thebibliography}

\vspace{0.3in}

\appendix
\section{Proofs in Section~\ref{sec:known_bias}}\label{app:proof}
This section contains full proofs of main results in Section~\ref{sec:known_bias}.
\subsection{Proof of Theorem~\ref{thm:sge}}\label{sec:proof_sge_thm}
To prove Theorem~\ref{thm:sge}, we first provide the following proposition, which characterizes the SGE recursion. This recursion lets us handle state-dependent variance and persistent bias transparently. The proof of this proposition, which is deferred to Appendix~\ref{subsec:proof_of_proposition}, follows a similar SGE analysis technique as in \citet{ilandarideva2025state}, but here it also incorporates a persistent bias term and the strongly convex regularization appearing in Algorithm~\ref{alg:sge}.


\begin{proposition}[SGE recursion]\label{prop:general_convergence}
Suppose $\{z_t\}$ and $\{x_t\}$ are generated by Algorithm~\ref{alg:sge}. Let the algorithm parameters $\{\alpha_t\}$, $\{\eta_t\}$ and $\{\tau_t\}$ satisfy $\tau_1 \geq 0,$ $\tau_t > 0$ for all $t\geq 2$, along with relations
\begin{subequations}
\label{eq:def_rell}
\begin{align}
\theta_{t-1} &= \alpha_t \theta_t, ~~\eta_t \le \alpha_t (\eta_{t-1} + \mu\alpha_{t-1}) , ~~ t= 2, \ldots,k \label{eq:def_eta_rel}\\
\tfrac{\eta_t \tau_{t-1}}{\alpha_{t}} &\ge{4}L,
\quad t= 3, \ldots, k \label{eq:def_Lip_rel}\\
 \tfrac{\eta_1\eta_2}{\alpha_2} &\geq  {16}L^2, ~~\eta_k \tau_k  \ge 2L, \label{eq:def_Lip_rel1}
\end{align}
\end{subequations}
 for some $\theta_t \ge 0$.
Then for all $x \in \feaReg$
\begin{align}
&\sum_{t=1}^k \theta_t \left\{  \tau_t [f(x_t) -f(x_{t-1})]- [\langle \nabla f(x_t), x - x_t \rangle + \tfrac{\mu}{2}\|x_t-x\|^2]\right\} + \tfrac{\theta_k}{2}(\tfrac{\eta_k}{2} + \mu\alpha_k ) \|z_k-x\|^2 \nn\\
&\leq \tfrac{\theta_1\eta_1}{2}\|x_0-x\|^2 - \tfrac{\theta_k \mu}{2}\|x_k-x\|^2- \sum_{t=1}^k \theta_t \langle {b(x_{t-1}) + \alpha_t (b(x_{t-1}) - b(x_{t-2}))}, z_t - x\rangle\nn\\
&\quad + \sum_{t=0}^{k-1}\left[  \tfrac{2\theta_{t+1}(1+\alpha_{t+1}^2)}{\eta_{t+1}}\|\delta_t(x_t)\|^2 + \tfrac{2\theta_{t+2}\alpha_{t+2}^2}{\eta_{t+2}}\|\delta_{t+1}(x_t)\|^2 \right]\nn\\
&\quad + {\sum_{t=0}^{k-1}  \theta_{t+1} \langle\alpha_{t+1} \delta_{t}(x_{t-1})- (1+\alpha_{t+1})\delta_{t}(x_{t}), z_{t} - x\rangle}.
\label{eq:prop_GEM_main_general}
\end{align}
where $g(x) := \bbe[G_t(x)]$, $b(x) := g(x)-\nabla f(x)$, and 
$\delta_t(x):=G_t(x)-g(x)$.
 \end{proposition}

Taking Proposition~\ref{prop:general_convergence} as given for the moment, let us complete the proof of the theorem.
\paragraph{1$^o$. An SGE recursion in expectation.}
Set $\theta_t=t$ and define
\begin{equation}\label{eq:def_epsilon_t_adaptive}
q_t
=
\tfrac{2\theta_{t+1}(1+\alpha_{t+1}^2)}{\eta_{t+1}}
+
\tfrac{2\theta_{t+2}\alpha_{t+2}^2}{\eta_{t+2}},
\qquad
\epsilon_t=\tfrac{q_t}{m}.
\end{equation}
A direct substitution shows that the parameter choices in Theorem~\ref{thm:sge} satisfy the conditions of Proposition~\ref{prop:general_convergence} and, moreover,
\begin{equation}\label{eq:batchsizeBt_adaptiveb}
\theta_t\tau_t+\cL\epsilon_{t-1}
\le
\theta_{t-1}(1+\tau_{t-1}),
\qquad t\ge2.
\end{equation}
The latter relation is the key condition that absorbs the state-dependent part of the stochastic variance into the telescoping function-value terms in the recursion.

Apply Proposition~\ref{prop:general_convergence} with $x=x^*$ and take expectations on both sides of Ineq.~\eqref{eq:prop_GEM_main_general}, we have
\begin{align}
&\sum_{t=1}^k \theta_t \bbe\left\{  \tau_t [f(x_t) -f(x_{t-1})]- [\langle \nabla f(x_t), x^* - x_t \rangle + \tfrac{\mu}{2}\bbe\|x_t-x^*\|^2]\right\} + \tfrac{\theta_k}{2}(\tfrac{\eta_k}{2} + \mu\alpha_k ) \|z_k-x^*\|^2 \nn\\
&\leq \tfrac{\theta_1\eta_1}{2}\|x_0-x^*\|^2 - \tfrac{\theta_k \mu}{2}\bbe\|x_k-x^*\|^2+\sum_{t=1}^k \theta_t (1+2\alpha_t)\cB \bbe\|z_t - x^*\|\nn\\
&\quad + \sum_{t=0}^{k-1} \cL \epsilon_t \bbe[f(x_t)-f^* 
] +  \sum_{t=0}^{k-1} \epsilon_t \sigma_*^2, \label{expected_recursion_sge}
\end{align}
where we use Cauchy-Schwarz inequality and $\|b(x_t)\|\leq \cB$ in \eqref{eq:bias} to bound the bias term, $\bbe[\delta_t (x_{j})] = 0$ for $j\leq t$ for the martingale term, and 
\[
\E\|\delta_t(x)\|^2
\le
\tfrac1m\Bigl(
\cL\,[f(x)-f^*
]
+\sigma_*^2
\Bigr),
\]
to bound the variance term.
Then, using strong convexity,
\[
\langle\nabla f(x_t),x^*-x_t\rangle
+\tfrac{\mu}{2}\|x_t-x^*\|^2
\le
f^*-f(x_t),
\]
we have
\begin{align}\label{convergence_sge}
&\theta_k(1+\tau_k)\bbe[f(x_k) -f^*] + \tfrac{\theta_k}{2}\left(\tfrac{\eta_k}{2} + \mu\alpha_k \right) \|z_k-x^*\|^2 + \tfrac{\theta_k \mu}{2}\bbe\|x_k-x^*\|^2 \nn\\
&\leq\tfrac{\theta_1\eta_1}{2}\|x_0-x^*\|^2 + \theta_1 \tau_1 [f(x_0) - f^*] + \cL \epsilon_0 [f(x_0) - f^* 
]\nn\\
& \quad + \sum_{t=1}^{k-1} (\theta_{t+1}\tau_{t+1} + \cL\epsilon_t - \theta_t (1+\tau_t)) [f(x_t) - f^*]  +\sum_{t=1}^k \theta_t (1+2\alpha_t)\cB \bbe\|z_t - x^*\| +  \sum_{t=0}^{k-1} \epsilon_t \sigma_*^2\nn\\
&\le
\tfrac{\theta_1\eta_1}{2}\|x_0-x^*\|^2 
+\cL \epsilon_0 [f(x_0) - f^* ]
+\sum_{t=1}^k\theta_t(1+2\alpha_t)\cB\,\E\|z_t-x^*\|
+\sum_{t=1}^k\epsilon_{t-1}\sigma_*^2.
\end{align}
where the second inequality follows from condition \eqref{eq:batchsizeBt_adaptiveb} 
and $\tau_1=0$.
Substituting the explicit choices of $\theta_t,\alpha_t,\eta_t,$ and $\tau_t$ into \eqref{convergence_sge}, and dropping the nonnegative terminal term involving $\|x_k-x^*\|^2$, we have
\begin{equation}\label{eq:app-master}
A_k\E[f(x_k)-f^*]
+\tfrac{P_k}{2}\E\|z_k-x^*\|^2
\le
\Gamma_k
+\cB\sum_{t=1}^k(3t-2)\E\|z_t-x^*\|,
\end{equation}
where
\[
A_k=\tfrac{k(k+2)}{3},
\qquad
P_k=\tfrac{\eta}{2}+\tfrac{\mu k(k-1)}{4},
\]
and
\[
\Gamma_k
=
\tfrac12\left(
24L+\tfrac{54\cL}{m}+\tfrac{162\cL^2}{\mu m^2}
\right)\|x_0-x^*\|^2 + \tfrac{2}{27}[f(x_0) -f^*]
+\tfrac{32k}{\mu m}\sigma_*^2.
\]

\paragraph{2$^o$. Self-bounding the accumulated bias.}
Next we bound the summation of the linear bias terms in Ineq.~\eqref{eq:app-master}. Define
\[
R_k
:=
\Gamma_k
+\cB\sum_{t=1}^k(3t-2)\E\|z_t-x^*\|.
\]
Since
$
\tfrac{P_t}{2}\E\|z_t-x^*\|^2\le R_t\le R_k
$
for every $t\le k$,
\[
\E\|z_t-x^*\|
\le
\sqrt{\E\|z_t-x^*\|^2}
\le
\sqrt{\tfrac{2R_k}{P_t}}.
\]
Moreover, $\eta\ge24L\ge24\mu$ implies
$
P_t\ge \tfrac{\mu(3t-2)^2}{36}$, thus
$
\tfrac{3t-2}{\sqrt{P_t}}
\le
\frac{6}{\sqrt{\mu}}.
$
Therefore, by Ineq.~\eqref{eq:app-master}, we have
\[
R_k \leq \Gamma_k + \tsum_{t=1}^k (3t-2)\cB \sqrt{\tfrac{2 R_k}{P_t}}
\le
\Gamma_k
+\frac{6k\cB \sqrt{2R_k}}{\sqrt\mu}
\le
\Gamma_k+\frac{R_k}{2}
+\frac{36k^2\cB^2}{\mu},
\]
where the last inequality follows from Young's inequality.
Rearranging the terms yields
\begin{align}\label{bound_R_k}
R_k
\le
2\Gamma_k+\tfrac{72k^2\cB^2}{\mu}.
\end{align}

Dividing \eqref{eq:app-master} by $A_k$ and using the bound of $R_k$ in Ineq.~\eqref{bound_R_k}, we obtain
\begin{align*}
\E[f(x_k)-f^*]
\le&
\left(
\tfrac{72L}{k(k+2)}
+\tfrac{162\cL}{mk(k+2)}
+\tfrac{486\cL^2}{\mu m^2k(k+2)}
\right)\|x_0-x^*\|^2 + \tfrac{2}{9k(k+2)}[f(x_0)-f^*]+\tfrac{192\sigma_*^2}{\mu m(k+2)}
+\tfrac{216\cB^2}{\mu}\nn\\
\le&
\left(
\tfrac{73L}{k(k+2)}
+\tfrac{162\cL}{mk(k+2)}
+\tfrac{486\cL^2}{\mu m^2k(k+2)}
\right)\|x_0-x^*\|^2 +\tfrac{192\sigma_*^2}{\mu m(k+2)}
+\tfrac{216\cB^2}{\mu},
\end{align*}
where the second inequality follows from $f(x_0) - f^* = f(x_0) - f^* - \langle \nabla f(x^*), x_0 - x^*\rangle \leq \tfrac{L}{2}\|x_0 - x^*\|^2$, and we prove Ineq.~\eqref{eq:sge-rate}. Finally, for $m=1$, the stated lower bound on $k$ makes the optimization and residual-variance terms at most $\mu\|x_0-x^*\|^2/8$. Strong convexity then indicates
\[
\E\|x_k-x^*\|^2
\le
\tfrac{2}{\mu}\E[f(x_k)-f^*]
\le
\tfrac14\|x_0-x^*\|^2
+\tfrac{432\cB^2}{\mu^2},
\]
which proves Ineq.~\eqref{sge_contraction}.

\subsubsection{Proof of Proposition~\ref{prop:general_convergence}}\label{subsec:proof_of_proposition}
\textbf{1$^o$.} 
First, 
by the smoothness of $f$, we have
\begin{equation}
\tfrac{1}{2L}\|\nabla f(x_t) - \nabla f(x_{t-1})\|^2 \le f(x_{t-1}) - [f(x_t) + \langle \nabla f(x_t), x_{t-1} - x_t\rangle].
\end{equation}
Then for any $x\in \feaReg$, using the above inequality and the fact that
$x -  x_t  = x -z_t -  \tau_t( x_{t-1} -  x_t)$ due to Step 5 in Algorithm~\ref{alg:sge},
we get
\begin{align*}
\tau_t f( x_t) - \langle \nabla f(x_t), x -  x_t \rangle
&= \tau_t [f( x_t) + \langle \nabla f(x_t),  x_{t-1} -  x_t \rangle] + \langle \nabla f(x_t), z_t - x \rangle  \\
&\le \tau_t [f( x_{t-1}) - \tfrac{1}{2L} \|\nabla f(x_t) - \nabla f(x_{t-1})\|^2] + \langle \nabla f(x_t), z_t - x \rangle .
\end{align*}
By the optimality condition of Step 3 in Algorithm~\ref{alg:sge}, we have
\begin{align*}
 &\langle \widetilde G_t, z_t - x \rangle + \tfrac{\mu\alpha_t}{2}\|x_{t-1}-z_t\|^2 \le  \tfrac{\eta_t}{2}\|z_{t-1}-x\|^2 - \tfrac{\eta_t+\mu\alpha_t}{2}\|z_t-x\|^2 - \tfrac{\eta_t}{2}\|z_{t-1}-z_t\|^2 + \tfrac{\mu\alpha_t}{2}\|x_{t-1}-x\|^2.
\end{align*}
Combining the two previous inequalities and rearranging the terms, we obtain
\begin{align*}
& \tau_t f( x_t) - [\langle {\nabla f(x_t)}, x -  x_t \rangle + \tfrac{\mu}{2}\|x_t-x\|^2]   - \tau_t f( x_{t-1})\\
&\le \tfrac{\eta_t}{2}\|z_{t-1}-x\|^2 - \tfrac{\eta_t+\mu\alpha_t}{2}\|z_t-x\|^2 + \langle {\nabla f(x_t)} - \widetilde G_t, z_t - x \rangle - \tfrac{\mu}{2}\|x_t-x\|^2\\
&\quad + \tfrac{\mu\alpha_t}{2}\|x_{t-1}-x\|^2-\tfrac{\tau_t}{2L} \|{\nabla f(x_t)-\nabla f(x_{t-1})}\|^2 - \tfrac{\eta_t}{2}\|z_{t-1}-z_t\|^2.
\end{align*}
Invoking the definition of $\delta_t(\cdot)$ and $b(\cdot)$, we have
\begin{align*}
\langle {\nabla f(x_t)} - \widetilde G_t, z_t - x \rangle
&= \langle {\nabla f(x_t) - \nabla f(x_{t-1}) - \alpha_t (\nabla f(x_{t-1}) - \nabla f(x_{t-2}))}, z_t - x \rangle\\
&\quad - \langle {\delta_{t-1}(x_{t-1}) + \alpha_t (\delta_{t-1}(x_{t-1}) - \delta_{t-1}(x_{t-2}))}, z_t - x\rangle\\
&\quad - \langle {b(x_{t-1}) + \alpha_t (b(x_{t-1}) - b(x_{t-2}))}, z_t - x\rangle\\
&= \langle {\nabla f(x_t) - \nabla f(x_{t-1})}, z_t - x\rangle - \alpha_t \langle  {\nabla f(x_{t-1}) - \nabla f(x_{t-2})}, z_{t-1} - x \rangle\\
&\quad + \alpha_t \langle {\nabla f(x_{t-1}) - \nabla f(x_{t-2})}, z_t - z_{t-1} \rangle\\
&\quad - \langle {\delta_{t-1}(x_{t-1}) + \alpha_t (\delta_{t-1}(x_{t-1}) - \delta_{t-1}(x_{t-2}))}, z_t - x\rangle\\
&\quad - \langle {b(x_{t-1}) + \alpha_t (b(x_{t-1}) - b(x_{t-2}))}, z_t - x\rangle.
\end{align*}
When taking the $\theta_t$-weighted sum of the above inequalities for $t= 1, \ldots, k$,
recalling that $x_0 = z_0$, and using \eqref{eq:def_eta_rel}, we obtain
\begin{align}
&\sum_{t=1}^k \theta_t \left\{  \tau_t [f( x_t)- f( x_{t-1})] - [\langle {\nabla f(x_t)}, x -  x_t \rangle + \tfrac{\mu}{2}\|x_t-x\|^2]  \right\}  \nn\\
&\le \tfrac{\theta_1\eta_1}{2}\|x_0-x\|^2 - \theta_k (\eta_k + \mu \alpha_k) \tfrac12\|z_k-x\|^2 - \sum_{t=1}^k \theta_t \langle {b(x_{t-1}) + \alpha_t (b(x_{t-1}) - b(x_{t-2}))}, z_t - x\rangle\nn\\
&\quad +\theta_k \langle {\nabla f(x_k) - \nabla f(x_{k-1})}, z_k - x \rangle - \theta_k \tfrac{\mu}{2}\|x_k-x\|^2 + \theta_1 \mu \alpha_1 \tfrac12\|x_0-x\|^2   + \Delta_k \label{eq:GEM_converge_rel1}
\end{align}
where
\begin{align*}
\Delta_k &:= \sum_{t=1}^k \theta_t \big[\alpha_t \langle {\nabla f(x_{t-1}) - \nabla f(x_{t-2})}, z_t - z_{t-1} \rangle -\tfrac{\tau_t}{2L} \|{\nabla f(x_t) - \nabla f(x_{t-1})}\|^2\\
& ~\qquad\quad\qquad - \langle {\delta_{t-1}(x_{t-1}) + \alpha_t (\delta_{t-1}(x_{t-1}) - \delta_{t-1}(x_{t-2}))}, z_t - x\rangle  - \tfrac{\eta_t}{2}\|z_{t-1}-z_t\|^2\big].
\end{align*}
Observe that
\[\theta_t\alpha_t\langle {\nabla f(x_{t-1})-\nabla f(x_{t-2})},z_t-z_{t-1}\rangle\leq \tfrac{\theta_t^2\alpha_t^2L}{ 2\theta_{t-1}\tau_{t-1}}\|z_t-z_{t-1}\|^2+\tfrac{\theta_{t-1}\tau_{t-1}}{ 2L}\|{\nabla f(x_{t-1})-\nabla f(x_{t-2})}\|^2,
\]
so that,
after rearranging the terms,
\begin{align*}
\Delta_k &\le \underbrace{\theta_2\alpha_2 \langle {\nabla f(x_1)-\nabla f(x_0)}, z_2-z_1\rangle +  \sum_{t=3}^k \tfrac{L \theta_{t}^2 \alpha_{t}^2}{2 \theta_{t-1} \tau_{t-1}} \|z_t - z_{t-1}\|^2} \\
&\quad \underbrace{- \tfrac{\theta_k \tau_k}{2 L}\|{\nabla f(x_k) - \nabla f(x_{k-1})}\|^2 - {\tfrac{1}{4}}\sum_{t=1}^k  \tfrac{\theta_t\eta_t}{2}\|z_{t-1}-z_t\|^2}_{=:\Delta_{k,1}} \\
& \quad \underbrace{-\sum_{t=1}^k \theta_t \langle {\delta_{t-1}(x_{t-1}) + \alpha_t (\delta_{t-1}(x_{t-1}) - \delta_{t-1}(x_{t-2}))}, z_t - x\rangle - {\tfrac{3}{4}}\sum_{t=1}^k \tfrac{\theta_t\eta_t}{2}\|z_{t-1}-z_t\|^2}_{=:\Delta_{k,2}}.
\end{align*}
\paragraph{2$^o$.}
We next bound $\Delta_{k,1}$. By Young's inequality,
\begin{align*}
\alpha_2\langle {\nabla f(x_1)-\nabla f(x_0)},z_2-z_1\rangle\leq \tfrac{2\alpha_2^2 }{\eta_2}\|{\nabla f(x_1)-\nabla f(x_0)}\|^2+\tfrac{\eta_2}{8}\|z_2-z_1\|^2.
\end{align*}
Thus,
\begin{align}\label{eq:bound_delta_1}
    \Delta_{k,1} &\leq \tfrac{{2}\theta_2\alpha_2^2 \|{\nabla f(x_1)-\nabla f(x_0)}\|^2}{\eta_2} - \tfrac{\theta_1\eta_1}{{8}}\|z_1-z_0\|^2\nn\\
    &\quad +\sum_{t=3}^k \left(\tfrac{L \theta_{t}^2 \alpha_{t}^2}{2 \theta_{t-1} \tau_{t-1}} - \tfrac{\theta_t \eta_t}{{8}}\right)\|z_t - z_{t-1}\|^2- \tfrac{\theta_k \tau_k}{2 L}\|{\nabla f(x_k) - \nabla f(x_{k-1})}\|^2\nn\\
    &\leq - \tfrac{\theta_k \tau_k}{2 L}\|{\nabla f(x_k) - \nabla f(x_{k-1})}\|^2
\end{align}
where the last inequality follows from \eqref{eq:def_Lip_rel} and the bound
\begin{align*}
	\tfrac{{2}\theta_2\alpha_2^2 \|{\nabla f(x_1)-\nabla f(x_0)}\|^2}{\eta_2}  \leq \tfrac{2\theta_2\alpha_2^2L^2\| x_1 -  x_0\|^2}{\eta_2} = \tfrac{2\theta_2\alpha_2^2L^2\|z_1 -  z_0\|^2}{\eta_2}  \leq \tfrac{\theta_1 \eta_1 \|z_1-z_0\|^2}{{8}}.
\end{align*}
Note that
\begin{align*}
&  -  \tfrac{\eta_k}{4} \|z_k-x\|^2
+ \langle {\nabla f(x_k) - \nabla f(x_{k-1})}, z_k - x \rangle
- \tfrac{ \tau_k}{2 L}\|{\nabla f(x_k) - \nabla f(x_{k-1})}\|^2 \nn \\
&\qquad\le - (\tfrac{\eta_k }{4} - \tfrac{L}{2\tau_k}) \|z_k-x\|^2 \le 0
\end{align*}
due to the second relationship in \eqref{eq:def_Lip_rel1}.
Now, substituting the bound \eqref{eq:bound_delta_1} into \eqref{eq:GEM_converge_rel1} and applying $\alpha_1= 0$ results in
\begin{align}
&\sum_{t=1}^k \theta_t \left\{  \tau_t [f( x_t)-f( x_{t-1})] - [\langle {\nabla f(x_t)}, x -  x_t \rangle + \tfrac{\mu}{2}\|x_t-x\|^2]\right\} \nn\\
&\qquad\le \tfrac{\theta_1\eta_1}{2}\|x_0-x\|^2 - \tfrac{\theta_k}{2} (\eta_k + \mu \alpha_k)\|z_k-x\|^2- \sum_{t=1}^k \theta_t \langle {b(x_{t-1}) + \alpha_t (b(x_{t-1}) - b(x_{t-2}))}, z_t - x\rangle \nn\\
&\qquad \quad +\theta_k \langle {\nabla f(x_k) - \nabla f(x_{k-1})}, z_k - x \rangle  -  \tfrac{\theta_k\mu}{2}\|x_k-x\|^2    - \tfrac{\theta_k \tau_k}{2 L}\|{\nabla f(x_k) - \nabla f(x_{k-1})}\|^2 + \Delta_{k,2}\nn\\
&\qquad \le \tfrac{\theta_1\eta_1}{2}\|x_0-x\|^2 - \tfrac{\theta_k}{2} (\tfrac{\eta_k}{2} + \mu \alpha_k) \|z_k-x\|^2 - \tfrac{\theta_k \mu}{2}\|x_k-x\|^2\nn\\
&\qquad \quad- \sum_{t=1}^k \theta_t \langle {b(x_{t-1}) + \alpha_t (b(x_{t-1}) - b(x_{t-2}))}, z_t - x\rangle + \Delta_{k,2}.\label{eq:GEM_converge_rel1_1}
\end{align}

\paragraph{\textbf{3$^o$.}}
It remains to bound $\Delta_{k,2}$. Observe that
\begin{align*}
    \Delta_{k,2} &\le - \sum_{t=1}^k\theta_t [ \langle {\delta_{t-1}(x_{t-1})}, z_t - z_{t-1}\rangle +  \langle {\delta_{t-1}(x_{t-1})}, z_{t-1} - x\rangle]\\
&\quad - \sum_{t=2}^k \theta_t \alpha_t [\langle {\delta_{t-1}(x_{t-1})}, z_t - z_{t-1}\rangle +  \langle {\delta_{t-1}(x_{t-1})}, z_{t-1} - x\rangle]\\
&\quad +  \sum_{t=2}^k\theta_t \alpha_t [ \langle {\delta_{t-1}(x_{t-2})}, z_t - z_{t-1}\rangle + {\langle \delta_{t-1}(x_{t-2}), z_{t-1} - x\rangle}] - {  \tfrac{3\theta_t }{4}\sum_{t=1}^k \tfrac{\eta_t}{2}\|z_{t-1}-z_t\|^2}.
\end{align*}
Recall that $z_t$ is ${\cal F}_{t-1}$-measurable. By using Young's inequality, we have
\begin{align}
  \Delta_{k,2} &\le  \sum_{t=1}^k [ {\tfrac{2\theta_t}{\eta_t}\|\delta_{t-1}(x_{t-1})\|^2} -  \theta_t\langle {\delta_{t-1}(x_{t-1})}, z_{t-1} - x\rangle]\nn\\
&\quad + \sum_{t=2}^k  [{\tfrac{2\theta_t\alpha_t^2}{\eta_t}\|\delta_{t-1}(x_{t-1})\|^2} -  \theta_t \alpha_t \langle {\delta_{t-1}(x_{t-1})}, z_{t-1} - x\rangle] \nn\\
&\quad + \sum_{t=2}^k[ {\tfrac{2\theta_t\alpha_t^2}{\eta_t}\|\delta_{t-1}(x_{t-2})\|^2}
+ \theta_t \alpha_t{\langle \delta_{t-1}(x_{t-2}), z_{t-1} - x\rangle}] \nn \\
&\le
  {\sum_{t=1}^k \big[\tfrac{2\theta_t(1+\alpha_t^2)}{\eta_t}\|\delta_{t-1}(x_{t-1})\|^2 + \tfrac{2\theta_{t+1}\alpha_{t+1}^2}{\eta_{t+1}}\|\delta_t(x_{t-1})\|^2 \big]}\nn\\
&\quad  + {\sum_{t=1}^{k}  \theta_t \langle\alpha_t \delta_{t-1}(x_{t-2})- (1+\alpha_t)\delta_{t-1}(x_{t-1}), z_{t-1} - x\rangle}.
\label{eq:bound_delta_2}
\end{align}
Substituting \eqref{eq:bound_delta_2} into \eqref{eq:GEM_converge_rel1_1} yields \eqref{eq:prop_GEM_main_general}. This completes the technical part of the argument; taking expectations and specializing the parameters produces the main recursion \eqref{expected_recursion_sge} used in the proof of Theorem~\ref{thm:sge}.\qed

\subsection{Proof of Corollary~\ref{corollary_known}}\label{proof_corollary_known}
\quad~~\textbf{1$^o$.}
For the SGE implementation, substituting $R_{s-1}^2=R_0^2/2^{s-1}$ into \eqref{eq:stage-complexity} yields \eqref{total_cost_results}. For convenience, define the stage-$s$ cost of oracle $\lambda$ by
\begin{equation}
    C_\lambda(s)
    :=\max\{A_\lambda,B_\lambda 2^{s-1}\},
    \qquad s=1,\ldots,K.
\end{equation}
Then we have
\begin{equation}
    \Cost_{\rm adap}(\epsilon)
    =\sum_{s=1}^K \min_\lambda g_\lambda(s),
    \quad
    \Cost_{\rm fixed}(\epsilon)
    =\min_\lambda \sum_{s=1}^K g_\lambda(s).
\end{equation}
which implies that
$
    \Cost_{\rm adap}(\epsilon)
    \le \Cost_{\rm fixed}(\epsilon).
$

\textbf{2$^o$.} We next show that the improvement can be of order $K/\log K$.
Assume $K\ge 4$, and let
\[
    m:=\lceil \log_2 K\rceil.
\]
Fix any $M>0$ and consider two oracles with
\[
    A_1=\tfrac{M}{K^2},
    \quad
    B_1=M2^{m-K+1},
    \quad
    A_2=M,
    \quad
    B_2=M2^{-K}.
\]
Then $A_1/A_2=K^{-2}$, whereas
\[
    \tfrac{B_1}{B_2}=2^{m+1}=\Theta(K),
\]
thus SO$1$ has a much smaller $A$-term but a much larger
residual-variance term.

For SO$2$, since
\[
    B_2 2^{s-1}
    \le B_2 2^{K-1}
    =\tfrac{M}{2}<M=A_2,
\]
we have
\[
    C_2(s)=M,
    \quad s=1,\ldots,K.
\]
Therefore, using SO$2$ throughout costs
\[
    \sum_{s=1}^K C_2(s)=MK.
\]
Now consider SO$1$. At stage $K-m$, we have 
\begin{align}\label{middle_cost_so1}
    B_1 2^{K-m-1}
    =
    M2^{m-K+1}2^{K-m-1}
    =M.
\end{align}
Therefore
\[
    C_1(s)\le M=C_2(s),
    \quad \text{for } s\le K-m,
\]
and
\[
    C_1(s)>M=C_2(s),
    \quad \text{for } s>K-m.
\]
Consequently, the adaptive rule uses SO$1$ during the early stages and
switches to SO$2$ for the final $m$ stages. In particular, the
switch occurs $m=\Theta(\log K)$ stages before termination. The adaptive cost therefore satisfies
\begin{align*}
    \Cost_{\rm adap}(\epsilon)
    &=
    \sum_{s=1}^{K-m} C_1(s)+mM\\
    &\le
    (K-m)A_1
    +B_1\sum_{s=1}^{K-m}2^{s-1}
    +mM\\
    &\le
    KA_1+B_1 2^{K-m}+mM=
    \tfrac{M}{K}+2M+mM,
\end{align*}
and consequently,
\[
    \Cost_{\rm adap}(\epsilon)
    =\Theta(mM)
    =\Theta(M\log K).
\]

It remains to check the best fixed oracle. As shown above,
SO$2$ costs exactly $MK$. For SO$1$, recalling Eq.~\eqref{middle_cost_so1}, its final-stage
cost satisfies
\[
    C_1(K)
    \ge B_1 2^{K-1}
    =M2^m
    \ge MK,
\]
Therefore, $
    \sum_{s=1}^K C_1(s)\ge MK,
$
and the best fixed oracle is SO2, with cost
\[
    \Cost_{\rm fixed}(\epsilon)=MK.
\]
We conclude that
\[
    \tfrac{\Cost_{\rm fixed}(\epsilon)}
         {\Cost_{\rm adap}(\epsilon)}
    =
    \Theta\left(\tfrac{K}{\log K}\right).
\]
Since $
    K=\left\lceil
    \log_2\frac{R_0^2}{\epsilon}
    \right\rceil,
$
this is equivalently
$
    \Theta\left(
    \tfrac{\log(R_0^2/\epsilon)}
         {\log\log(R_0^2/\epsilon)}
    \right).
$

\textbf{3$^o$.} Finally, we show that this order cannot be improved, up to constants,
for arbitrary stage costs of the form
$C_\lambda(s)=\max\{A_\lambda,B_\lambda2^{s-1}\}$.
Let
\[
    d_s:=\min_\lambda C_\lambda(s)
\]
denote the adaptive cost at stage $s$. Since every $C_\lambda(s)$ is
nondecreasing in $s$, $d_s$ is also nondecreasing.
Again let $m=\lceil\log_2 K\rceil$.
Choose an oracle $\bar\lambda$ attaining the minimum cost at stage $K-m$, i.e.,
$
    g_{\bar\lambda}(K-m)=d_{K-m}.
$
It follows that
\[
    A_{\bar\lambda}\le d_{K-m},
    \quad \text{and ~}
    B_{\bar\lambda}2^{K-m-1}\le d_{K-m}.
\]
Hence the cost of using this oracle throughout is bounded by
\begin{align*}
    \Cost_{\rm fixed}(\epsilon)
    &\le
    \sum_{s=1}^K
    \max\{A_{\bar\lambda},
          B_{\bar\lambda}2^{s-1}\}\le
    \sum_{s=1}^K
    \left(A_{\bar\lambda}
          +B_{\bar\lambda}2^{s-1}\right)\le
    Kd_{K-m}+B_{\bar\lambda}2^K\le
    \left(K+2^{m+1}\right)d_{K-m}.
\end{align*}
Since $m=\lceil\log_2K\rceil$, we have $2^m<2K$, and therefore
\[
    \Cost_{\rm fixed}(\epsilon)
    \le 5K d_{K-m}.
\]
On the other hand, monotonicity of $d_s$ indicates
\[
    \Cost_{\rm adap}(\epsilon)
    =\sum_{s=1}^K d_s
    \ge
    \sum_{s=K-m}^K d_s
    \ge
    (m+1)d_{K-m}.
\]
Combining the last two inequalities yields
\[
    \tfrac{\Cost_{\rm fixed}(\epsilon)}
         {\Cost_{\rm adap}(\epsilon)}
    \le
    \tfrac{5K}{m+1}
    =
    O\left(\tfrac{K}{\log K}\right).
\]
Together with the two-oracle construction above, this proves that the
$\Theta(K/\log K)$ savings factor is tight up to constants. \qed

\section{Proofs in Section~\ref{sec:unknown_bias}}\label{app:proof_unknown}
This section collects proofs in Section~\ref{sec:unknown_bias}.
\subsection{Proof of Lemma~\ref{lem:pg-certificate}}\label{proof:lem:pg-certificate}
Let
\[
    x^+
    =
    \Pi_X(x-\gamma\nabla f(x))
    =
    x-\gamma\PG_{X,\gamma}(x).
\]
The optimality condition of the projection and the definition of $\PG_{X,\gamma}(x)$ indicate that
\begin{align*}
\langle \nabla f(x) - \PG_{X, \gamma}(x), y - x^+\rangle \geq 0, \quad \forall y \in X.
\end{align*}
By rearranging terms and using the definition of $\PG_{X,\gamma}(x)$ again, we obtain
\begin{align*}
    \langle \nabla f(x), y - x^+\rangle - \langle \PG_{X, \gamma}(x), y-x\rangle - \gamma\|\PG_{X, \gamma}(x)\|^2 \geq 0.
\end{align*}
Invoking that $f$ is $\mu$-strongly convex and $L$-smooth, we have
\begin{align*}
\langle \nabla f(x), y - x^+\rangle &= \langle \nabla f(x), y - x\rangle + \langle \nabla f(x), x - x^+\rangle\\
& \leq f(y) - f(x) -\tfrac{\mu}{2}\|y-x\|^2 + f(x) - f(x^+) + \tfrac{L}{2}\|x^+ - x\|^2\\
& = f(y) - f(x^+) -\tfrac{\mu}{2}\|y-x\|^2 + \tfrac{\gamma^2 L}{2}\|\PG_{X,\gamma}(x)\|^2. 
\end{align*}
Combining the two inequalities above yields
\[
f(y)
\ge
f(x^+)
+
\inner{\PG_{X,\gamma}(x)}{y-x}
+
(\gamma - \tfrac{\gamma^2L}{2})
\|\PG_{X,\gamma}(x)\|^2
+
\tfrac{\mu}{2}\|y-x\|^2
\]
for every $y\in X$. Taking $y = x$ yields 
\begin{align*}
f(x) - f(x^+) \geq (\gamma - \tfrac{\gamma^2L}{2})
\|\PG_{X,\gamma}(x)\|^2.
\end{align*}
Taking $y=x^*$ and using $f(x^*)\le f(x^+)$ and $\gamma \leq 2/L$ provide us with
\[
\tfrac{\mu}{2}\|x-x^*\|^2
\le
\inner{\PG_{X,\gamma}(x)}{x-x^*}
\le
\|\PG_{X,\gamma}(x)\|\,\|x-x^*\|.
\]
Dividing by $\|x-x^*\|$ proves the first claim, and the second claim follows
immediately. \qed

\subsection{Proof of Lemma~\ref{lem:rr-noise}}\label{proof:lem:rr-noise}
Because the quadratic regularization is deterministic, we have
\[
\begin{aligned}
g_j(x)-\nabla\phi_j(x)
&=
g(x)
+\sum_{i=1}^j\rho_i(x-\widehat x_i)
-\nabla f(x)
-\sum_{i=1}^j\rho_i(x-\widehat x_i)=
g(x)-\nabla f(x),
\end{aligned}
\]
which proves the bias claim. Similarly,
\[
    \cG_j(x,\xi)-g_j(x)
    =
    \cG(x,\xi)-g(x),
\]
so the stochastic error itself is unchanged.

It remains only to rewrite the state-dependent term in terms of
the regularized objective. For every $i\ge1$,
\begin{align*}
\phi_{i-1}(x)-\phi_{i-1}^*
&=
\phi_i(x)-\phi_i^*
+\phi_i^*-\phi_{i-1}^*
-\tfrac{\rho_i}{2}\|x-\widehat x_i\|^2\\
&\le
\phi_i(x)-\phi_i^*
+\tfrac{\rho_i}{2}
\|x_{i-1}^*-\widehat x_i\|^2,
\end{align*}
where the last inequality follows from
\[
    \phi_i^*
    \le
    \phi_i(x_{i-1}^*)
    =
    \phi_{i-1}^*
    +
    \tfrac{\rho_i}{2}
    \|x_{i-1}^*-\widehat x_i\|^2.
\]
By iterating the preceding inequality, we obtain
\[
f(x)-f^* = \phi_0(x) - \phi_0^*
\le
\phi_j(x)-\phi_j^*
+
\sum_{i=1}^j
\tfrac{\rho_i}{2}
\|x_{i-1}^*-\widehat x_i\|^2.
\]
Substituting this into the original variance condition \eqref{assump:state-dependent} proves
\[
\E\|\cG_j(x,\xi)-g_j(x)\|_*^2
\le
\cL(\phi_j(x)-\phi_j^*)
+\widehat\sigma_{j,*}^2.
\]

Finally, $\phi_{i-1}$ is $\rho_i$-strongly convex because
\[
    \mu+\sum_{r=1}^{i-1}\rho_r
    =
    2^{i-1}\mu
    =
    \rho_i.
\]
Therefore we have
\[
    \tfrac{\rho_i}{2}
    \|\widehat x_i-x_{i-1}^*\|^2
    \le
    \phi_{i-1}(\widehat x_i)-\phi_{i-1}^*.
\]
Summing up the above inequality from $i=1$ to $j$ and taking expectations yields
\[
    \E[\widehat\sigma_{j,*}^2]
    \le
    \sigma_*^2
    +
    \cL\sum_{i=1}^j\delta_i,
\]
which completes the proof. \qed

\subsection{Proof of Proposition~\ref{prop:rr-pg}}\label{proof:prop:rr-pg}
We first control the displacement of the minimizers induced by the
regularizers. Since $\phi_{j-1}$ is $\rho_j$-strongly convex,
\begin{equation}
    \E\|\widehat x_j-x_{j-1}^*\|^2 \leq \tfrac{2}{\rho_j}\bbe[\phi_{j-1}(\widehat x_j) - \phi_{j-1}^*]
    \le
    \tfrac{2\delta_j}{\rho_j}.
\label{eq:rr-prev-dist}
\end{equation}
Moreover, we have
$
    \phi_j(\widehat x_j)
    =
    \phi_{j-1}(\widehat x_j),
$
while
$
    \phi_j^*
    \ge \phi_{j-1}^*.
$
Thus
\[
    \E[
       \phi_j(\widehat x_j)-\phi_j^*
    ] \leq \E[
       \phi_{j-1}(\widehat x_j)-\phi_{j-1}^*
    ]
    \le\delta_j.
\]
Since $\phi_j$ is $2\rho_j$-strongly convex, we have
\begin{equation}
    \E\|\widehat x_j-x_j^*\|^2 \leq \tfrac{1}{\rho_j}\bbe[\phi_{j}(\widehat x_j) - \phi_{j}^*]
    \le
    \tfrac{\delta_j}{\rho_j}.
\label{eq:rr-new-dist}
\end{equation}
For $t=1,\ldots,S-1$, define
$
    P_t
    :=
    \sum_{j=1}^t
    \rho_j\|x_t^*-\widehat x_j\|.
$
By using the triangle inequality, we obtain
\begin{align*}
P_t-P_{t-1}
&\le
\rho_t\|x_t^*-\widehat x_t\|
+
\left(\sum_{j=1}^{t-1}\rho_j\right)
\|x_t^*-x_{t-1}^*\|\\
&\le
2\rho_t\|x_t^*-\widehat x_t\|
+
\rho_t
  \|\widehat x_t-x_{t-1}^*\|.
\end{align*}
Substituting Ineqs.~\eqref{eq:rr-prev-dist} and \eqref{eq:rr-new-dist} into the above inequality and using Jensen's inequality provides us with
\begin{align*}
    \E[P_t-P_{t-1}] &\le
2\rho_t\sqrt{\bbe\|x_t^*-\widehat x_t\|^2}
+
\rho_t
  \sqrt{\bbe\|\widehat x_t-x_{t-1}^*\|^2}\le
    (2+\sqrt2)\sqrt{\rho_t\delta_t}.
\end{align*}
Summing up the above inequality from $t=1$ to $S-1$ yields
\begin{equation}
    \E[P_{S-1}]
    \le
    (2+\sqrt{2})\sum_{j=1}^{S-1}\sqrt{\rho_j\delta_j}.
\label{eq:rr-weighted-drift}
\end{equation}
For convenience, let us denote
$
    q(x)
    :=
    \sum_{j=1}^{S-1}
    \tfrac{\rho_j}{2}\|x-\widehat x_j\|^2$, sot that
    $\phi_{S-1}=f+q$.
By the nonexpansiveness of the Euclidean projection, we have
\[
\|\PG^f_{X,\gamma}(x)
   -\PG_{X,\gamma}^{\phi_{S-1}}(x)\|
\le
\|\nabla q(x)\|.
\]
At $x=\widehat x_S$, by the triangle inequality,
\begin{align*}
\|\nabla q(\widehat x_S)\|
&\le
\sum_{j=1}^{S-1}
\rho_j\|x_{S-1}^*-\widehat x_j\|
+
\left(\sum_{j=1}^{S-1}\rho_j\right)
\|\widehat x_S-x_{S-1}^*\| \leq P_{S-1} + \rho_S \|\widehat x_S-x_{S-1}^*\|,
\end{align*}
where the last inequality follows from the definition of $P_t$ and $
    \sum_{j=1}^{S-1}\rho_j
    =
    \rho_S-\mu
    \leq\rho_S.
$
Recall that the strong convexity modulus of $\phi_{S-1}$ is
$
    \rho_S=2^{S-1}\mu.
$
By Ineq.~\eqref{projected_gradient_ineq} in Lemma~\ref{lem:pg-certificate}, we have
\[
   \rho_S\|\widehat x_S-x_{S-1}^*\|
    \le
    2\|\PG_{X,\gamma}^{\phi_{S-1}}(\widehat x_S)\|.
\]
Consequently, combining the three inequalities above, we have
\begin{equation}
\|\PG^f_{X,\gamma}(\widehat x_S)\| \leq \|\PG_{X,\gamma}^{\phi_{S-1}}(\widehat x_S)\| + \|\nabla q(\widehat x_S)\|
\le
P_{S-1}
+
3\|\PG_{X,\gamma}^{\phi_{S-1}}(\widehat x_S)\|.
\label{eq:rr-pg-transfer}
\end{equation}
Finally, since
$\phi_{S-1}$ is smooth with modulus $
L+\sum_{j=1}^{S-1}\rho_j \leq L + 2^{S-1}\mu\leq2L,
$
our choice $\gamma=1/(2L)$ is a valid projected-gradient stepsize
for $\phi_{S-1}$. The projected-gradient descent inequality in \eqref{decent_ineq} yields
\[
\tfrac{1}{4L}
\|\PG_{X,\gamma}^{\phi_{S-1}}(\widehat x_S)\|^2
\le
\phi_{S-1}(\widehat x_S)-\phi_{S-1}^*,
\]
which implies
\[
\E
\|\PG_{X,\gamma}^{\phi_{S-1}}(\widehat x_S)\|
\le
2\sqrt{L\delta_S} \le
    2\sqrt{\rho_S\delta_S}.
\]
By combining the above inequality with Ineqs.~\eqref{eq:rr-weighted-drift} and \eqref{eq:rr-pg-transfer}, we obtain
\begin{align*}
\E\|\PG_{X,\gamma}(\widehat x_S)\|
\le
(2+\sqrt{2})\sum_{j=1}^{S-1}\sqrt{\rho_j\delta_j}
+
6\sqrt{\rho_S\delta_S},
\end{align*}
which completes the proof. \qed

\section{Prompt templates for the LLM-completion experiments}
\label{app:llm-prompts}
This appendix reports the system prompts used to construct the LLM-based data-completion oracles in Section~\ref{sec:numerics}. We use three prompt templates: a simple-imputation prompt, a short low-rank prompt, and a more detailed low-rank prompt. In the pipeline names, ``LR-short'' refers to the short low-rank prompt and ``LR'' refers to the detailed low-rank prompt. For the low-rank variants, the suffixes ``-7'' and ``-10'' indicate that 7 and 10 features are supplied as input to the prompt, respectively.

\paragraph{Simple imputation prompt (\texttt{SYSTEM\_SIMPLE}).}
\begin{quote}
{\small\ttfamily\raggedright
You are imputing missing numeric entries in a table.\par
Only fill missing values in these columns:\par
10percentage, 12percentage, collegeGPA, English, Logical, Quant, Domain.\par\medskip
Method: SIMPLE IMPUTATION\par
- For continuous/score-like columns, imitate median imputation logic: choose a typical central value consistent with the row context.\par
- For binary/discrete indicators (if any appear), imitate mode imputation logic.\par
- Do NOT change non-missing values.\par
- Return ONLY JSON matching the schema. No reasoning, no extra text.\par}
\end{quote}

\paragraph{Short low-rank prompt (\texttt{SYSTEM\_LOWRANK\_SHORT}).}
\begin{quote}
{\small\ttfamily\raggedright
You are imputing missing numeric entries in a table.\par
Only fill missing values in these columns:\par
10percentage, 12percentage, collegeGPA, English, Logical, Quant, Domain.\par\medskip
Method: LOW-RANK MATRIX COMPLETION\par
- Do NOT change non-missing values.\par
- Domain should be in [0,1].\par
- Return ONLY JSON matching the schema. No reasoning, no extra text.\par}
\end{quote}

\paragraph{Detailed low-rank prompt (\texttt{SYSTEM\_LOWRANK}).}
\begin{quote}
{\small\ttfamily\raggedright
You are imputing missing numeric entries in a table.\par
Only fill missing values in these columns:\par
10percentage, 12percentage, collegeGPA, English, Logical, Quant, Domain.\par\medskip
Method: LOW-RANK MATRIX COMPLETION\par
- Assume the data matrix is approximately low-rank.\par
- Emulate iterative SVD / nuclear-norm minimization logic: fill missing entries so rows/columns align with a few latent factors.\par
- Domain should be in [0,1].\par
- It is fine if lengthy calculations needed.\par
- Prefer smooth, globally consistent imputations.\par
- Do NOT change non-missing values.\par
- Return ONLY JSON matching the schema. No reasoning, no extra text.\par}
\end{quote}

\end{document}